\documentclass[12pt]{amsart}
\usepackage{amsfonts,amssymb}
\usepackage[mathscr]{eucal}
\usepackage[]{hyperref}
\usepackage{setspace,url}
\usepackage{tikz}

\usepackage[hmargin=2cm,bmargin=2cm,tmargin=2cm]{geometry}

\newtheorem{theorem}{Theorem}
\newtheorem{lemma}[theorem]{Lemma}
\newtheorem{corollary}[theorem]{Corollary}
\newtheorem{proposition}[theorem]{Proposition}

\theoremstyle{definition}
\newtheorem{definition}[theorem]{Definition}

\newtheorem{remark}[theorem]{Remark}
\newtheorem{example}[theorem]{Example}

\numberwithin{equation}{section}
\numberwithin{theorem}{section}

\DeclareMathOperator{\diag}{diag}

\DeclareMathOperator{\Dom}{Dom}
\DeclareMathOperator{\dn}{dn}
\DeclareMathOperator{\cn}{cn}

\newcommand{\bbR}{\ensuremath{\mathbb{R}}} 
\newcommand{\bbC}{\ensuremath{\mathbb{C}}} 

\newcommand{\sech}{\ensuremath{\text{sech}}}

\DeclareMathOperator{\Real}{Re}

\newcommand{\bigO}[1]{\mathcal{O}\left(#1\right)}

\newcommand{\Graph}{\Gamma} 
\newcommand{\DTN}{\mathrm{M}} 
\newcommand{\Neu}{\mathcal{N}} 

\newcommand{\Lmin}{\ell_{\mathrm{min}}}

\newcommand{\vecp}{\mathbf{p}} 
\newcommand{\vecq}{\mathbf{q}} 
\newcommand{\tpsi}{\psi} 

\newcommand{\vecc}{\mathbf{c}} 

\title[Edge-localized states in the limit of large mass]{Edge-localized states on quantum graphs \\ in the limit of large mass}

\author[G.~Berkolaiko]{Gregory Berkolaiko}
\address{Department of Mathematics, Texas A \& M University, College Station, TX 77843-3368, USA}
\email{berko@math.tamu.edu}

\author[J.L.~Marzuola]{Jeremy L. Marzuola}
\address{Department of Mathematics, University of North Carolina - Chapel Hill, Chapel Hill, NC 27599, USA}
\email{marzuola@math.unc.edu}

\author[D.E.~Pelinovsky]{Dmitry E. Pelinovsky}
\address{Department of Mathematics, McMaster University, Hamilton, Ontario, L8S 4K1, Canada}
\email{dmpeli@math.mcmaster.ca}

\begin{document}

\begin{abstract}
  In this work, we construct and quantify asymptotically in the limit
  of large mass a variety of edge-localized stationary states of the
  focusing nonlinear Schr\"odinger equation on a quantum
  graph.  The method is applicable to general bounded and unbounded
  graphs.  The solutions are constructed by matching a localized large
  amplitude elliptic function on a single edge with an exponentially
  smaller remainder on the rest of the graph.  This is done by
  studying the intersections of Dirichlet-to-Neumann manifolds
  (nonlinear analogues of Dirichlet-to-Neumann maps) corresponding to
  the two parts of the graph. For the quantum graph with a given set
  of pendant, looping, and internal edges, we find the edge on which
  the state of smallest energy at fixed mass is localized.  Numerical
  studies of several examples are used to illustrate the analytical
  results.
\end{abstract}

\maketitle

\section{Introduction}

Here we study stationary states of the focusing cubic nonlinear
Schr\"{o}dinger (NLS) equation on a quantum graph $\Graph$. The cubic
NLS equation can be written in the normalized form:
\begin{equation}
  i U_t + \Delta U + 2 |U|^2 U = 0,
  \label{nls-time}
\end{equation}
where $U(x,t) : \Graph \times \mathbb{R} \mapsto \mathbb{C}$ is the
wave function and $\Delta$ is the Laplacian operator on the quantum
graph $\Graph$. We assume that the graph $\Graph$ has finitely many vertex points
and finitely many edges (which are either line segments or half-lines).
Neumann--Kirchhoff (sometimes called
``standard'' or ``natural'') boundary conditions are used at the
vertices of the graph: at each vertex the wave function is
continuous and the sum of its outgoing derivatives is zero.  For
general terminology concerning differential operators on graphs the
reader is invited to consult \cite{BerKuc_graphs,Exner}.

The NLS equation is used to describe two distinct
physical phenomena that are studied on networks of nano-wires:
propagation of optical (electromagnetic) pulses and Bose-Einstein
condensation.  A thorough discussion of the physics literature from
mathematical point of view can be found in \cite{Noj_ptrsl07}.
The most important class of solutions for applications are {\em the stationary
states} which are characterized by solutions of the following elliptic problem:
\begin{equation}
  \label{statNLS}
  -\Delta \Phi - 2 |\Phi|^{2} \Phi =  \Lambda \Phi,
\end{equation}
where $\Lambda \in \mathbb{R}$ is the spectral parameter and the
Laplacian $\Delta$ is extended to a self-adjoint operator in
$L^2(\Graph)$ with the domain
$$
H^2_{\Graph} = \left\{ U \in H^2(\Graph) : \  {\rm Neumann-Kirchhoff \; conditions \; at \; vertices} \right\}.
$$
Since $-\Delta$ is positive, it makes sense to restrict the range of
$\Lambda$ in (\ref{statNLS}) to negative values, hence $\Lambda <
0$. The stationary NLS equation (\ref{statNLS}) is the Euler--Lagrange
equation of the action functional
$H_{\Lambda}(U) := \mathcal{E}(U) - \Lambda \mathcal{Q}(U)$, where
$\mathcal{Q}(U)$ and $\mathcal{E}(U)$ are the conserved mass and
energy of the cubic NLS equation:
\begin{equation}
  \label{energy}
    \mathcal{Q}(U) = \int_{\Graph} |U|^2 dx,
  \qquad
  \mathcal{E}(U) = \int_{\Graph} \left( |\partial_x U|^2 - |U|^4 \right) dx.
\end{equation}
The conserved quantities $\mathcal{E}(U)$ and $\mathcal{Q}(U)$ are defined in the weaker space
$$
H^1_{\Graph} = \left\{ U \in H^1(\Graph) \colon \quad U \mbox{ is continuous at vertices}\right\}.
$$
We use consistently notations $H^k(\Graph)$ with $k = 1,2$ to denote Sobolev spaces of
component-wise $H^k$ functions and $H^k_\Graph$ to include vertex boundary conditions
for component-wise $H^k$ functions.

Among stationary states, we single out the standing wave of smallest
energy at fixed mass which, if it exists, coincides with a
solution of the following constrained minimization problem:
\begin{equation}
\label{minimizer}
E_q = \inf_{U \in H^1_{\Graph}} \left\{ \mathcal{E}(U) :  \ \ \mathcal{Q}(U) = q \right\},
\end{equation}
if such a minimizer exists.
In the variational setting, $\Lambda$ is the Lagrange multiplier of the constrained minimization problem (\ref{minimizer}).
Thanks to the Gagliardo--Nirenberg inequality on the graph $\Graph$ (see Proposition 2.1 in \cite{AdaSerTil_jfa16}),
\begin{equation}
  \| U \|_{L^4(\Graph)}^4
  \leq C_{\Graph} \| U \|_{L^2(\Graph)}^3 \| U \|_{H^1(\Graph)},
  \quad U \in H^1_{\Graph},
\label{GN-inequality}
\end{equation}
the infimum in (\ref{minimizer}) is bounded from below, hence $E_q > -\infty$.

If the infimum in (\ref{minimizer}) is attained, the global minimizer
is called the \emph{ground state} of the cubic NLS equation
(\ref{nls-time}) and it coincides with the stationary state of
the Euler--Lagrange equation (\ref{statNLS}) with the smallest energy $E_q$ at fixed mass $q$.
The infimum is always attained in the case of bounded graphs. However, the infimum may not
be attained in the case of unbounded graphs due to the lack of compactness:
$E_q$ could be approached by a minimizing sequence ``escaping'' to
infinity along one of the unbounded edge of the quantum graph \cite{AdaSerTil_cvpde15,AdaSerTil_jfa16}.
See \cite{Ada_mmnp16} for  a review of various techniques used to analyze
the existence and non-existence of the ground state.

In this work, we study existence and properties of the stationary states
that localize exponentially on a single edge of a graph in the limit
of large mass $q$.  We call such states
{\em the edge-localized states}. The relevant
asymptotic approach was pioneered in \cite{MarPel_amrx16} for a
particular bounded graph, the dumbbell graph.  Here we generalize
and formalize this approach for any bounded and unbounded graph.  We
summarize the properties of these states below.

\begin{theorem}
  \label{thm:main_modest1}
  Let $\Graph$ be a graph with finitely many edges and
  Neumann--Kirchhoff conditions at vertices.  Then for any edge $e$ of
  finite length $\ell$ and for large enough $\mu := \sqrt{-\Lambda}$
  there exists a solution $\Psi$ with the following properties
  \begin{enumerate}
  \item $\Psi$ is positive,
  \item $\Psi$ has a single local maximum on $\Graph$; this maximum is
    located on $e$; $\Psi$ monotone between its maximum and
    the end-vertices of $e$,
  \item $\Psi$ concentrates on $e$ in the following sense,
    \begin{equation}
      \label{eq:L2_proportion}
      \frac{\left\|\Psi \right\|_{L^2(e)}}
      {\left\|\Psi\right\|_{L^2(\Graph)}}
      \geq 1-Ce^{-2\mu\ell},
    \end{equation}
    where the constant $C$ is independent of $\mu$.
  \end{enumerate}
\end{theorem}

Full description of solutions $\Psi$, including the location of the
maximum, uniqueness properties and asymptotics of $\mathcal{E}(\Psi)$
and $\mathcal{Q}(\Psi)$ can be found in
Theorems~\ref{thm:localized_pendant}, \ref{thm:localized_loop}, and
\ref{thm:generic_edge} and their proofs.  The asymptotics allow us to
compare solutions localized on different edges and to choose (in most
cases) the edge-localized solution of the smallest energy
$\mathcal{E}(\Psi)$ with a given mass $\mathcal{Q}(\Psi)$.  Since
edge-localized solutions are good candidates for the role of the
ground state (see Proposition~\ref{prop-ground-state} for a
description of known properties of a ground state) this comparison is
an important step towards the ultimate goal of fully describing the
ground state for any graph.

\begin{figure}
  \centering
  \includegraphics{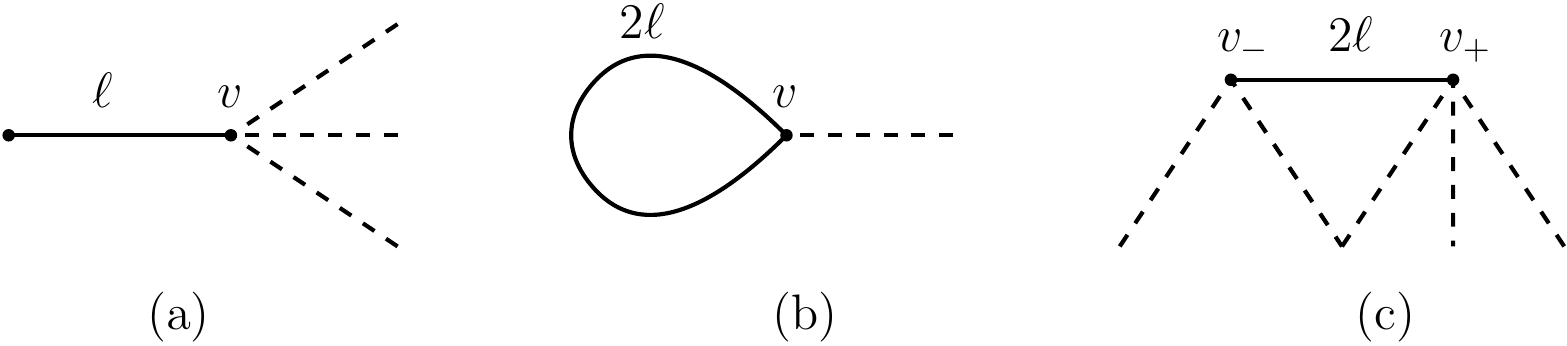}
  \caption{A single edge of a finite length can be connected to the remainder of
    the graph (shown in dashed lines) in three different ways.}
  \label{fig:edge_cases}
\end{figure}

To summarize the answers to the question of comparison, we introduce
the necessary terminology.  We distinguish three types of edges,
illustrated in Fig.~\ref{fig:edge_cases}: a \emph{pendant edge} (or
simply a ``pendant'') is an edge with one vertex of degree one, a
\emph{looping edge} (or simply a ``loop'') is an edge whose
end-vertices coincide, and an \emph{internal edge} is an edge not
belonging to the above classes --- one with distinct end-vertices,
each of degree greater than one.  When we say an edge is incident to
$N$ other edges (at an end-vertex), the number $N$ counts the edges in
the remainder of the graph connected to the end-vertex. In
Fig.~\ref{fig:edge_cases} the pendant edge is incident to $N=3$ edges
at vertex $v$, the looping edge is incident to $N=1$ edge at vertex
$v$, and the internal edge is incident to $N_-=2$ and $N_+=3$ edges at
its end-vertices $v_-$ and $v_+$ correspondingly. The following
theorem gives comparison between the energy levels at a fixed (large)
mass among the edge-localized states of Theorem
\ref{thm:main_modest1}.

\begin{theorem}
  \label{thm:main_modest2}
  Let $\Gamma$ be a compact graph (a graph with finitely many edges,
  all of finite length) and Neumann--Kirchhoff conditions at vertices.
  Among the edge-localized states of Theorem~\ref{thm:main_modest1}
  with a given sufficiently large $\mathcal{Q}(\Psi)=q$, the state
  with the smallest energy localizes on the following edge of the
  graph $\Graph$:
  \begin{itemize}
  \item[(i)] The longest among pendants; in the case of a tie, the
    pendant incident to fewest edges.

  \item[(ii)] If (i) is void, the shortest among loops incident to
    a single edge.

  \item[(iii)] If (i)--(ii) are void, a loop incident to two edges.

  \item[(iv)] If (i)--(iii) are void, the longest edge among the following:
    loops incident to $N \geq 3$ edges, or internal
    edges incident to $N_- \geq 2$ and $N_+ \geq 2$ other edges; in the
    case of two edges of the same length, the edge for which the
    quantity
    \begin{equation*}
      \begin{cases}
        \frac{N-2}{N+2} & \mbox{for a loop}\\[5pt]
        \sqrt{\frac{(N_--1)(N_+-1)}{(N_-+1)(N_++1)}} & \mbox{for an
          internal edge}
      \end{cases}
    \end{equation*}
    is the smallest.
  \end{itemize}
\end{theorem}

In the case of unbounded graph, we can enlarge the class of graphs for
which we guarantee the \emph{existence} of a ground state.  This
result builds upon \cite[Cor.~3.4 and Prop.~4.1]{AdaSerTil_jfa16}.

\begin{corollary}
  \label{cor:main_modest3}
  Consider an unbounded graph $\Graph$ with Neumann--Kirchhoff
  conditions and with finitely many edges (and thus at least one edge
  as a half-line).  The ground state of the constrained minimization
  problem (\ref{minimizer}) exists for sufficiently large $q$ if
  $\Graph$ has at least one pendant or a loop incident to a single
  edge.  If the graph $\Graph$ has no pendants and no loops incident
  to \emph{one or two edges}, the ground state does not exist among
  the edge-localized states of Theorem \ref{thm:main_modest1}.
\end{corollary}

\begin{remark}
  If the
  graph $\Graph$ has a loop connected to two edges,
  the existence of the ground state is inconclusive and needs separate
  consideration.  For the same reasons, there is no ``tie-breaker'' in
  case (iii) of Theorem~\ref{thm:main_modest2}.  This issue has been
  pointed out before, in \cite[Theorem 2.5]{AdaSerTil_cvpde15}.
\end{remark}

\begin{figure}
\begin{tabular}{cc}
\includegraphics[trim=50 50 30 120,clip=true,width=6cm]{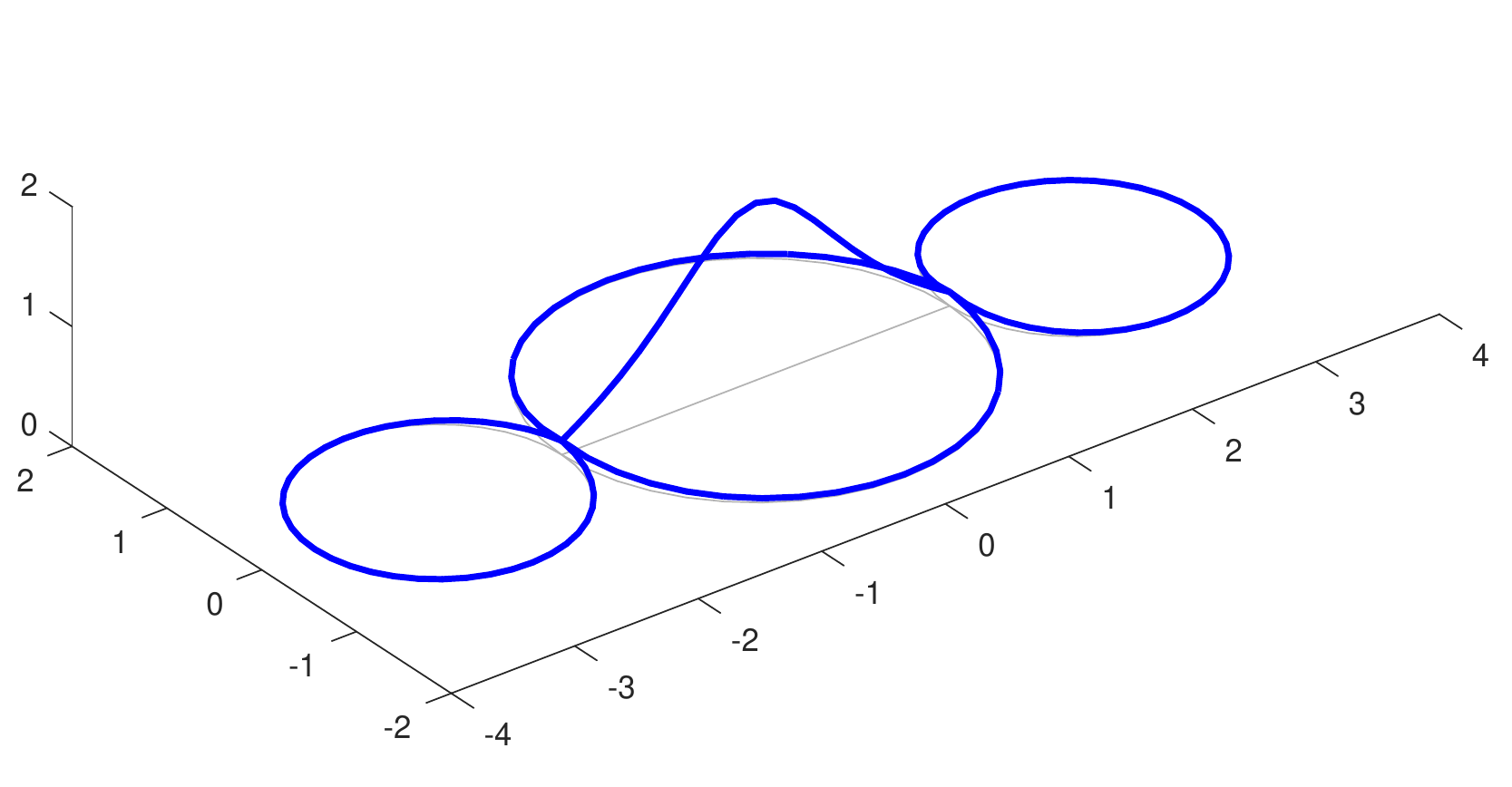}
  & \includegraphics[trim=50 50 30 120,clip=true,width=6cm]{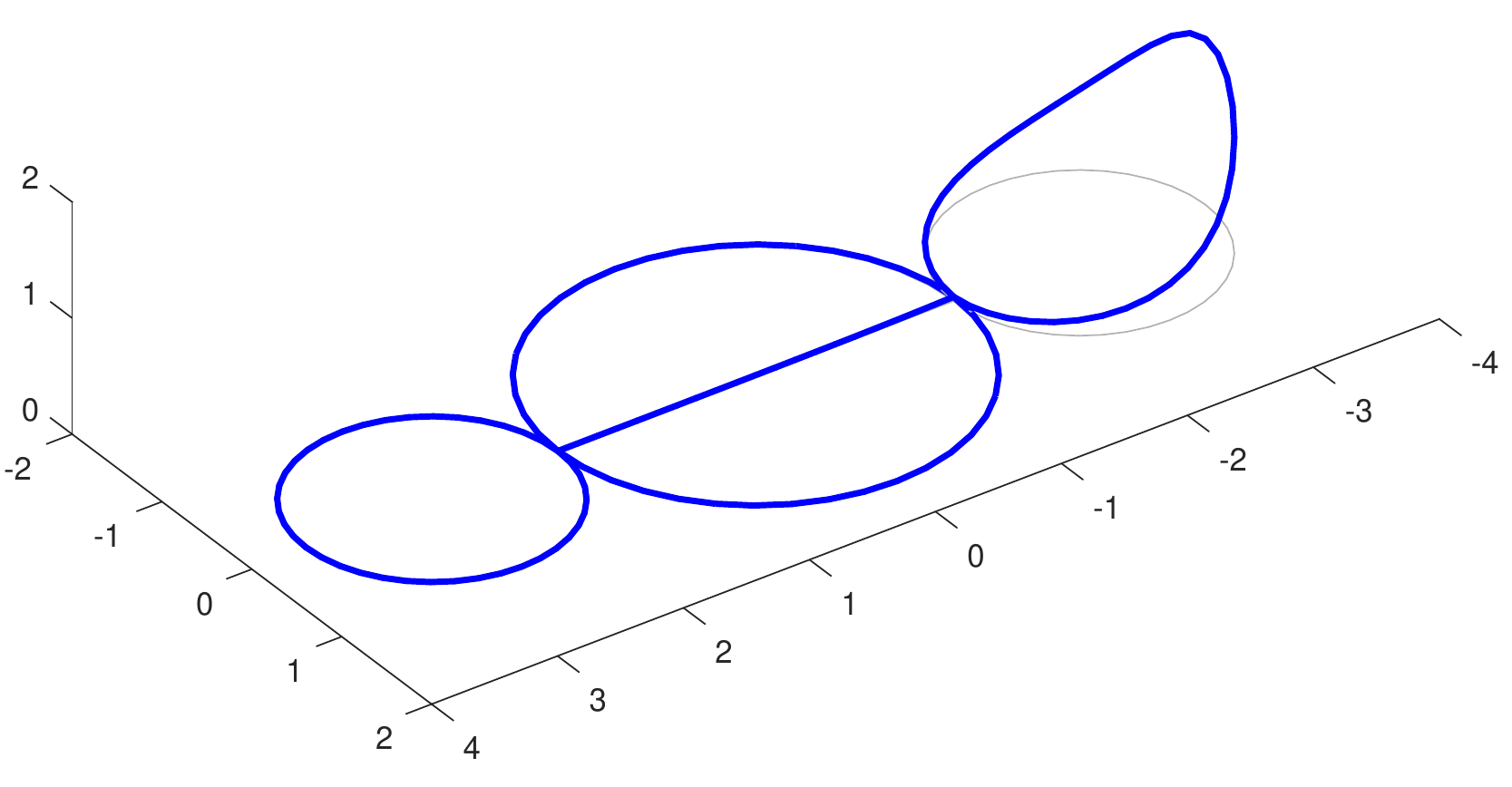} \\
 \includegraphics[width=6cm]{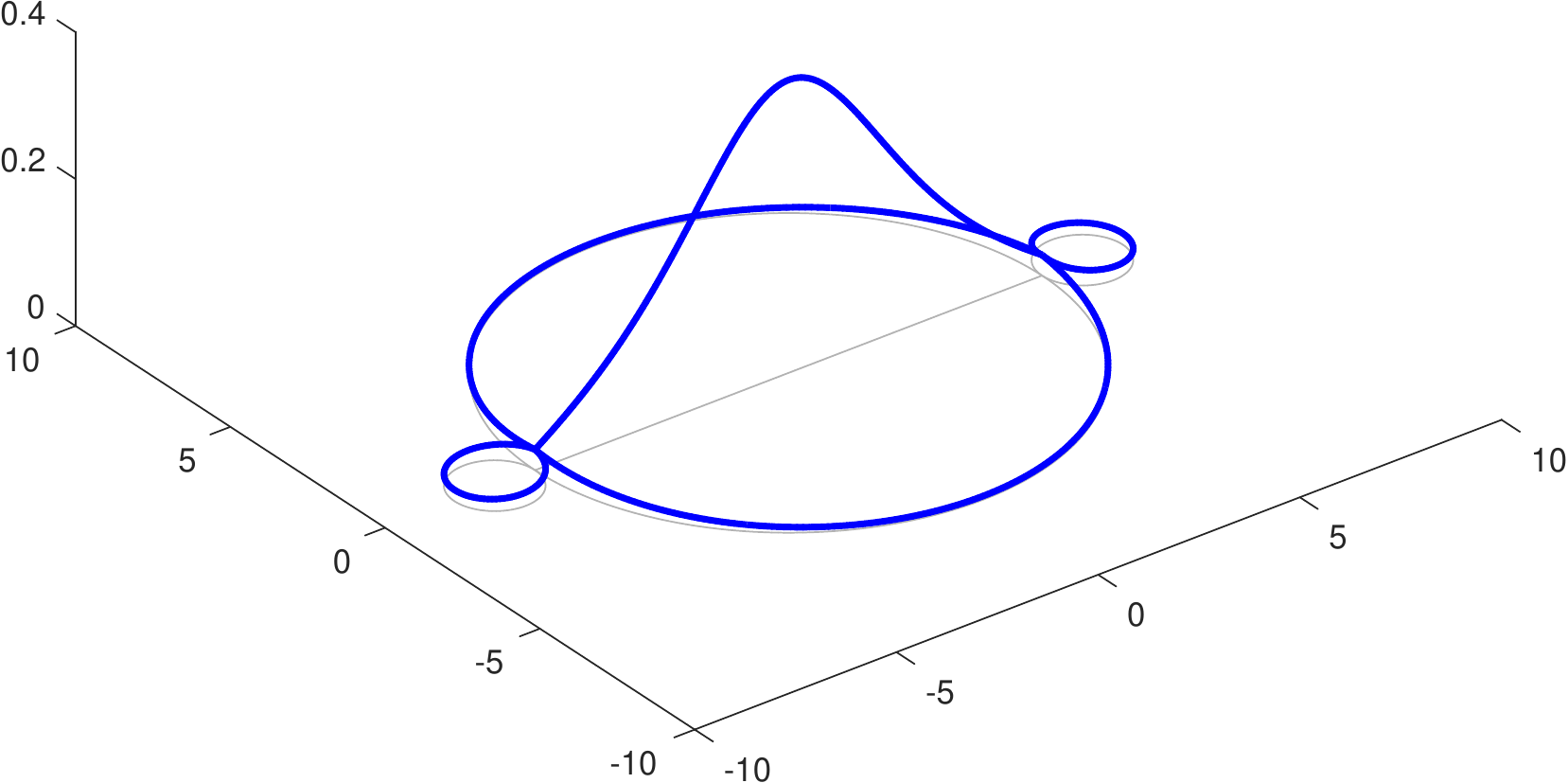}
  & \includegraphics[width=6cm]{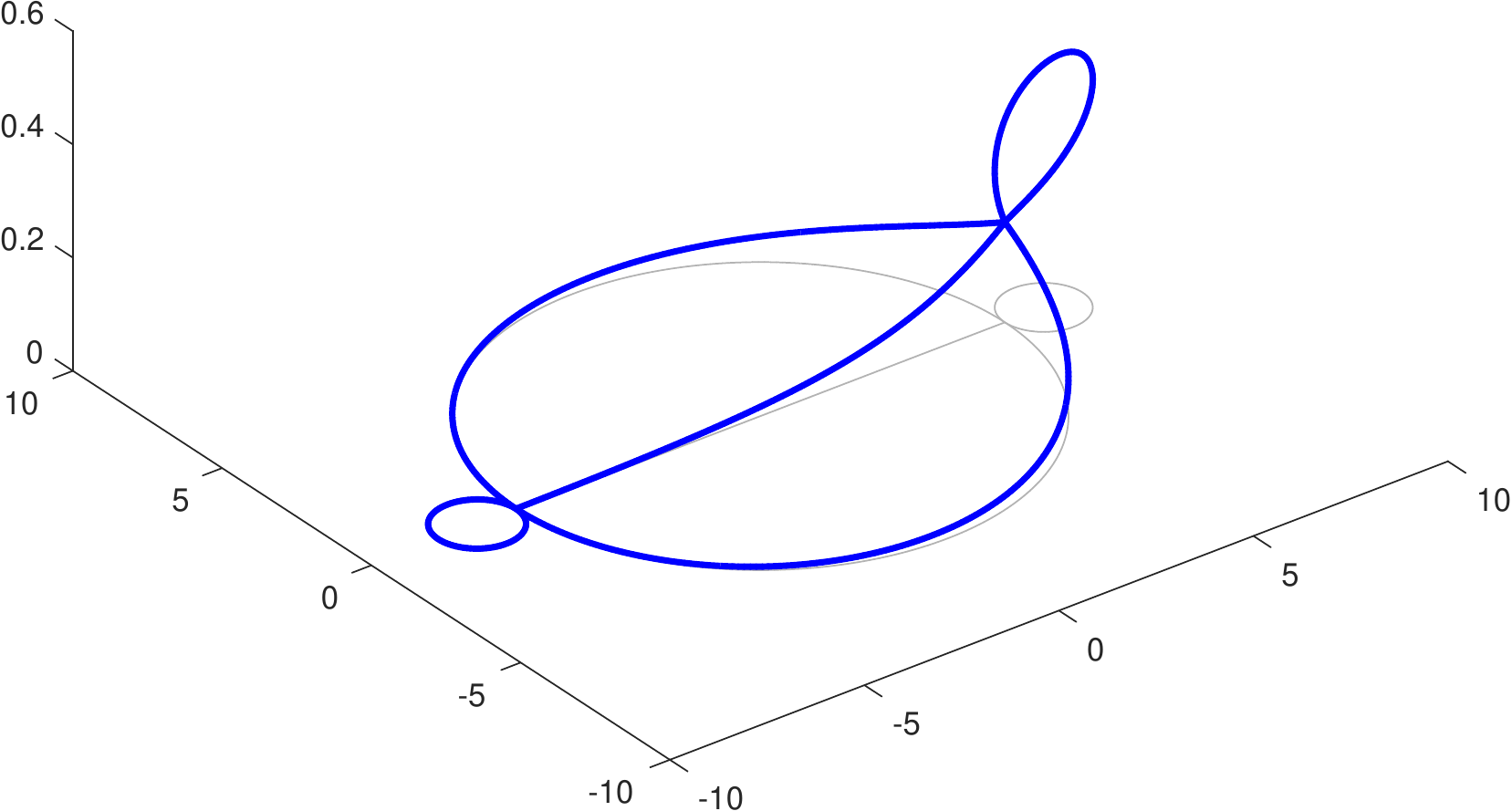}
\end{tabular}
\caption{Stationary states from Theorem \ref{thm:main_modest2} in a graph with
 two loops and three internal edges.  We demonstrate an edge-centered states (left) and
  a loop-centered states (right) for short internal edges relative to the loops  (top) and for long
  internal edges relative to the loops (bottom).  Note that lengths of all internal edges are the same in each of these
  plots, we have just drawn the upper and lower edges as semi-circles
  for visualization.}
\label{db_Kgeq2_sv_fig}
\end{figure}

\begin{remark}
  To illustrate Theorem \ref{thm:main_modest2}$(iv)$, in Figure
  \ref{db_Kgeq2_sv_fig} we show relevant states from an example of the
  graph $\Gamma$ with two loops and three internal edges.  For the
  states plotted on the top, the internal edges are short relative to
  the loops, and the loop state (top right) has smaller energy at
  large mass. For the states plotted on the bottom, the internal edges
  are long relative to the loops, and hence the edge state (bottom
  left) has smaller energy at larger mass.
\end{remark}

The main results of this work are comparable and complimentary to the recent work
\cite{AST2019} on existence of stationary states for the subcritical NLS equation
(which includes, as a particular case, the cubic NLS equation). In \cite[Theorem 3.3]{AST2019},
the existence of local energy minimizers in the limit of large fixed
mass was proven
in the restricted space of functions that attain their maximum on a given edge.
Because we are using elliptic functions, our results are only limited to the cubic NLS
equation compared to the subcritical NLS equation in \cite{AST2019}. On the other hand,
our work extends to both bounded and unbounded graphs. Moreover,
we are computing the exponentially small corrections to the mass of each edge-localized
state in terms of large negative Lagrange multiplier $\Lambda$. With the help of
the main comparison result (Lemma \ref{lem:comparison}), this tool allows us to
compare different edge-localized states and identify the state of minimal energy
in the limit of large fixed mass. One important result which follows from
\cite{AST2019} is that every edge-localized state constructed in our work
is a local minimizer of energy (at least in the case of unbounded graphs considered in \cite{AST2019}),
hence it is orbitally stable in the time evolution of the cubic NLS equation (\ref{nls-time}).

Existence and stability of stationary states in the NLS equation defined on a metric graph
have been recently investigated in great detail  \cite{Noj_ptrsl07}. Existence and variational characterization
of standing waves was developed for star graphs \cite{AdamiAH,AdamiJDE1,AdamiJDE2,Angulo1,Angulo2,Kairzhan,KP1,KP2} and
for general metric graphs \cite{AdaSerTil_cvpde15,AdaSerTil_jfa16,CDS,D18}. Bifurcations and stability of standing waves
were further explored for tadpole graphs \cite{NPS}, dumbbell graphs \cite{G19,MarPel_amrx16}, double-bridge graphs
\cite{NRS}, and periodic ring graphs \cite{D19,GilgPS,P18,PS17}. A variational characterization of standing waves was developed
for graphs with compact nonlinear core \cite{T1,T2,T3}. Some of these examples will be reviewed in the limit
of large mass as applications of our general results.

The paper is organized as follows. The rigorous formulation of the
asymptotic approach is achieved by defining a nonlinear analogue of
the well-known Dirichlet-to-Neumann (DtN) map, an object we call {\em
  the DtN manifold}.  The properties of the DtN manifold are described
in Section \ref{sec:DTN}, first in the linear theory and then for the
stationary NLS equation in the limit of large mass.  Edge-localized
states are constructed by matching a strongly localized
large-amplitude elliptic function constructed on a single edge of the
graph with a small amplitude solution on the rest of the graph and the
matching is done by finding an intersection of two relevant DtN
manifolds.  This is performed in Section \ref{sec:constructing} for
the three types of edges (a pendant, a loop, and an internal edge).
In Section~\ref{sec:comparing} we prove the comparison lemma and apply
it to the proof of Theorem~\ref{thm:main_modest2} and
Corollary~\ref{cor:main_modest3}.  In Section \ref{sec:examples}, we
present numerical studies of generalized dumbbell graphs, generalized
tadpole graph and a periodic graph.

Appendix~\ref{section:DtN_appendix} contains a proof of the
asymptotic representation of the Dirichlet--to--Neumann map in the
linear theory. Appendix~\ref{sec:maximum_principle} quotes a maximum
principle that is useful to understanding the behavior of solutions
in the region where they are small.
Appendix~\ref{sec:contraction_mapping} collects together the
well-known results on the contraction mapping principle and the
implicit function theorem used in our work.
Appendix~\ref{sec:elliptic} reports on useful asymptotic expansions
for the elliptic functions and gives a ``reverse Sobolev inequality'':
an estimate of the $H^2$ norm of an apriori bounded solution of the
stationary NLS in terms of its small $L^\infty$ norm.

\vspace{0.25cm}

{\sc Acknowledgments.} The first author was
supported in part by National Science Foundation under Grants
DMS--1815075 and DMS--1410657.  The second author was supported in
part by U.S. NSF Grant DMS--1312874 and NSF CAREER Grant DMS-1352353.
The third author is supported in part by the NSERC Discovery Grant.

The authors wish to thank Riccardo Adami, Roy Goodman and Enrico Serra
for helpful conversations that led to the development of this work.
In particular, this article corrects a computational error that
occurred in the final proof of \cite{MarPel_amrx16} and was discovered
thanks to an observation of R. Adami and E. Serra.

\section{Graphs inside-out: Dirichlet-to-Neuman map}
\label{sec:DTN}

The main idea for constructing the edge-localized state satisfying the
stationary NLS equation (\ref{statNLS}) is to match a large solution
of the known form on a single edge of the graph $\Graph$ with a small
solution on the rest of the graph denoted by $\Graph^c$.  The
``feedback'' from the small solution on $\Graph^c$ to the large
solution on the single edge is encoded via the nonlinear analogue of
the Dirichlet-to-Neumann (DtN) map which is developed in this
section. For simplicity of notations in this section, we use the same
notation $\Graph$ instead of $\Graph^c$.

\subsection{Linear DtN map; asymptotics below spectrum}

We start by reviewing the linear DtN map. Consider a graph $\Graph$ with a finite number of vertices and a
finite number of edges, which either connect a pair of vertices and
have finite length or have only one vertex and are identified with the
half-line.  We impose Neumann--Kirchhoff (NK) conditions at every
vertex.  Declare a subset $B$ of the graph's vertices to be the
\emph{boundary}.  We are interested in the asymptotics of the DtN map
on the boundary $B$ for the operator $-\Delta + \mu^2$ as
$\mu \to \infty$.

Before we give a precise definition, a couple of remarks are in order.
The operator $-\Delta$ on $L^2(\Graph)$ with NK
conditions on the vertices is well known to be non-negative and therefore we are looking at asymptotics far
away from its spectrum.  By using a scaling transformation, the same question
can be interpreted as asymptotics for the operator $-\Delta + 1$ on $L^2(\Graph_{\mu})$
as $\mu \to \infty$, where the graph $\Graph_{\mu}$ is
obtained from $\Graph$ by scaling all the edge lengths by a large parameter $\mu$.
This is the point of view we will use in most of the manuscript.

Let the boundary vertices be denoted $v_1, \ldots, v_b$, $b=|B|$, and let
$\vecp = (p_1, \ldots, p_b)^T \in \bbR^b$ be a vector of ``Dirichlet
values'' on the vertices. Figure ~\ref{fig:dtn_map} (left) gives a schematic
representation of the graph $\Graph$ with boundary vertices.
The graph $\Graph_{\mu}$ is obtained from $\Graph$ by multiplying all edge lengths by the same value $\mu$.
The infinite edges of the graph $\Graph$ are unaffected by this transformation.

\begin{figure}[t]
  \centering
  \includegraphics{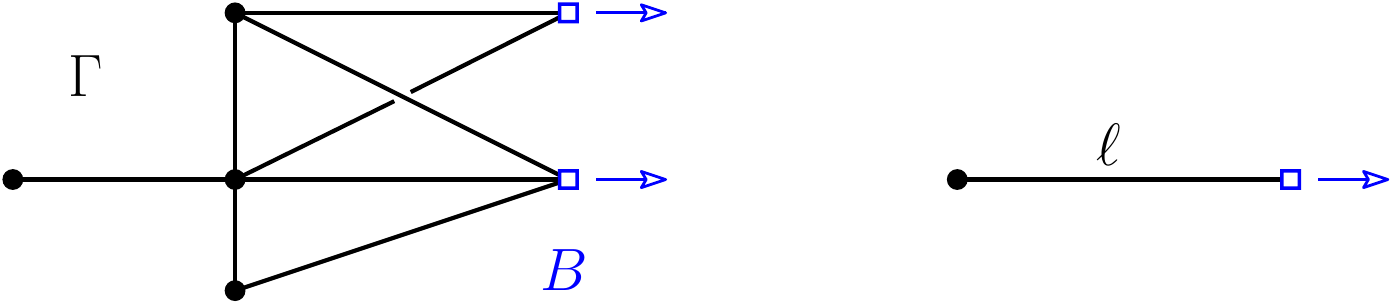}
  \caption{Left: a graph $\Graph$ with boundary vertices $B$ marked as empty
    squares. Arrows indicate the outgoing derivatives of the eigenfunction
    in the Neumann data. Right: a simple graph from Example~\ref{ex:DTN}.}
  \label{fig:dtn_map}
\end{figure}

Let $u \in H^2(\Graph_\mu)$ be a solution of the following boundary--value problem:
\begin{equation}
  \begin{cases}
    \left(-\Delta + 1\right) u = 0, \qquad
    & \mbox{on every }e\in\Graph_{\mu},\\
    u \mbox{ satisfies NK conditions}
    & \mbox{for every} \;\; v\in V\setminus B,\\
    u(v_j) = p_j,\qquad
    & \mbox{for every} \;\; v_j \in B.
  \end{cases}
  \label{eq:DtN_linear_rescaled}
\end{equation}
Existence and uniqueness of the solution $u$ follows from
invertibility of the operator $\left(-\Delta + 1\right)$ with
homogeneous vertex conditions; see, for
example, \cite[Section 3.5.2]{BerKuc_graphs}.
Note that $u$ is not required to satisfy the current
conservation conditions at $v_j \in B$.  Let
\begin{equation}
  \label{eq:def_Neumann_data}
  q_j = \Neu(u)_j := \sum_{e \sim v_j} \partial u_e(v_j),
\end{equation}
be the Neumann data of the function $u$ at the vertex $v_j \in B$,
where $\partial$ denotes the outward derivative from the vertex $v_j$.
The map $\DTN(\mu) : \vecp \mapsto {\bf q} := (q_1,\ldots,q_b)^T$ is
called the DtN map.  The following theorem (proved in Appendix A)
provides its asymptotics as $\mu\to\infty$.

\begin{theorem}
  \label{thm:asymptotic_dtn_lengths}
  The unique solution
  $u \in H^2(\Graph_{\mu})$ to the boundary value problem
  (\ref{eq:DtN_linear_rescaled}) satisfies asymptotically, as
  $\mu \to \infty$,
  \begin{equation}
    \label{eq:asymptotics_L2H2_norm_lengths}
    \left\| u \right\|_{H^2(\Graph_{\mu})}^2
    \sim \left\| u \right\|_{L^2(\Graph_{\mu})}^2 \leq C
    \left(1 + \mathcal{O}(\mu e^{-\mu \Lmin}) \right)\|\vecp\|^2
  \end{equation}
  and
  \begin{equation}
    \label{eq:asymptotics_dtn_lengths}
    \DTN(\mu) = \diag(d_j)_{j=1}^b
    + \mathcal{O}\left(e^{-\mu \Lmin}\right),
  \end{equation}
  where $\Lmin$ is the minimal edge length of the original graph $\Graph$.
\end{theorem}

\begin{example}
  \label{ex:DTN}
  In the simplest case, the graph $\Graph$ is one edge of length
  $\ell$ with the boundary vertex at $x = \ell$ and the other vertex at $x=0$
  under the Neumann condition, as is shown on Fig.~\ref{fig:dtn_map} (right).
  It is straightforward to obtain the following solution of the
  boundary-value problem (\ref{eq:DtN_linear_rescaled}):
  \begin{equation}
  \label{dtn-0}
    u(z) = p \frac{\cosh(z)}{\cosh(\mu \ell)}, \quad z \in [0,\mu \ell].
  \end{equation}
  The DtN map $M(\mu) : p \mapsto q$ is one-dimensional with $q = u'(\mu \ell)$ and
  \begin{equation}
  \label{dtn-1}
    \DTN(\mu) = \tanh(\mu \ell),
  \end{equation}
  and the solution (\ref{dtn-0}) satisfies
  \begin{equation}
  \label{dtn-2}
    \left\| u \right\|_{L^2(0,\mu \ell)}^2
    = \frac{1}{2} p^2 \left[ \tanh(\mu \ell)
      + \mu\ell {\rm sech}^2(\mu \ell) \right].
  \end{equation}
  The latter quantities are expanded as $\mu \to \infty$,
  in agreement with \eqref{eq:asymptotics_L2H2_norm_lengths} and
  \eqref{eq:asymptotics_dtn_lengths}. Note that the error bound $\mathcal{O}(e^{-\mu \ell})$
  in \eqref{eq:asymptotics_L2H2_norm_lengths} and
  \eqref{eq:asymptotics_dtn_lengths} is larger compared to the error
  bound $\mathcal{O}(e^{-2\mu \ell})$ following from (\ref{dtn-1}) and (\ref{dtn-2}),
  due to cancellations specific to this simple example.
\end{example}

For future use we now establish a related auxiliary estimate for
the following non-homogeneous boundary--value problem:
\begin{equation}
    \label{eq:DtN_nonhom}
    \begin{cases}
      \big(-\Delta + 1 - W\big) f = g,
      & \mbox{on every } e\in\Graph_{\mu},\\
      f \mbox{ satisfies NK conditions}
      & \mbox{for every } v\in V\setminus B,\\
      f(v_j) = p_j,
      & \mbox{for every } v_j \in B,
    \end{cases}
  \end{equation}
where $g \in L^2(\Graph_{\mu})$ and $W \in L^{\infty}(\Graph_{\mu})$ are given
and $f \in H^2(\Graph_{\mu})$ is to be found.

\begin{lemma}
  \label{lem:Neumann_est_nonhom}
  For every $\mu > 0$, $g \in L^2(\Graph_\mu)$ and
  $W \in L^{\infty}(\Graph_{\mu})$ satisfying
  $\| W \|_{L^\infty(\Graph_{\mu})} \leq \alpha < 1$, there exists a
  unique solution $f \in H^2(\Graph_{\mu})$ to the boundary value
  problem (\ref{eq:DtN_nonhom}).  Asymptotically in $\mu \to \infty$
  (assuming that $\alpha$ is independent of $\mu$) we have
    \begin{equation}
    \label{eq:f_norm_nonhom}
    \| f \|_{L^2(\Graph_{\mu})} \leq C_{\alpha} \left( \|\vecp\| + \|g\|_{L^2(\Graph_{\mu})} \right),
  \end{equation}
  with the Neumann data of $f$ on $B$ satisfying
  \begin{equation}
    \label{eq:Neumann_est_nonhom}
    \|\Neu(f)\| \leq C_{\alpha} \left( \|\vecp\| + \|g\|_{L^2(\Graph_{\mu})} \right),
  \end{equation}
  where the constant $C_{\alpha}$ is independent of $\mu$.
\end{lemma}

\begin{proof}
  Represent $f = u + \xi$, where $u$ is the solution to the boundary-value
  problem~\eqref{eq:DtN_linear_rescaled}.  Let us define the operator
  \begin{equation}
  \label{invertible}
  -\Delta + 1 - W : \quad {\rm Dom}(\Graph^D_{\mu}) \subset L^2(\Graph_{\mu}) \mapsto L^2(\Graph_{\mu}),
  \end{equation}
  where ${\rm Dom}(\Graph^D_{\mu}) \subset H^2(\Graph_{\mu})$ is the domain of the Laplacian
  $-\Delta$ on the graph $\Graph_{\mu}$ with homogeneous Dirichlet conditions at the boundary $B$
  (the rest of the vertices retain their NK conditions). Since $1 - W(z) \geq 1 - \alpha > 0$
  the operator (\ref{invertible}) is invertible which implies that
  there cannot be more than one solution $f$.  Since $\big(-\Delta +
  1\big)u=0$ and $u$ takes care of the non-homogeneous boundary values, the remainder term $\xi$ is given by
  \begin{equation}
  \label{xi-term}
    \xi = \big(-\Delta + 1 - W\big)^{-1} (g + Wu).
  \end{equation}
  The inverse operator $\big(-\Delta + 1 - W\big)^{-1}$ is bounded as an
  operator from $L^2(\Graph_\mu)$ to $H^2(\Graph_\mu)$ uniformly in
  $\mu \geq \mu_0$.  Therefore,
  the $L^2(\Graph_{\mu})$ norm of $\xi$ and the Neumann trace $\Neu(\xi)$ are estimated from (\ref{xi-term}) as follows:
  \begin{equation*}
  \| \xi \|_{L^2(\Graph_{\mu})} + \| \Neu(\xi) \| \leq C_{\alpha} \|g+Wu\|_{L^2(\Graph_{\mu})} \leq C_{\alpha} \left( \|g\|_{L^2(\Graph_{\mu})}
    + \alpha \|u\|_{L^2(\Graph_{\mu})} \right),
  \end{equation*}
  where we have implicitly used the Sobolev embedding $\| \partial u \|_{L^\infty (\Graph)} \leq C \| u \|_{H^2( \Graph)}$.
  The $L^2(\Graph_{\mu})$ norm of $u$ and the Neumann data $\Neu(u)$
  is bounded by $C\|\vecp\|$ thanks to
  \eqref{eq:asymptotics_L2H2_norm_lengths} and
  \eqref{eq:asymptotics_dtn_lengths} respectively.  Using $f=u+\xi$
  and $\Neu(f) = \Neu(u) + \Neu(\xi)$ we obtain
  (\ref{eq:f_norm_nonhom}) and (\ref{eq:Neumann_est_nonhom}).
\end{proof}

\subsection{Definition of nonlinear DtN manifold}

The analogue of DtN map for the stationary NLS equation is what we call a ``nonlinear DtN manifold.''  The name
is chosen because in most cases of interest (such as
in the example we consider in Section~\ref{sec:nlinDTNexample}) this
object turns out to be a geometric manifold.  It is also not a
``map'' due to lack of uniqueness of the solution to the stationary NLS equation.

\begin{definition}
  \label{def:DTNvariety}
  Consider a $\mu$-scaled graph $\Graph_{\mu}$ with a boundary $B$.  The
  \emph{DtN manifold} is the set
  $N \subset \bbR^{|B|}\times\bbR^{|B|}$ of $(\vecp,\vecq)$ such that
  there is a solution $\Psi \in H^2(\Graph_{\mu})$ of the following
  nonlinear boundary value problem:
  \begin{equation}
  \left\{
  \begin{array}{ll}
    \left(-\Delta + 1 \right) \Psi = 2|\Psi|^2 \Psi,
      \qquad & \mbox{on every }e\in\Graph_{\mu},\\
    \Psi \mbox{ satisfies NK conditions}
             & \mbox{for every} \;\; v\in V\setminus B,\\
    \Psi(v_j) = p_j,\qquad & \mbox{for every }v_j \in B,\\
    \sum_{e \sim v} \partial\Psi_e(v_j) = q_j, \qquad
             & \mbox{for every }v_j \in B,
  \end{array}\right.
  \label{eq:DtN_nonlinear}
  \end{equation}
  where $\partial$ denotes the outward derivative at the vertex $v_j \in B$.
\end{definition}

In the same way that linear DtN map is intricately related to the
scattering matrix, the DtN manifold is related to the nonlinear
scattering map defined in \cite{GnuSmiDer_pra11}.  Exploring this
connection further lies outside the scope of this article.

\subsection{An example of DtN manifold}
\label{sec:nlinDTNexample}

We will now describe the nonlinear analogue of Example~\ref{ex:DTN}. To do so,
let us briefly recall the structure of the solutions of the stationary NLS equation on the line given by
\begin{equation}
  \label{eq:stationaryNLSstandard}
  -\Psi'' + \Psi = 2 |\Psi|^2 \Psi.
\end{equation}
Equation~\eqref{eq:stationaryNLSstandard} is translation- and
phase-invariant.  We will impose, for definiteness, the condition
$\Psi'(0)=0$ and $\Psi(0) \in \mathbb{R}$, obtaining
a list of real-valued solutions of the differential equation (\ref{eq:stationaryNLSstandard}).
All other real-valued solutions may be obtained from the listed ones by translations.
More general complex-valued solutions also exist but they are beyond the scope of this work.

There are three constant solutions to (\ref{eq:stationaryNLSstandard}):
$\Psi = 0$ and $\Psi = \pm \frac1{\sqrt{2}}$.  There exists a $H^2(\bbR)$ solution
called the {\em NLS soliton}:
\begin{equation}
  \label{eq:sech_soliton}
  \Psi(z) =  \sech(z), \quad x \in \bbR.
\end{equation}
This solution separates two families of periodic wave solutions
expressible in terms of Jacobian elliptic functions (see 8.14 in
\cite{GraRyzh_table}).  These are the sign-indefinite {\em cnoidal waves}
\begin{equation}
  \label{eq:cnoidal}
  \Psi_{\rm cn}(z) = \frac{\kappa}{\sqrt{2\kappa^2-1}}
  \cn\left(\frac{z}{\sqrt{2\kappa^2-1}};\kappa\right), \quad \kappa \in \left(\frac{1}{\sqrt{2}},1\right)
\end{equation}
and the sign-definite {\em dnoidal waves}
\begin{equation}
  \label{eq:dnoidal}
  \Psi_{\rm dn}(z) = \frac{1}{\sqrt{2-k^2}}
  \dn\left(\frac{z}{\sqrt{2-k^2}};k\right), \quad k \in (0,1),
\end{equation}
where $\kappa$ (corresp. $k$) is the elliptic modulus.  These
solutions are illustrated in Fig.~\ref{fig:phase_port_nls}.

\begin{figure}[t]
  \centering
  \includegraphics[scale=0.75]{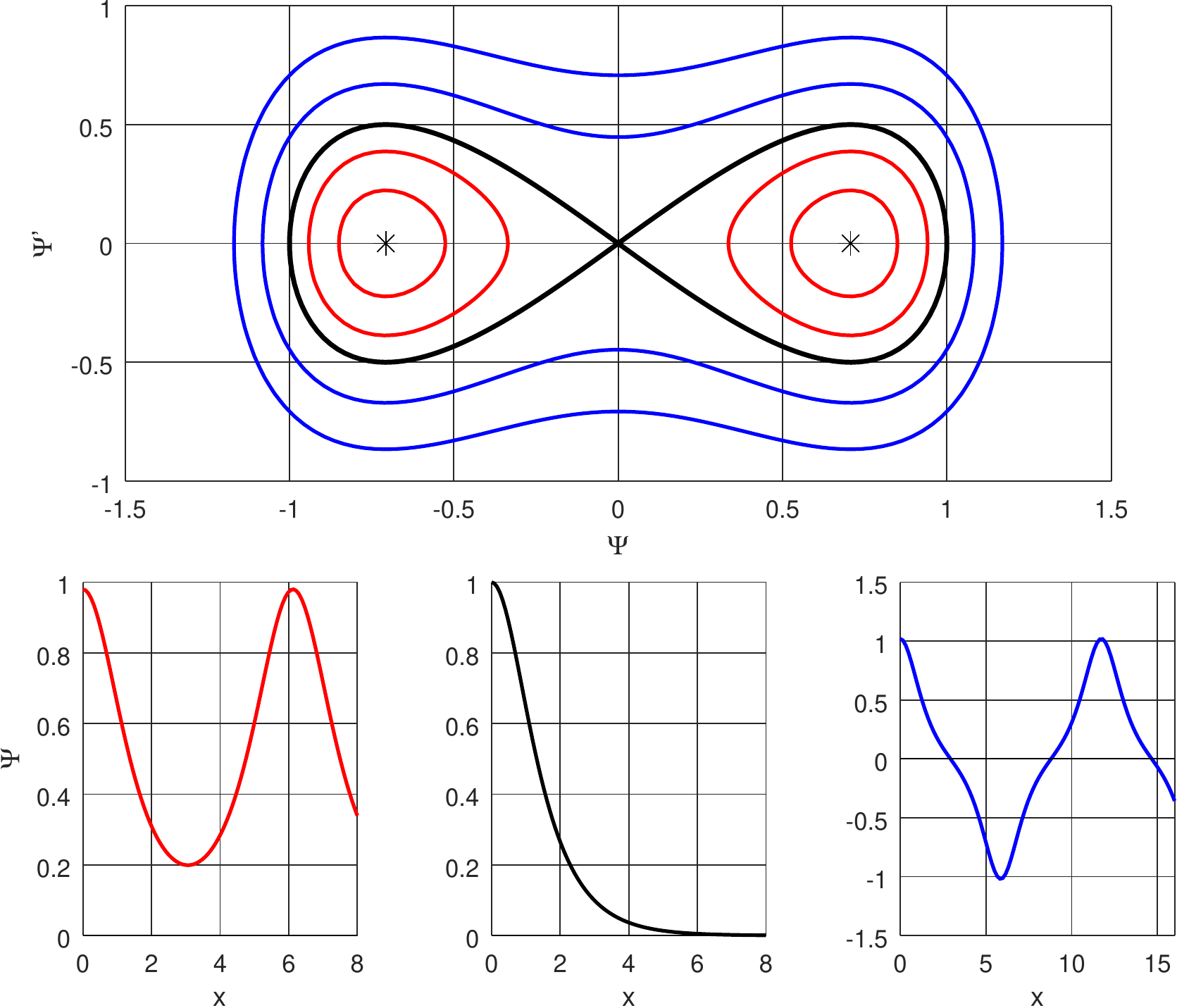}
  \caption{Top: phase portrait for the second-order equation
    $-\Psi''+\Psi-2\Psi^3=0$.  Bottom: typical solutions with initial
    conditions $\Psi'(0)=0$ and, from left to right, $\Psi(0) = 0.98$
    (``dnoidal wave''), $\Psi(0)=1$ (``NLS soliton'') and
    $\Psi(0)=1.02$ (``cnoidal wave'').}
  \label{fig:phase_port_nls}
\end{figure}

The Jacobi real transformation implies that letting
$\kappa = 1/k$ in equation \eqref{eq:cnoidal} transforms it into
equation \eqref{eq:dnoidal} with $k>1$ (see 8.153.5-6 in
\cite{GraRyzh_table}).  We will thus use the single analytic
expression
\begin{equation}
  \label{eq:dnoidal_universal}
  \Psi_{\rm n}(z) = \frac{1}{\sqrt{2-k^2}}
  \dn\left(\frac{z}{\sqrt{2-k^2}};k\right), \quad k \in (0,\sqrt{2})
\end{equation}
to describe the solutions (letter ``n'' can be interpreted as
``noidal'' or as referring to the Neumann-type condition
$\Psi'(0)=0$).  In particular, setting $k=1$ reproduces the NLS soliton~\eqref{eq:sech_soliton}.

\begin{figure}[t]
  \centering
  \includegraphics[scale=0.75]{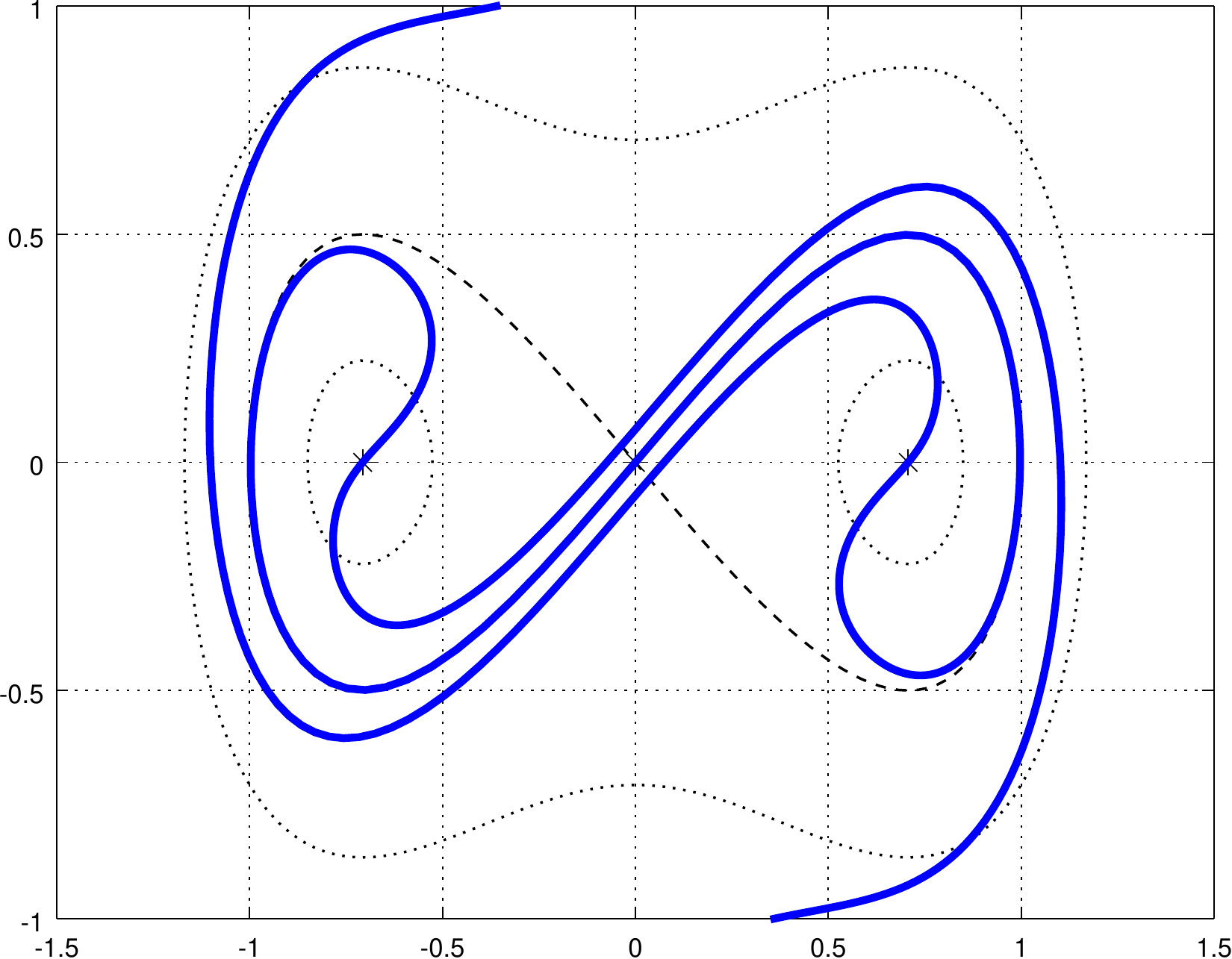}
  \caption{Nonlinear DtN manifold (thicker line) for (\ref{eq:DTN_nls_edge}) with $L=4$
  plotted on top of the phase portrait from Fig. \ref{fig:phase_port_nls} (dashed and dotted lines).}
  \label{fig:DTN_nls_edge}
\end{figure}

\begin{example}
  \label{ex:nlin_DTN}
  Consider the simple graph of Example~\ref{ex:DTN}.  The DtN manifold can be
  obtained by going through all real solutions of the second-order equation (\ref{eq:stationaryNLSstandard})
  on the interval $[0,L]$ with zero derivative and variable initial value at $z = 0$.  In other words,
  \begin{equation}
    \label{eq:DTN_nls_edge}
    N_L = \Big\{\big(\Psi(L), \Psi'(L)\big): \quad  -\Psi''+\Psi-2\Psi^3=0,\
    \Psi'(0)=0,\ \Psi(0) \in \bbR \Big\}.
  \end{equation}
  The DtN manifold is shown in Fig.~\ref{fig:DTN_nls_edge} for $L = 4$.  There are
  many peculiar and complex features, but we will concentrate on the
  three nearly straight parallel curves in the neighborhood of (0,0).  The
  middle curve is tangential to the corresponding \emph{linear} DtN
  map; the other two curves will allow us to construct stationary
  states localized on a single edge of the graph.
\end{example}

\subsection{Nonlinear DtN manifold in the almost linear regime}

Consider the nonlinear boundary value problem on a $\mu$-scaled graph $\Graph_\mu$ with a boundary $B$,
\begin{equation}
\left\{ \begin{array}{ll}
    \left(-\Delta + 1\right) \Psi = 2|\Psi|^2 \Psi,
      \qquad & \mbox{on every }e\in\Graph_{\mu},\\
    \Psi \mbox{ satisfies NK conditions} & \mbox{for every }v\in V\setminus B,\\
    \Psi(v_j) = p_j,\qquad & \mbox{for every }v_j \in B.
  \end{array} \right.
  \label{eq:nlin}
\end{equation}
We will establish the existence and uniqueness of small solutions
of this boundary value problem in the limit $\mu \to \infty$ and for
small boundary data
$\vecp = (p_1,\ldots,p_{|B|})$ in Theorem~\ref{thm:nlin_DTN} below.
But first we discuss some helpful properties of such small solutions.

\begin{lemma}
  \label{lem:suppressed_solution}
  Suppose $\Psi\in H^2(\Graph_{\mu})$ is a solution to
  the boundary-value problem \eqref{eq:nlin}
  satisfying the uniform bound
  \begin{equation}
    \label{eq:unif_bound}
    |\Psi(z)| < \frac{1}{\sqrt{2}}
    \qquad
    \mbox{for all }z\in \Graph_\mu.
  \end{equation}
  Then $|\Psi|$ has no internal local maxima in $\Graph_{\mu} \setminus B$
  and the maximum of $|\Psi|$ is attained on $B$.
  If, additionally, all $p_j \geq 0$, then $\Psi(z) \geq 0$ for all
  $z\in \Graph_\mu$.

  Conversely, if the boundary values $p_j$ of a function $\Psi$
  satisfy $|p_j| < \frac{1}{\sqrt{2}}$ and $|\Psi|$ has no internal
  local maxima in $\Graph_{\mu} \setminus B$, the global bound
  \eqref{eq:unif_bound} is satisfied.
\end{lemma}

\begin{remark}
The upper bound in \eqref{eq:unif_bound} comes from the
location of the rightmost fixed point in the phase portrait in
Fig.~\ref{fig:phase_port_nls}.
\end{remark}

\begin{proof}
  If $|\Psi|$ has no internal local maxima on $\Graph_{\mu} \setminus B$,
  the maximum of $|\Psi|$ is attained on the boundary $B$, since $|\Psi(z)| \to 0$ along the unbounded edges of
  $\Graph_{\mu}$.

  Since $\Psi$ satisfies \eqref{eq:nlin}, we can view it as a solution
  to $(-\Delta + V) \Psi = 0$ with $V=1 - 2|\Psi|^2$.  By assumption,
  $V > 0$ and we can use a maximum principle in the form quoted in
  Appendix~\ref{sec:maximum_principle}, Lemma $B.1$ to conclude that
  $\max(\Psi, 0)$ has no local maxima in $\Graph_{\mu} \setminus B$.
  Similarly, $\max(-\Psi,0)$ has no local maxima.

  If, additionally,
  all boundary values $p_j$ are non-negative and $\Psi(z)<0$ is achieved for some
  $z \in \Gamma_{\mu}$, the function $\Psi$ must have a negative
  internal local minimum.  Therefore $|\Psi|$ would have an internal
  local maximum, a possibility that we just ruled out. Hence,
  $\Psi(z) \geq 0$ for all $z \in \Gamma_{\mu}$ if $p_j \geq 0$ for all $j$.
\end{proof}

We also prove a useful ``reverse Sobolev estimate'' which is so
called because it goes in the reverse direction to the usual
Sobolev-type estimates of $L^\infty$ norm in terms of a Sobolev norm.
The ``reverse'' inequality becomes possible if we assume a priori that
the function satisfies NLS and is sufficiently small.

\begin{lemma}
  \label{lem:reverse_Sobolev_graph}
  There exist $c_0$, $\mu_0$ and $C$ (which may depend on the graph
  $\Gamma_{\mu}$), such that every real solution
  $\Psi\in H^2(\Gamma_\mu)$ of the stationary NLS equation
  $-\Psi'' + \Psi = 2 |\Psi|^2 \Psi$ satisfying
  \begin{equation}
    \label{eq:bound-on-Psi}
    |\Psi(z)| < c_0, \quad z \in \Gamma_\mu, \quad \mu > \mu_0
  \end{equation}
  also satisfies
  \begin{equation}
    \label{eq:H2_from_Linf}
    \| \Psi \|_{H^2(\Gamma_\mu)} \leq C \|\Psi\|_{L^\infty(\Gamma_\mu)}.
  \end{equation}
\end{lemma}

Lemma~\ref{lem:reverse_Sobolev_graph} follows from the corresponding
inequality on every edge of the graph $\Gamma_\mu$, see
Proposition~\ref{prop:H2_from_Linf}.  The proof of
Proposition~\ref{prop:H2_from_Linf} is rather technical and is postponed to
Appendix~\ref{sec:elliptic}.


We now formulate and prove the main result of this section.

\begin{theorem}
  \label{thm:nlin_DTN}
  There are $C_0 > 0$, $p_0 > 0$ and $\mu_0>0$ such that for every
  $\vecp = (p_1, \ldots, p_{|B|})$ with $\|\vecp\| < p_0$ and every
  $\mu>\mu_0$, there exists a solution $\Psi \in H^2(\Graph_{\mu})$
  to the boundary-value problem (\ref{eq:nlin}) which
  is unique among functions satisfying the uniform bound \eqref{eq:unif_bound}.

  The solution $\Psi$ satisfies the estimate
  \begin{equation}
    \label{eq:nlin_DTN_solution}
    \| \Psi \|_{H^2(\Graph_\mu)} \leq C_0 \|\vecp\|,
  \end{equation}
  while its Neumann data
  ${\bf q} = (q_1, \ldots, q_{|B|}) := \Neu(\Psi)$ satisfies
  \begin{equation}
    \label{eq:nlin_DTN_value}
    |q_j - d_j p_j| \leq C_0 \left( \|\vecp\| e^{-\mu \Lmin} +
      \|\vecp\|^3 \right),
    \qquad 1 \leq j \leq |B|,
  \end{equation}
  where $d_j$ is the degree of the $j$-th boundary vertex and $\Lmin$
  is the length of the shortest edge in $\Graph$.  The Neumann data
  ${\bf q}$ is $C^1$ with respect to $\vecp$ and $\mu$.  The partial
  derivatives satisfy the following estimates:
  \begin{equation}
    \label{eq:nlin_DTN_value_der}
    \left| \frac{\partial q_j}{\partial p_i} - d_j \delta_{ij} \right| \leq C_0 \left( e^{-\mu \Lmin} + \|\vecp\|^2 \right),
    \quad 1 \leq i,j \leq |B|,
  \end{equation}
  and
  \begin{equation}
    \label{eq:nlin_DTN_value_der_mu}
    \left|\frac{\partial q_j}{\partial\mu}\right|
    \leq C_0 \mu^{-1} \|\vecp\|, \quad 1 \leq j \leq |B|.
  \end{equation}
 Furthermore, if $p_j \geq 0$ for every $j$, then $\Psi(z) \geq 0$ for
 all $z \in \Graph$.
\end{theorem}

\begin{proof}
For the nonlinear boundary value problem (\ref{eq:nlin}) we decompose
  \begin{equation}
  \label{decomp-psi}
    \Psi = u + \tpsi,
  \end{equation}
  where $u \in H^2(\Graph_{\mu})$ satisfies the linear boundary value problem
  \eqref{eq:DtN_linear_rescaled} and $\tpsi\in H^2(\Graph_\mu)$ satisfies
  \begin{equation}
    \label{eq:nlin-term}
    \begin{cases}
      \left(-\Delta + 1 \right) \tpsi
      = 2 |u + \tpsi|^2 (u + \tpsi),   \qquad
      & \mbox{on every }e\in\Graph_{\mu},\\
      \tpsi \mbox{ satisfies NK conditions}
      & \mbox{for every }v\in V\setminus B,\\
      \tpsi(v_j) = 0,\qquad & \mbox{for every }v_j \in B.
    \end{cases}
  \end{equation}
  Let us denote by $\Dom(\Graph_\mu^D) \subset H^2(\Graph_\mu)$ the
  domain of the Laplacian $-\Delta$ on the graph $\Graph_\mu$ with
  Dirichlet conditions at the boundary $B$ (the rest of the vertices
  retain their NK conditions).  This is a self-adjoint positive
  operator, therefore $-\Delta + 1$ is invertible with
  $(-\Delta + 1)^{-1}$ bounded as an operator from $L^2(\Graph_\mu)$ to
  $H^2(\Graph_\mu)$.

  Since $H^2(\Graph_\mu)$ is a Banach algebra by an application of the
  Sobolev inequality (see Lemma 3.1 in \cite{GilgPS} for the periodic
  graphs setting), the mapping
  $T: \Dom(\Graph_\mu^D) \to \Dom(\Graph_\mu^D)$ defined by
  \begin{equation}
    \label{system-ode-persistence}
    T: \tpsi \mapsto 2 (-\Delta + 1)^{-1} |u + \tpsi|^2
    (u + \tpsi)
  \end{equation}
  satisfies the estimates
  \begin{align}
    \label{eq:estimate_norm}
    \| T(\tpsi) \|_{H^2(\Graph_{\mu})}
    \leq C_1 \|u + \tpsi\|^3_{H^2(\Graph_{\mu})},
  \end{align}
  and
  \begin{align}
    \| T(\tpsi_1) - T(\tpsi_2)\|_{H^2(\Graph_{\mu})}
    \leq C_2 \left(\|u + \tpsi_1\|_{H^2(\Graph_{\mu})}^2 + \|u +
      \tpsi_2\|_{H^2(\Graph_{\mu})}^2 \right)
      \|\tpsi_1 - \tpsi_2\|_{H^2(\Graph_{\mu})}.
  \end{align}
  The latter estimate follows from the elementary inequality
  $$
  |a^3-b^3| = |a-b| |a^2+ab+b^2| \leq \frac32 |a-b| (a^2 + b^2)
  $$
  thanks to the fact that all functions in (\ref{system-ode-persistence}) are real.

  It follows from Theorem~\ref{thm:asymptotic_dtn_lengths} that
  $\|u\|_{_{H^2(\Graph_{\mu})}} \leq C_3\|\vecp\| < C_3 p_0$, hence, taking $p_0$ small
  enough we obtain that $T$ satisfies the conditions of the
  Contraction Mapping Principle (see Theorem \ref{thm:contraction_mapping} in
  Appendix~\ref{sec:contraction_mapping}) in the ball
  $\|\tpsi\|_{H^2(\Graph_{\mu})} < p_0$.  This yields a unique solution
  $\tpsi \in \Dom(\Graph_\mu^D)$ as a fixed point of $T$ satisfying thanks to
  \eqref{eq:fp_estimate} the following estimate:
  \begin{eqnarray}
    \label{estimate-psi}
    \| \tpsi \|_{H^2(\Graph_{\mu})} \leq C_4 \|u\|_{H^2(\Graph_{\mu})}^3
    \leq C_5 \|\vecp\|^3,
  \end{eqnarray}
  for some $\vecp$-independent $C_4,C_5 > 0$.  These estimates,
  together with Theorem~\ref{thm:asymptotic_dtn_lengths}, immediately
  yield estimate \eqref{eq:nlin_DTN_solution} for $\Psi = u + \psi$.

  In order to confirm that $\Psi$ satisfies the uniform bound \eqref{eq:unif_bound}, we use the
    classical Sobolev's inequality (see, for example, \cite[Lemma 1.3.8]{BerKuc_graphs})
    \begin{equation*}
      \|\Psi\|_{L^\infty(0,L)} \leq C \|\Psi\|_{H^2(0,L)},
    \end{equation*}
    where $C$ is independent of $L$ as long as $L>L_0$.
    Hence
    \begin{equation*}
      \|\Psi\|_{L^\infty(\Graph_{\mu})} \leq C \|\Psi\|_{H^2(\Graph_{\mu})},
    \end{equation*}
    where constant $C$ is independent of $\mu$. Then, the bound \eqref{eq:nlin_DTN_solution}
    implies estimate \eqref{eq:unif_bound}.

    In order to show that the small solution $\Psi$ with the given small
    boundary data $\vecp$ is unique, we use
    Lemma~\ref{lem:suppressed_solution} to conclude that
    \begin{equation*}
      \|\Psi\|_{L^\infty(\Graph_\mu)} \leq \max_{j} p_j \leq p_0,
    \end{equation*}
    and then use Proposition~\ref{prop:H2_from_Linf} to get a bound on
    $\|\Psi\|_{H^2(\Graph_\mu)}$.  We conclude that $\psi := \Psi - u$
    is $H^2$-small which puts it into the domain of contraction of
    $T$.  Uniqueness of $\psi$ and hence of $\Psi$ then follows from the unique solution in the
    Contraction Mapping Principle.

  The Neumann data for $\Psi$ is the sum of the Neumann data for $u$
  and the Neumann data for $\psi$. The former is bounded by \eqref{eq:asymptotics_dtn_lengths}.
  The latter is estimated using \eqref{estimate-psi} and the
  continuity in $H^2(\Graph_{\mu})$ of the Neumann trace.  Combining the two
  estimates, we obtain \eqref{eq:nlin_DTN_value}.

  We now apply Corollary~\ref{cor:contraction_smooth} to the mapping
  $T$ defined in \eqref{system-ode-persistence} to conclude that the
  fixed point $\psi$ is $C^1$ in $u \in H^2(\Graph_{\mu})$. In turn, $u \in H^2(\Graph_{\mu})$ is $C^1$ in ${\bf p}$
  because the boundary value problem \eqref{eq:DtN_linear_rescaled}
  is linear in ${\bf p}$. The derivative $\partial_{p_i} u$ satisfies equation
  \eqref{eq:DtN_linear_rescaled} with $\vecp =
  (\delta_{ij})_{j=1}^b$.  By
  Theorem~\ref{thm:asymptotic_dtn_lengths}, we have
  \begin{equation}
    \label{derivative-bound-1}
    \|\partial_{p_i} u\|_{H^2(\Graph_{\mu})} \leq C_5
    \qquad \mbox{and} \qquad
    \partial_{p_i} \Neu(u)_j
    = \Neu\left(\partial_{p_i} u\right)_j
    = d_j \delta_{ij}  + \bigO{e^{-\mu \Lmin}}.
  \end{equation}
  To estimate the derivative of $\Neu(\tpsi)$ we differentiate equation
  \eqref{eq:nlin-term} in $p_j$ (allowed since we already
  established smoothness in $p_j$), to obtain
  \begin{equation}
    \label{eq:dpsi_dp_equation}
    \begin{cases}
      \left(-\Delta + 1 -6 \Psi^2 \right) \partial_{p_i} \tpsi
      = 6 \Psi^2 \partial_{p_i} u,
      & \mbox{on every }e\in\Graph_{\mu},\\
      \partial_{p_i} \tpsi \mbox{ satisfies NK
        conditions}
      & \mbox{for every }v\in V\setminus B,\\
      \partial_{p_i} \tpsi(v) = 0,\qquad & \mbox{for
        every }v \in B.
    \end{cases}
  \end{equation}
  Taking small enough $p_0$ we can ensure,
  see~\eqref{eq:nlin_DTN_solution}, that $\Psi$ is uniformly bounded
  on $\Graph_\mu$ by, say, $1/\sqrt{12}$ and therefore
  \begin{equation}
    \label{eq:Linf_bound}
    12 \Psi^2(z) \leq 1, \qquad z \in \Graph_{\mu}.
  \end{equation}
  We can now apply Lemma~\ref{lem:Neumann_est_nonhom} with
  $\|\vecp\|=0$, $W = 6\Psi^2$ and $g=6 \Psi^2 \partial_{p_i} u$ to
  estimate
  \begin{equation*}
    \|\Neu(\partial_{p_i} \tpsi)\| \leq C_6 \| \Psi^2 \partial_{p_i} u\|_{L^2(\Graph_{\mu})}
    \leq C_7 \|\vecp\|^2,
  \end{equation*}
  using our bounds on $\Psi$ and $\partial_{p_i} u$, see
  \eqref{eq:nlin_DTN_solution} and \eqref{derivative-bound-1}.  Combining this estimate with
  the derivative of $\Neu(u)$ in \eqref{derivative-bound-1} we obtain (\ref{eq:nlin_DTN_value_der}).

  To establish smoothness of ${\bf q}$ in $\mu$ we have to overcome a
  technical difficulty.  The Banach spaces $H^2(\Graph_{\mu})$ and
  $\Dom(\Graph_\mu^D)$ containing $u$ and $\psi$ depend on the parameter
  $\mu$.  To circumvent this problem, we rescale
  \begin{equation}
    \label{scaling}
    \Phi(x) = \mu \Psi(\mu x), \qquad x \in \Graph,
  \end{equation}
  and obtain the boundary value problem on the original graph $\Graph$:
  \begin{equation}
    \begin{cases}
      \left(-\Delta + \mu^2 \right) \Phi = 2|\Phi|^2 \Phi,
      & \mbox{on every }e\in\Graph,\\
      \Phi \mbox{ satisfies NK conditions},
      & \mbox{for every }v\in V\setminus B,\\
      \Phi(v_j) = \mu p_j,
      & \mbox{for every }v_j \in B.
    \end{cases}
    \label{eq:nlin-scaled}
  \end{equation}
  We already established that there exists a unique solution
  $\Phi \in H^2_{\Graph}$ to the boundary-value problem (\ref{eq:nlin-scaled})
  for every $\mu > \mu_0$.  Moreover, bound
  \eqref{eq:Linf_bound} on $\Psi$ translates into the similar bound
  on $\Phi$, namely
    \begin{equation}
    \label{eq:Linf_bound-Gamma}
    12 \Phi^2(x) \leq \mu^2, \qquad x \in \Graph.
  \end{equation}
  We will now fix $\mu$ and reformulate \eqref{eq:nlin-scaled} in a
  form where we can apply the Implicit Function Theorem (see
  Theorem~\ref{thm:implicit_function}).  In particular,
  to get a mapping smooth in $\Phi$ (the Jacobian must be a
  \emph{bounded} operator) we need to invert
  $\left(-\Delta + \mu^2 \right)$ which means that we have to fix the
  boundary conditions first.  Similarly to previous decomposition $\Psi = u + \psi$, we decompose
  $\Phi = w + \phi$, where $w(x) = \mu u(\mu x)$ satisfies
  the inhomogeneous boundary-value problem:
  \begin{equation}
  \begin{cases}
    \left(-\Delta + \mu^2 \right) w = 0, \qquad
    & \mbox{on every }e\in\Graph,\\
    w \mbox{ satisfies NK conditions}
    & \mbox{for every} \;\; v\in V\setminus B,\\
    w(v_j) = \mu p_j,\qquad
    & \mbox{for every} \;\; v_j \in B.
  \end{cases}
  \label{eq:DtN_linear_w}
\end{equation}
  The remainder $\phi$ belongs to
  $H^2(\Graph)$ with Dirichlet conditions at $B$ and NK conditions
  elsewhere; we denote this space by $\Dom(\Graph^D) \subset H^2(\Graph)$.
  Let $F$ be the following mapping from $X\times Y := \bbR^1 \times \Dom(\Graph^D)$ to $Z := \Dom(\Graph^D)$:
  \begin{equation}
  \label{map-F-tech}
    F: \left(\mu, \phi\right) \mapsto
    \phi - 2 \left( -\Delta + \mu^2 \right)^{-1} |w + \phi|^2 (w + \phi).
  \end{equation}
 Note that the map $F$ in (\ref{map-F-tech}) can be derived from the map \eqref{system-ode-persistence}
 after rescaling \eqref{scaling} and rewriting the fixed-point problem as the root-finding problem.

  There exists a solution $\phi \in \Dom(\Graph^D)$ given by
  $\phi(x) = \mu \psi(\mu x)$, where $\psi \in H^2(\Graph_{\mu})$ is the
  fixed point of $T$ in (\ref{system-ode-persistence}). We check that
  the Jacobian $D_\phi F(\mu,\phi)$ has a bounded inverse.
  The Jacobian applied to $h \in \Dom(\Graph^D)$ is given by
  \begin{equation}
    \label{eq:diff_F_phi}
    D_\phi F(\mu,\phi) h := h - 6 \left( -\Delta + \mu^2 \right)^{-1} | w + \phi |^2 h,
  \end{equation}
  and solving $D_\phi F(\mu,\phi) h = g$ results in
  \begin{equation}
    \label{eq:diff_F_phi_inverse}
    h = g + 6 \left( -\Delta + \mu^2 - 6 |w + \phi |^2 \right)^{-1} |w + \phi |^2 g.
  \end{equation}
  The right-hand side is a bounded operator from
  $\Dom(\Graph^D)$ to $\Dom(\Graph^D)$ because of the bound (\ref{eq:Linf_bound-Gamma}).
  In addition, since $w$ is $C^1$ in $\mu$ as follows from (\ref{eq:DtN_linear_w}),
  we have that $F(\mu,\phi)$ is $C^1$ in $\mu$. By
  the Implicit Function Theorem (Theorem \ref{thm:implicit_function}), $\phi$ is $C^1$ in $\mu$, so that
  $\Phi = w + \phi \in H^2_{\Graph}$ is also $C^1$ in $\mu$.

  Having proved smoothness of $\Phi \in H^2_{\Graph}$ in $\mu$, we can now differentiate
  equation~\eqref{eq:nlin-scaled} in $\mu$, resulting in the following
  equation for $\hat{\Phi} := \partial_\mu \Phi$:
  \begin{equation}
    \begin{cases}
      \left( -\Delta + \mu^2 - 6 |\Phi|^2 \right) \hat{\Phi} + 2 \mu \Phi
      = 0 \quad & \mbox{on every }e\in\Graph,\\
      \hat{\Phi} \mbox{ satisfies NK conditions}
      & \mbox{for every }v\in V\setminus B,\\
      \hat{\Phi}(v_j) = p_j
      & \mbox{for every }v_j \in B,
    \end{cases}
    \label{eq:nlin-scaled-dmu}
  \end{equation}
  We undo the rescaling (\ref{scaling}) and
  introduce $\hat{\Phi}(x) = \hat{\Psi}(\mu x)$ satisfying
  \begin{equation}
    \begin{cases}
      \left(-\Delta + 1 - 6 |\Psi|^2 \right) \hat{\Psi} + 2 \Psi
      = 0, \qquad & \mbox{on every }e\in\Graph_\mu,\\
      \hat{\Psi} \mbox{ satisfies NK conditions}
      & \mbox{for every }v\in V\setminus B,\\
      \hat{\Psi}(v_j) = p_j, & \mbox{for every }v_j \in B.
    \end{cases}
    \label{eq:nlin-hat}
  \end{equation}
  We are again in a position to apply
  Lemma~\ref{lem:Neumann_est_nonhom}, with $W = 6 |\Psi|^2$ and $g = -2 \Psi$, obtaining
  from (\ref{eq:Neumann_est_nonhom}):
  \begin{equation}
    \label{eq:bound_Neu_Psihat}
    \big\|\Neu\big(\hat\Psi\big)\big\| \leq C_8 \|\vecp\|.
  \end{equation}
  We unwind all rescalings, first $\vecq = \Neu(\Psi) = \mu^{-2} \Neu(\Phi)$
  and then
  \begin{equation*}
    \frac{\partial \vecq}{\partial \mu}
    = \frac1{\mu^2} \Neu\left(\partial_\mu\Phi\right) - \frac2\mu
    \vecq
    = \frac1{\mu} \left(\Neu\big(\hat\Psi\big) - 2 \vecq\right).
  \end{equation*}
  Both terms in the brackets are bounded by $\|\vecp\|$, due to
  \eqref{eq:bound_Neu_Psihat} and \eqref{eq:nlin_DTN_value}, resulting in \eqref{eq:nlin_DTN_value_der_mu}.
\end{proof}

\begin{remark}
  \label{rem:estimate_on_dPhi}
  Applying Lemma~\ref{lem:Neumann_est_nonhom}
  to \eqref{eq:dpsi_dp_equation} yields
  \begin{equation*}
    \| \partial_{p_i} \psi \|_{L^2(\Graph_{\mu})} \leq C \|\Psi^2 \partial_{p_i} u\|_{L^2(\Graph_{\mu})}
    \leq C \|\vecp\|^2.
  \end{equation*}
  Combining this with (\ref{derivative-bound-1}), we get for
  $\Psi = u + \psi$ and its rescaled version $\Phi$,
  \begin{equation}
    \label{eq:dPhi_dp_est}
    \| \partial_{p_i} \Psi\|_{L^2(\Graph_{\mu})} \leq C, \qquad
    \| \partial_{p_i} \Phi\|_{L^2(\Graph)} \leq C \mu^{1/2}.
  \end{equation}
  Similarly, it follows from \eqref{eq:nlin-hat} that
  \begin{equation}
    \label{eq:dPhi_dmu_est}
    \| \hat{\Psi} \|_{L^2(\Graph_{\mu})} \leq C \| \vecp\|, \qquad
    \| \partial_\mu \Phi \|_{L^2(\Graph)} \leq C \mu^{-1/2} \|\vecp\|,
  \end{equation}
  where the constant $C > 0$ is independent of $\mu$ as $\mu \to \infty$.
\end{remark}

\begin{remark}
  The back-and-forth rescaling in the proof of
  Theorem~\ref{thm:nlin_DTN} may seem superfluous, but there are
  limitations to each setting.  For example, we cannot differentiate
  $\Psi$ with respect to $\mu$ since the domain $\Graph_\mu$ of $\Psi$
  depends on $\mu$.  On the other hand, the Jacobian
  \eqref{eq:diff_F_phi} may not be
  bounded uniformly in $H^2(\Graph)$ as
  $\mu \to \infty$ since the $H^2(\Graph)$ norm of $\Phi$
  grows fast in $\mu$.
\end{remark}

\subsection{Single bump part of the DtN manifold of a Neumann edge}
\label{sec:single_bump_DTN}

We now describe the part of the DtN manifold for the single edge of Example~\ref{ex:nlin_DTN}
that corresponds to single bump solutions.

\begin{figure}[t]
  \centering
  \includegraphics[scale=0.75]{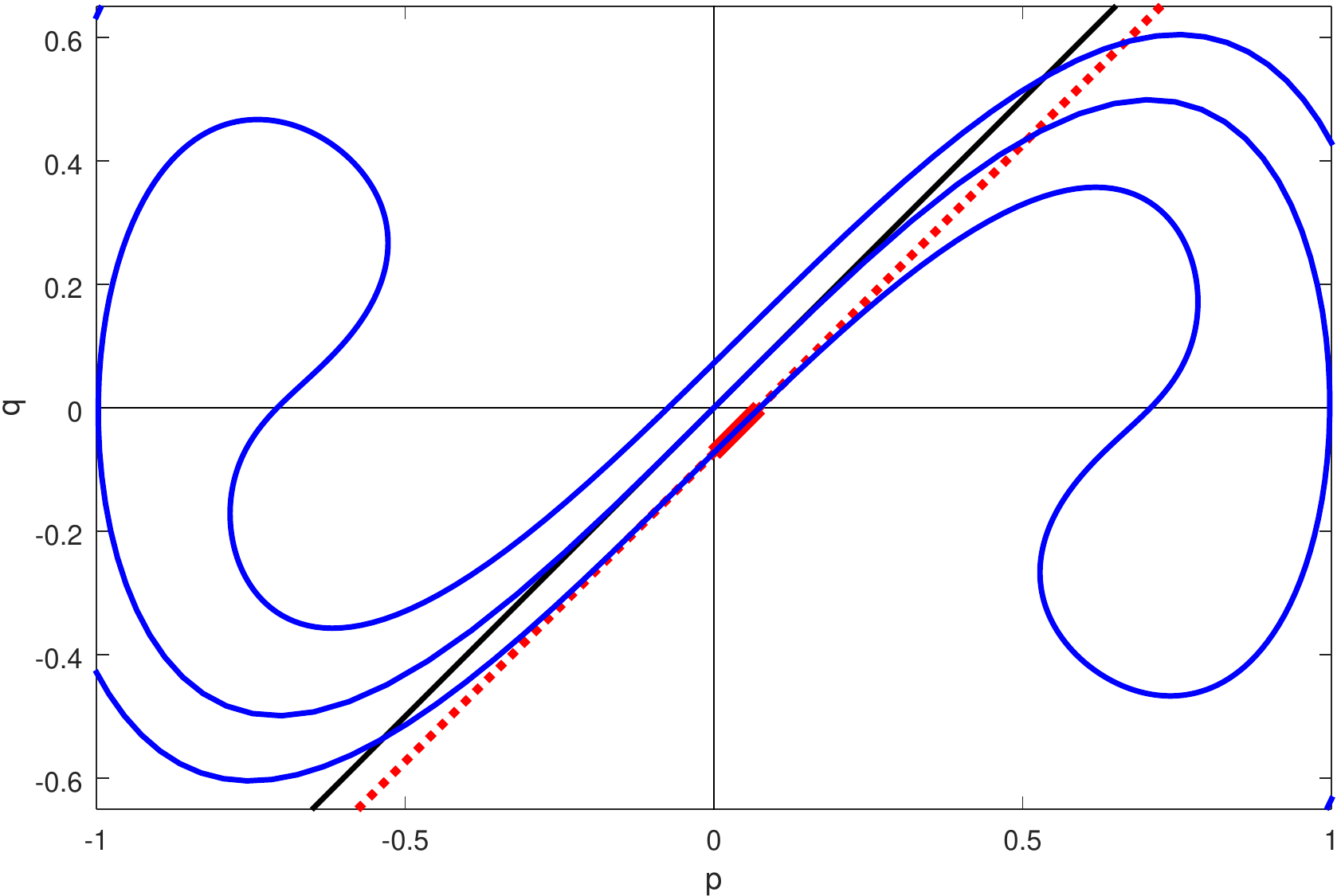}
  \caption{Nonlinear DtN manifold for a single
    interval of length $L=4$ (blue curve) superimposed with the
    asymptotic approximations provided by
    equation~\eqref{eq:nlin_DTN_value} of Theorem~\ref{thm:nlin_DTN}
    (solid straight line) and equations
    \eqref{eq:single-bump-DtN} of
    Lemma~\ref{lem:single_bump_DTN} (dotted straight line).  Thick segment highlights the part
    of the dotted line corresponding to $k\in(k_-,k_+)$ in equation~\eqref{eq:kpm_asympt}.}
  \label{fig:DTN_nls_edge_asympt}
\end{figure}

\begin{lemma}
  \label{lem:single_bump_DTN}
  Consider the DtN manifold in (\ref{eq:DTN_nls_edge}) for the graph $\Graph_{\mu}$ consisting of
  a single edge $[0,L]$ under the Neumann condition at $z = 0$ and the boundary vertex at $z = L$.
  Parameterize $\Psi(0) \in \left(\frac{1}{\sqrt{2}},\infty \right)$ by
  \begin{equation}
    \label{eq:NLS_edge}
   \Psi(0) = \frac{1}{\sqrt{2-k^2}},
  \end{equation}
  where $k \in (0, \sqrt{2})$ is a parameter.  There is an interval
  $(k_-, k_+)$ such that
  a solution $\Psi$ in (\ref{eq:dnoidal_universal}) satisfies
  \begin{align}
    \label{eq:single_bump}
    \Psi(z) > 0,\quad \Psi'(z) < 0, \quad z\in(0,L]
  \end{align}
  if and only if $k\in(k_-, k_+)$.  The boundaries $k_\pm$ have
  the asymptotic expansion
  \begin{equation}
    \label{eq:kpm_asympt}
    k_\pm = 1 \pm 8 e^{-2L} + \bigO{L e^{-4 L}}
    \quad \mbox{\rm as} \quad L \to \infty,
  \end{equation}
  while the boundary values of $\Psi$ are given asymptotically as
  $L \to \infty$ by
  \begin{equation}
    \left\{  \begin{array}{l}
    p_L := \Psi(L) =  2e^{-L} - \frac{1}{4} (k-1) e^{L}
      + \bigO{Le^{-3L}},\\
    q_L :=  \Psi'(L) = -2e^{-L} - \frac{1}{4} (k-1) e^{L }
      + \bigO{Le^{-3L}},
  \end{array} \right.
  \label{eq:single-bump-DtN}
  \end{equation}
where the correction terms denoted by $\bigO{Le^{-3L}}$ are bounded in absolute value by
  $C Le^{-3L }$ for some constant $C$ which is independent of $L$ and
  of $k$, provided $k \in (k_-,k_+)$. Furthermore, the boundary values are $C^1$ functions
  with respect to $k$ and their derivatives are given asymptotically as $L \to \infty$ by
  \begin{equation}
\frac{\partial p_L}{\partial k},  \frac{\partial q_L}{\partial k}  =  - \frac{1}{4} e^{L} + \bigO{Le^{-L}}.
  \label{eq:single-bump-DtN-der}
  \end{equation}
\end{lemma}

\begin{remark}
  \label{rem:monotonicity_k_interval}
  By definition of the interval $(k_-,k_+)$ as the
  maximal set satisfying conditions~\eqref{eq:single_bump}, it is
  monotone in $L$, namely
  \begin{equation*}
    \big(k_-(L'), k_+(L') \big) \subset
    \big(k_-(L), k_+(L)\big)
    \qquad\mbox{if } L > L'.
  \end{equation*}
\end{remark}

\begin{remark}
  Because $k$ is exponentially close to $1$ in \eqref{eq:kpm_asympt}, the two terms in the expansion of
  $p$ and $q$ in \eqref{eq:single-bump-DtN} are of the same order.  These equations give a
  parametric description (the parameter being $k$) of a piece of DtN
  manifold as a line plus smaller order corrections.  This line is
  shown in Fig.~\ref{fig:DTN_nls_edge_asympt} by dotted line together with the
  ``linear approximation'' $q=p$ from Theorem~\ref{thm:nlin_DTN} shown on
  Fig.~\ref{fig:DTN_nls_edge_asympt} by solid straight line.
  The part of the dotted line which corresponds to $k \in (k_-,k_+)$
  is shown on Fig.~\ref{fig:DTN_nls_edge_asympt} by thick solid line.
\end{remark}

\begin{proof}[Proof of Lemma~\ref{lem:single_bump_DTN}]
By using the exact solution (\ref{eq:dnoidal_universal}) satisfying the initial
condition (\ref{eq:NLS_edge}), we obtain
\begin{equation}
\label{p-dn}
p_L =   \frac{1}{\sqrt{2-k^2}} {\rm dn}\left(\frac{L}{\sqrt{2-k^2}};k\right),
\end{equation}
and
\begin{equation}
\label{q-dn}
q_L = -\frac{k^2}{2-k^2} {\rm sn}\left(\frac{L}{\sqrt{2-k^2}};k\right) {\rm cn}\left(\frac{L}{\sqrt{2-k^2}};k\right).
\end{equation}
Let us consider the case $k < 1$. It follows from the single-bump condition (\ref{eq:single_bump})
that $q_L \leq 0$ if and only if
$$
L \leq \sqrt{2-k^2} K(k),
$$
where $K(k)$ is the complete elliptic integral of the first
kind. We will give a brief review of elliptic integrals in Appendix~\ref{sec:elliptic}.
By using the asymptotic expansion (see 8.113 in \cite{GraRyzh_table})
\begin{equation}
\label{expansion-K}
K(k) = \log\left(\frac{4}{\sqrt{1-k^2}}\right) + \mathcal{O}\left((1-k^2)|\log(1-k^2)|\right)\quad
\mbox{\rm as} \quad k \to 1,
\end{equation}
we verify that $k_- = 1 - 8 e^{-2L} + \mathcal{O}(L e^{-4 L})$ is an asymptotic
solution to $L = \sqrt{2-k^2} K(k)$ in the limit $L \to \infty$ and that
the condition $L \leq \sqrt{2-k^2} K(k)$ is satisfied for all $k \in (k_-,1)$.
By Proposition \ref{prop-elliptic}, the asymptotic expansions
\eqref{eq:single-bump-DtN} follow from expansion of (\ref{p-dn}) and (\ref{q-dn}) as $k \to 1$
uniformly in $k \in [k_-,1]$.

The case $k > 1$ is obtained similarly but the condition $k \in [1,k_+]$
appears from the requirement that $p_L \geq 0$ in the single-bump condition (\ref{eq:single_bump}).

The asymptotic expansions for derivatives (\ref{eq:single-bump-DtN-der}) follow
from differentiation of (\ref{p-dn}) and (\ref{q-dn}) with respect to $k$ and substitution
of the asymptotic results of Proposition \ref{prop-elliptic}.
\end{proof}

\section{Constructing the edge-localized stationary solutions}
\label{sec:constructing}

We now prove the existence of edge-localized solutions of the stationary NLS equation (\ref{statNLS})
in the limit $\Lambda \to -\infty$.
We will match the single-bump parts of the DtN manifold
on a single edge of the graph $\Graph$ with the almost linear parts of the DtN manifold on the remainder of the graph,
henceforth denoted $\Graph^c$.  The solution will then be small on
$\Graph^c$ while it will be large and localized on the single edge of $\Graph$.

The scaling transformation (\ref{scaling})  transforms the stationary NLS equation (\ref{statNLS})
with $\Lambda = -\mu^2 < 0$ on the graph $\Graph$
to the stationary NLS equation on the $\mu$-scaled graph $\Graph_{\mu}$,
\begin{equation}
\label{statNLS-limit-graph}
(-\Delta + 1) \Psi = 2 |\Psi|^2 \Psi.
\end{equation}
$\Phi \in H^2_{\Graph}$ is a solution of (\ref{statNLS}) if and only if
$\Psi \in H^2_{\Graph_{\mu}}$ is a solution of (\ref{statNLS-limit-graph}).
We shall now develop the asymptotic solution for $\Psi \in H^2_{\Graph_{\mu}}$
separately for the three types of edges on Fig.~\ref{fig:edge_cases}.

\subsection{Pendant edge}
\label{sec:contructing_pendant}

\begin{theorem}
  \label{thm:localized_pendant}
  Let $\Graph_\mu$ be a graph with $\mu$-scaled edge lengths and with
  a pendant edge of length $L=\mu\ell$ attached to the remainder of
  the graph, $\Graph_\mu^c$, by a vertex $v$ of degree $N+1$, see
  Fig.~\ref{fig:edge_cases}(a).  Then, for large enough $\mu$, there
  is a unique solution $\Psi \in H^2_{\Graph_{\mu}}$ to the stationary
  NLS equation (\ref{statNLS-limit-graph}) with the following properties:
  \begin{itemize}
  \item the solution is strictly positive on the pendant edge and decreases monotonically
  from its maximum at the vertex of degree one to the attachment vertex $v$,
  \item it is positive and has no internal local maxima on the remainder graph $\Graph_\mu^c$.
  \end{itemize}
  On the pendant edge, the solution is described by
  \eqref{eq:dnoidal_universal} with
  \begin{equation}
    \label{eq:k_answer_pendant}
    k = 1 + 8\frac{N-1}{N+1}e^{-2 \mu \ell} + \bigO{e^{-2\mu \ell - \mu \Lmin}},
  \end{equation}
  where $\Lmin$ is the length of the shortest edge in $\Graph^c$.
  The corresponding solution $\Phi \in H^2_{\Graph}$
  to the stationary NLS equation (\ref{statNLS}) with $\Lambda = -\mu^2$
  on the original graph $\Graph$ concentrates on the pendant edge, so that
    \begin{equation}
    \label{eq:norm_remainder_pendant}
    \|\Phi\|_{L^2(\Graph^c)}^2 \leq C \mu e^{-2 \mu \ell}.
  \end{equation}
whereas the mass and energy integrals $\mathcal{Q} := \mathcal{Q}(\Phi)$
and $\mathcal{E} := \mathcal{E}(\Phi)$ in (\ref{energy}) are expanded asymptotically by
  \begin{equation}
    \label{eq:mass_answer_pendant}
    \mathcal{Q} = \mu - 8 \frac{N-1}{N+1} \mu^2 \ell e^{-2 \mu \ell} + \bigO{\mu e^{-2 \mu \ell}}
  \end{equation}
and
  \begin{equation}
    \label{eq:energy_answer_pendant}
    \mathcal{E} = -\frac{1}{3} \mu^3 + \bigO{\mu^4 e^{-2 \mu \ell}}.
  \end{equation}
  The mass integral $\mathcal{Q}$ is a $C^1$ increasing function of $\mu$ when $\mu$ is large.
\end{theorem}

\begin{remark}
  \label{rem:case_N1_pendant}
  Unless the graph $\Graph$, which we assume to be connected, is a
  single interval, the degree of the attachment vertex $v$ is
  $N+1 \geq 2$, that is, $N \geq 1$. It is well-known that a vertex of
  degree $2$ with NK conditions can be absorbed into the edge without
  affecting any solutions, while increasing effective edge length $\ell$ and thus
  making our estimates sharper.  Therefore, the result above is only
  useful with $N\geq 2$.  Still, it is valid for $N=1$.
\end{remark}

\begin{proof}
  Let $L = \mu \ell$ be the length of the pendant edge on the
  $\mu$-scaled graph $\Graph_{\mu}$.  All solutions $\Psi$ satisfying
  the desired properties in the pendant edge are described by
  Lemma~\ref{lem:single_bump_DTN} with $k$ in the allowed region
  $(k_-,k_+)$.  On the other hand, given a small boundary value $p$
  on the attachment vertex $v$, there is a unique solution in $\Psi \in H^2(\Graph^c)$,
  which is described by Lemma \ref{lem:suppressed_solution} and
  Theorem~\ref{thm:nlin_DTN}.  Matching two DtN manifolds in
  Theorem~\ref{thm:nlin_DTN} and Lemma~\ref{lem:single_bump_DTN} at
  the attachment vertex $v$ between the pendant edge and the graph
  $\Graph^c_{\mu}$ we get
  \begin{equation}
    \label{eq:matching_pendant_general}
    p = p_{L}, \qquad q = -q_{L}.
  \end{equation}
  The above discussion shows that the solutions $\Psi$ with the
  desired properties are \emph{in one-to-one correspondence} with the
  roots $k$ of equation~\eqref{eq:matching_pendant_general}, where $p_L$ and
  $q_L$ are functions of $k$ by Lemma \ref{lem:single_bump_DTN}
  and $q$ is a function of $p$ via Theorem~\ref{thm:nlin_DTN}.

  Since, by Lemma \ref{lem:single_bump_DTN}, the value $p_L$ at the
  vertex $v$ of the pendant edge is exponentially small in the large
  parameter $L$, we are indeed justified in using
  Theorem~\ref{thm:nlin_DTN} to conclude that $q \approx N p = N p_L$,
  where $N$ is the degree of the attachment vertex $v$ in the graph
  $\Graph^c$.  More precisely, we have
  \begin{equation}
    \label{eq:matching_pendant1}
    -q_L = N p_L + \mathcal{R}(p_L,\mu),
  \end{equation}
  where by \eqref{eq:nlin_DTN_value}, \eqref{eq:nlin_DTN_value_der}
  and (\ref{eq:nlin_DTN_value_der_mu}), the remainder function
  $\mathcal{R}$ satisfies the bounds
  \begin{eqnarray}
  \label{remainder-estimate-1}
    \left\{ \begin{array}{l}
    |\mathcal{R}(p,\mu)| \leq C(pe^{-\mu \Lmin} + p^3), \\
    |\partial_p \mathcal{R}(p,\mu)| \leq C(e^{-\mu \Lmin} + p^2), \\
    | \partial_{\mu} \mathcal{R}(p,\mu) | \leq C \mu^{-1} p,\end{array} \right.
  \end{eqnarray}
  for some $C > 0$ independently of large $\mu$ and small $p$.  Here
  $\Lmin$ is the minimal edge length in $\Graph^c$, but naturally the
  estimate remains valid if we take $\Lmin$ to be the minimal edge
  length in the whole of $\Graph$.  For large enough $\mu$ and any $k$
  in the allowed region $(k_-, k_+)$, it follows from
  \eqref{eq:single-bump-DtN} that $0 \leq p_L \leq c e^{-\mu \ell}$
  for some $c > 0$.  Therefore, the absolute value of
  $\mathcal{R}(p_L)$ (which depends on $k$) is uniformly bounded by
  $C (e^{-\mu(\ell +\Lmin)} + e^{-3 \mu \ell}) \leq C e^{-\mu(\ell
    +\Lmin)}$ because $\ell \geq \Lmin$.

  For convenience we rescale the parameter $k$ by substituting
  \begin{equation}
    \label{eq:k_rescaling}
    k-1 = 8e^{-2 \mu \ell} x,  \qquad x\in(x_-, x_+)
  \end{equation}
  with $x_{\pm} = \pm 1 + \bigO{\mu e^{-2 \mu \ell}}$. Thanks to the expansion \eqref{eq:single-bump-DtN} with $L = \mu \ell$ and
 the scaling (\ref{eq:k_rescaling}), we can write
 \begin{equation}
 \label{p_L-expansion}
 p_L = 2 e^{-\mu \ell} \left[ 1 - x + \mathcal{R}_p(x,\mu) \right],
 \end{equation}
  where the remainder function $\mathcal{R}_p$ satisfies the bounds
  \begin{eqnarray}
  \label{remainder-estimate-2}
\left\{ \begin{array}{l}
    |\mathcal{R}_p(x,\mu)| \leq C \mu e^{-2\mu \ell}, \\
    |\partial_x \mathcal{R}_p(x,\mu)| \leq C \mu e^{-2\mu \ell},\\
    |\partial_{\mu} \mathcal{R}_p(x,\mu) | \leq C \mu e^{-2\mu \ell},\end{array} \right.
  \end{eqnarray}
  for some $C > 0$ independently of large $\mu$ and $x \in (x_-,x_+)$.
  In order to derive (\ref{remainder-estimate-2}) for the derivatives of $\mathcal{R}_p(x,\mu)$,
  we use the chain rule and the estimates (\ref{eq:single-bump-DtN}) and (\ref{eq:single-bump-DtN-der}):
  \begin{eqnarray*}
  \frac{\partial p_L}{\partial x} = 8 e^{-2\mu \ell} \frac{\partial p_L}{\partial k} =
  -2 e^{-\mu \ell} \left[1 + \mathcal{O}(\mu e^{-2\mu \ell}) \right] = 2 e^{-\mu \ell}
  \left[ -1  + \partial_x \mathcal{R}_p(x,\mu)\right]
  \end{eqnarray*}
  and
  \begin{eqnarray*}
  \frac{\partial p_L}{\partial \mu} & = & \ell \frac{\partial p_L}{\partial L} -
  16 \ell x e^{-2\mu \ell} \frac{\partial p_L}{\partial k} =
  \ell q_L - 2 \ell x \frac{\partial p_L}{\partial x} =
  -2 \ell e^{-\mu \ell} \left[1 - x + \mathcal{O}(\mu e^{-2\mu \ell}) \right] \\
  & = & - 2 e^{-\mu \ell}
  \left[ \ell(1 - x + \mathcal{R}_p(x,\mu)) - \partial_{\mu} \mathcal{R}_p(x,\mu)\right].
  \end{eqnarray*}
 Similarly, we can write
 \begin{equation}
 \label{q_L-expansion}
 q_L = -2 e^{-\mu \ell} \left[ 1 + x + \mathcal{R}_q(x,\mu) \right],
 \end{equation}
  where the remainder function $\mathcal{R}_q$ satisfies the same bounds (\ref{remainder-estimate-2}) as $\mathcal{R}_p$.
  Upon substituting  (\ref{eq:k_rescaling}), (\ref{p_L-expansion}), and (\ref{q_L-expansion})
  into equation \eqref{eq:matching_pendant1}, we obtain
  \begin{equation}
    \label{eq:matching_pendant2-additional}
  1 + x + \mathcal{R}_q(x,\mu) = N \left[ 1 - x + \mathcal{R}_p(x,\mu) \right] + \hat{\mathcal{R}}(x,\mu),
  \end{equation}
  where $\mathcal{R}(p_L,\mu) = 2 e^{-\mu \ell} \hat{\mathcal{R}}(x,\mu)$
  and the new remainder term $\hat{\mathcal{R}}$ satisfies the bounds
  \begin{eqnarray}
  \label{remainder-estimate-3}
   \left\{ \begin{array}{l}
   |\hat{\mathcal{R}}(x,\mu)| \leq C e^{-\mu \Lmin}, \\
   |\partial_x \hat{\mathcal{R}}(x,\mu)| \leq C e^{-\mu \Lmin},\\
   |\partial_{\mu} \hat{\mathcal{R}}(x,\mu) | \leq C \left( \mu^{-1} + e^{-\mu \Lmin} \right),\end{array} \right.
  \end{eqnarray}
  for some $C > 0$ independently of large $\mu$ and $x \in (x_-,x_+)$.
  In order to derive (\ref{remainder-estimate-3}), we have used the fact that
  $\mu e^{-3 \mu \ell} \ll e^{-\mu (\ell+\Lmin)}$ because $\ell \geq \Lmin$,
  as well as the chain rule
  \begin{eqnarray*}
  \frac{\partial \hat{\mathcal{R}}}{\partial x} = \frac{1}{2} e^{\mu \ell} \frac{\partial \mathcal{R}}{\partial p_L} \frac{\partial p_L}{\partial x}
  =  \left[ -1  + \partial_x \mathcal{R}_p(x,\mu)\right] \frac{\partial \mathcal{R}}{\partial p_L}
  \end{eqnarray*}
  and
    \begin{eqnarray*}
  \frac{\partial \hat{\mathcal{R}}}{\partial \mu} =
  \ell \hat{\mathcal{R}} + \frac{1}{2} e^{\mu \ell} \frac{\partial \mathcal{R}}{\partial \mu}
  + \frac{1}{2} e^{\mu \ell}  \frac{\partial \mathcal{R}}{\partial p_L}  \frac{\partial p_L}{\partial \mu}.
  \end{eqnarray*}
   Rearranging (\ref{eq:matching_pendant2-additional}), we get
  \newcommand{\tcR}{\widetilde{\mathcal{R}}}
  \begin{equation}
    \label{eq:matching_pendant2}
    x = \frac{N-1}{N+1} + \tcR(x,\mu),
  \end{equation}
  where the remainder $\tcR := \frac{1}{N+1} (\hat{\mathcal{R}} + N \mathcal{R}_p - \mathcal{R}_q)$
  satisfies the bounds
  \begin{eqnarray}
    \label{remainder-estimate-last}
\left\{ \begin{array}{l} |\tcR(x,\mu)| \leq C \left( e^{-\mu \Lmin} + \mu e^{-2\mu \ell} \right), \\
|\partial_x \tcR(x,\mu)| \leq C \left( e^{-\mu \Lmin} + \mu e^{-2\mu \ell} \right), \\
|\partial_{\mu} \tcR(x,\mu) | \leq C \left( \mu^{-1} + e^{-\mu \Lmin} + \mu e^{-2\mu \ell}\right), \end{array} \right.
  \end{eqnarray}
  for some $C > 0$ independently of large $\mu$ and $x \in (x_-,x_+)$.
  For large enough $\mu$, the
  right-hand side of \eqref{eq:matching_pendant2} maps the interval
  $(x_-, x_+)$ into a subset of $(x_-,x_+)$, moreover, the map is contractive in $(x_-,x_+)$.
  By the Contraction Mapping Principle (see Theorem \ref{thm:contraction_mapping}),
  there exists a unique solution of the scalar equation (\ref{eq:matching_pendant2}) satisfying the estimate
  \begin{equation}
    \label{eq:matching_pendant3}
    \left| x - \frac{N-1}{N+1} \right| \leq C \left( e^{-\mu \Lmin} + \mu e^{-2\mu \ell} \right)
  \end{equation}
  where the constant $C > 0$ is independent of $\mu$ for large $\mu$. Since $p = p_L$ is expanded by
  (\ref{p_L-expansion}), the estimate (\ref{eq:k_answer_pendant}) follows from (\ref{eq:k_rescaling}) and (\ref{eq:matching_pendant3}).
Since the scalar equation (\ref{eq:matching_pendant2}) is $C^1$ in $\mu$,
Corollary~\ref{cor:contraction_smooth} implies that the root $x$ in (\ref{eq:matching_pendant3}) is $C^1$
in $\mu$ satisfying the estimate
\begin{equation}
\label{eq:matching_pendant4}
\left|\frac{dx}{d\mu} \right| \leq C \left( \mu^{-1} + e^{-\mu \Lmin} + \mu e^{-2\mu \ell}\right),
\end{equation}
where the constant $C > 0$ is independent of $\mu$ for large $\mu$.

The estimate (\ref{eq:norm_remainder_pendant}) follows from (\ref{eq:nlin_DTN_solution}) with $p = p_L$ given by
(\ref{p_L-expansion}) and (\ref{eq:matching_pendant3}) and the scaling transformation (\ref{scaling}).

We now turn to the expansion (\ref{eq:mass_answer_pendant}) for the mass $\mathcal{Q} := \mathcal{Q}(\Phi)$.  Thanks
to the scaling transformation (\ref{scaling}) and the estimate (\ref{eq:norm_remainder_pendant}), we
can split the mass $\mathcal{Q}$ as follows:
\begin{equation}
\label{estimate-charge-1-cn}
\mathcal{Q} = \| \Phi \|^2_{L^2(0,\ell)} + \| \Phi \|^2_{L^2(\Graph^c)} = \mu \| \Psi \|_{L^2(0,\mu \ell)}^2 + \mathcal{O}(\mu e^{-2\mu \ell}),
\end{equation}
where the first term is needed to be computed up to the accuracy of the remainder term of the $\mathcal{O}(\mu e^{-2\mu \ell})$ error.
The first term in the splitting (\ref{estimate-charge-1-cn}) is estimated from the explicit expression
(\ref{eq:dnoidal_universal}):
\begin{align}
\nonumber
\| \Psi \|_{L^2(0,\mu \ell)}^2 & =  \frac{1}{\sqrt{2-k^2}} \int_0^{\xi_0} {\rm dn}(\xi;k)^2 d\xi \\
& =  \frac{1}{\sqrt{2-k^2}} \int_0^{K(k)} {\rm dn}(\xi;k)^2 d\xi + \frac{1}{\sqrt{2-k^2}}
\int_{K(k)}^{\xi_0} {\rm dn}(\xi;k)^2 d\xi, \label{estimate-charge-2-cn}
\end{align}
where $\xi_0 := \frac{\mu \ell}{\sqrt{2-k^2}} < K(k)$ and $K(k)$ is the complete
elliptic integral of the first kind, see Appendix \ref{sec:elliptic}.
The second term of the decomposition (\ref{estimate-charge-2-cn}) is estimated by
\begin{eqnarray}
\label{estimate-charge-0-cn}
\left| \int_{K(k)}^{\xi_0} {\rm dn}(\xi;k)^2 d\xi \right| = \mathcal{O}(e^{-2 \mu \ell}),
\end{eqnarray}
since $\xi_0 = \mu \ell + \mathcal{O}(\mu e^{-2\mu \ell})$,
$k = 1 + \mathcal{O}(e^{-2\xi_0})$ by using (\ref{eq:k_answer_pendant})
and ${\rm dn}(\xi;k)^2 = \mathcal{O}(e^{-2 \xi})$ for every $\xi \in (\xi_0,K(k))$
thanks to Proposition \ref{prop-elliptic}. Thanks to the estimate (\ref{estimate-charge-0-cn}),
the second term in (\ref{estimate-charge-2-cn}) is comparable with the remainder term in (\ref{estimate-charge-1-cn})
and is much smaller than the first term in (\ref{estimate-charge-2-cn}).
To estimate the first term in (\ref{estimate-charge-2-cn}), we
consider the case $k < 1$ (computations for $k > 1$ are similar).
It follows from 8.114 in \cite{GraRyzh_table} for $k < 1$ and $k \to 1$ that
\begin{align*}
E(k) & :=  \int_0^{K(k)} {\rm dn}(\xi;k)^2 d\xi \\
 & =  1 + \frac{1}{2} (1-k^2) \left[ \log \frac{4}{\sqrt{1-k^2}} - \frac{1}{2} \right] + \mathcal{O}\left((1-k^2)^2|\log(1-k^2)|\right),
\end{align*}
where $E(k)$ is a complete elliptic integral of the second kind, see Appendix \ref{sec:elliptic}. Therefore, we have
\begin{eqnarray}
\nonumber
\| \Psi \|_{L^2(0,\mu \ell)}^2 & = & 1 - \frac{1}{2} (1-k) \log(1-k) + \mathcal{O}(e^{-2 \mu \ell}) \\
\label{estimate-charge-3-cn}
& = & 1 - 8 \frac{N-1}{N+1} \mu \ell e^{-2 \mu \ell} + \mathcal{O}(e^{-2\mu \ell}),
\end{eqnarray}
where the estimate (\ref{eq:k_answer_pendant}) has been used.
Combining (\ref{estimate-charge-1-cn}), (\ref{estimate-charge-2-cn}), (\ref{estimate-charge-0-cn}), and
(\ref{estimate-charge-3-cn}) yields the expansion (\ref{eq:mass_answer_pendant}).

Thanks to the differentiability of $\Phi \in H^2_{\Graph}$ and $k$ in $\mu$,
the map $\mu \mapsto \mathcal{Q}$ is $C^1$. In order to prove monotonicity of $\mathcal{Q}$ with respect to $\mu$,
we differentiate \eqref{estimate-charge-1-cn} in $\mu$ keeping in mind that the solution $\Psi$ (or its
rescaled form $\Phi$) depends on $\mu$ both directly and indirectly,
via the parameters $p$ and $k$, correspondingly.  We have from (\ref{estimate-charge-1-cn}):
  \begin{equation}
    \label{estimate-charge-dmu-2}
    \frac{d \mathcal{Q}}{d\mu} = \| \Psi \|_{L^2(0,\mu \ell)}^2
    + \mu \frac{d}{d\mu} \| \Psi \|_{L^2(0,\mu \ell)}^2
    + \frac{d}{d\mu} \|\Phi\|_{L^2(\Graph^c)}^2.
  \end{equation}
The first term in (\ref{estimate-charge-dmu-2}) yields $1 + \mathcal{O}(\mu e^{-2\mu \ell})$ due
to the estimate \eqref{estimate-charge-3-cn}.
The second term in (\ref{estimate-charge-dmu-2}) is estimated from the chain rule:
  \begin{equation}
    \label{estimate-charge-dmu-3}
    \frac{d}{d\mu} \| \Psi \|_{L^2(0,\mu \ell)}^2 =
    \left(\frac{\partial}{\partial\mu} + \frac{\partial k}{\partial\mu}
      \frac\partial{\partial k}\right)
    \frac{1}{\sqrt{2-k^2}} \int_0^{\xi_0} {\rm dn}(\xi;k)^2 d\xi,
  \end{equation}
  where $\xi_0 := \frac{\mu \ell}{\sqrt{2-k^2}}$.
It follows from (\ref{eq:k_rescaling}) and (\ref{eq:matching_pendant4}) that
$$
\left| \frac{\partial k}{\partial \mu} \right| \leq C e^{-2\mu \ell}.
$$
Furthermore, recall that since $\xi_0 = \frac{\mu \ell}{\sqrt{2-k^2}} = \mu \ell + \mathcal{O}(\mu e^{-2\mu \ell})$,
$k = 1 + \mathcal{O}(e^{-2\xi_0})$, $|{\rm dn}(\xi;k)| \leq C e^{-\xi}$, and $|\partial_k {\rm dn}(\xi;k)| \leq C e^{\xi}$.
As a result, we obtain from (\ref{estimate-charge-dmu-3}) that
\begin{equation*}
\left| \frac{d}{d\mu} \| \Psi \|_{L^2(0,\mu \ell)}^2 \right| \leq C \mu e^{-2\mu \ell}.
\end{equation*}
The last term in \eqref{estimate-charge-dmu-2} is estimated from another chain rule:
  \begin{equation}
    \label{eq:estimate_charge_dmu-4}
    \frac{d}{d\mu} \|\Phi\|_{L^2(\Graph^c)}^2
    = \int_{\Graph^c} \left( \frac{\partial \Phi}{\partial \mu}
      + \frac{d p}{d\mu} \frac{\partial \Phi}{\partial p} \right) \Phi dx,
      \end{equation}
      so that
      \begin{equation*}
   \left|     \frac{d}{d\mu} \|\Phi\|_{L^2(\Graph^c)}^2 \right| \leq C \mu e^{-2\mu \ell},
  \end{equation*}
thanks to the estimates \eqref{eq:dPhi_dp_est}, \eqref{eq:dPhi_dmu_est}, \eqref{eq:norm_remainder_pendant}, (\ref{p_L-expansion}), and (\ref{eq:matching_pendant3}). Combining all estimates together in (\ref{estimate-charge-dmu-2}),
we obtain that
$$
   \frac{d \mathcal{Q}}{d\mu} = 1 + \mathcal{O}(\mu^2 e^{-2\mu \ell}),
$$
hence $\mathcal{Q}$ is monotonically increasing in $\mu$.

Finally, we establish the expansion (\ref{eq:energy_answer_pendant}) for the energy $\mathcal{E} := \mathcal{E}(\Phi)$.
By using the scaling (\ref{scaling}), we split the energy $\mathcal{E}$ into two parts:
\begin{eqnarray}
\nonumber
\mathcal{E} & = & \| \Phi' \|^2_{L^2(0,\ell)} - \| \Phi \|^4_{L^4(0,\ell)} + \| \Phi' \|^2_{L^2(\Graph^c)} - \| \Phi \|^4_{L^4(\Graph^c)} \\
\label{estimate-energy-1}
& = & \mu^3 \left( \| \Psi' \|^2_{L^2(0,\mu \ell)} - \| \Psi \|^4_{L^4(0, \mu \ell)} \right) + {\rm o}(1),
\end{eqnarray}
where ${\rm o}(1)$ denote terms vanishing in the limit of $\mu \to \infty$ thanks to
the estimate (\ref{eq:nlin_DTN_solution}) with $|p| \leq C e^{-\mu \ell}$.
Thanks to the exact solution (\ref{eq:dnoidal_universal}), we estimate the expression in the bracket in (\ref{estimate-energy-1}):
\begin{equation}
\label{estimate-energy-2}
\| \Psi' \|^2_{L^2(0,\mu \ell)} - \| \Psi \|^4_{L^4(0, \mu \ell)} = \frac{1}{(2-k^2)^{3/2}} \int_0^{\xi_0} \left[
k^4 {\rm sn}(\xi;k)^2 {\rm cn}(\xi;k)^2 - {\rm dn}(\xi;k)^4 \right] d\xi,
\end{equation}
where $\xi_0 := \frac{\mu \ell}{\sqrt{2 - k^2}} = \mu \ell + \mathcal{O}(\mu e^{-2 \mu \ell})$.
Thanks to the estimate in Proposition \ref{prop-elliptic} with $k = 1 + \mathcal{O}(e^{-2\mu \ell})$
from (\ref{eq:k_answer_pendant}), the remainder terms to the limiting hyperbolic functions in
(\ref{sn-der}), (\ref{cn-der}), and (\ref{dn-der}) are as small as $\mathcal{O}(e^{-\mu \ell})$ in the $L^{\infty}(0,\xi_0)$ norm,
hence we obtain from (\ref{estimate-energy-2}) that
\begin{eqnarray}
\nonumber
\| \Psi' \|^2_{L^2(0,\mu \ell)} - \| \Psi \|^4_{L^4(0, \mu \ell)} & = & \int_0^{\xi_0} \left[
{\rm sech}(\xi)^2 \tanh(\xi)^2 - {\rm sech}(\xi)^4 \right] d\xi + {\rm o}(1) \\
\nonumber
& = & \int_0^{\infty} \left[
{\rm sech}(\xi)^2 \tanh(\xi)^2 - {\rm sech}(\xi)^4 \right] d\xi + {\rm o}(1) \\
\label{estimate-energy-3}
& = & -\frac{1}{3} + {\rm o}(1).
\end{eqnarray}
Combining (\ref{estimate-energy-1}) and (\ref{estimate-energy-3}) yields $\mathcal{E} = -\frac{1}{3} \mu^3 + {\rm o}(1)$.
Since $\Phi$ is a critical point of the augmented energy $S_{\Lambda}(U) := \mathcal{E}(U) - \Lambda \mathcal{Q}(U)$,
it follows that $\mathcal{Q}$ and $\mathcal{E}$ satisfy the differential equation:
\begin{equation}
\label{estimate-energy-4}
\frac{d \mathcal{E}}{d \Lambda} = \Lambda \frac{d \mathcal{Q}}{d \Lambda} \quad \Rightarrow \quad
\frac{d \mathcal{E}}{d \mu} = - \mu^2 \frac{d \mathcal{Q}}{d \mu}.
\end{equation}
Since $\mathcal{Q}$ is $C^1$ in $\mu$, then $\mathcal{E}$ is $C^1$ in $\mu$. It follows from
the balance of exponential terms in (\ref{eq:mass_answer_pendant}) and (\ref{estimate-energy-4})
that the remainder ${\rm o}(1)$ is given by $\mathcal{O}(\mu^4 e^{-2 \mu \ell})$, which completes
the proof of (\ref{eq:energy_answer_pendant}).
\end{proof}

\subsection{Looping edge}

\begin{theorem}
  \label{thm:localized_loop}
  Let $\Graph_\mu$ be a graph with $\mu$-scaled edge lengths and with
  a looping edge of length $2L=2\mu\ell$ attached to the remainder of
  the graph $\Graph_\mu^c$ by a vertex $v$ of degree $N+2$, see
  Fig.~\ref{fig:edge_cases}(b).  Then, for large enough $\mu$, there
  is a unique solution $\Psi \in H^2_{\Graph_{\mu}}$ to the stationary NLS
  equation (\ref{statNLS-limit-graph}) with the following properties:
  \begin{itemize}
  \item the solution is strictly positive on the looping edge and decreases monotonically
  from its maximum at the midpoint towards the attachment vertex $v$,
  \item it is positive and has no internal local maxima on the remainder graph $\Graph_\mu^c$.
  \end{itemize}
  On the looping edge, the solution is described by
  \eqref{eq:dnoidal_universal} with the origin (maximum) located at the midpoint and with
  \begin{equation}
    \label{eq:k_answer_loop}
    k = 1 + 8\frac{N-2}{N+2}e^{-2\mu \ell} + \bigO{e^{-2\mu \ell - \mu \Lmin}}
  \end{equation}
  where $\Lmin$ is the length of the shortest edge in $\Graph^c$.
  The corresponding solution $\Phi \in H^2_{\Graph}$
  to the stationary NLS equation (\ref{statNLS}) with $\Lambda = -\mu^2$
  satisfies the concentration estimate
  \begin{equation}
    \label{eq:norm_remainder_loop}
    \|\Phi\|_{L^2(\Graph^c)}^2 \leq C \mu e^{-2 \mu \ell},
  \end{equation}
  whereas the mass and energy integrals $\mathcal{Q} := \mathcal{Q}(\Phi)$
  and $\mathcal{E} := \mathcal{E}(\Phi)$ in (\ref{energy}) are expanded asymptotically by
  \begin{equation}
    \label{eq:mass_answer_loop}
   \mathcal{Q} = 2 \mu - 16\frac{N-2}{N+2} \mu^2 \ell e^{-2 \mu \ell}
    + \bigO{\mu e^{-2 \mu \ell}}.
  \end{equation}
and
  \begin{equation}
    \label{eq:energy_answer_loop}
    \mathcal{E} = -\frac{2}{3} \mu^3 + \bigO{\mu^4 e^{-2 \mu \ell}}.
  \end{equation}
The mass integral $\mathcal{Q}$ is a $C^1$ increasing function of
  $\mu$ when $\mu$ is large.
\end{theorem}

\begin{remark}
  The wave on the looping edge is dnoidal for $N=1$ (since $k<1$) and
  cnoidal for $N \geq 3$ (since $k>1$).  Its character in the case
  $N=2$ is undetermined since the first correction vanishes and our
  results do not provide higher order corrections.  However, since
  neither solution changes sign on the edge thanks to the constraints
  \eqref{eq:single_bump} in Lemma~\ref{lem:single_bump_DTN}, the
  difference between the cnoidal and dnoidal waves is largely
  irrelevant.
\end{remark}

\begin{proof}
  Continuity of the solution at the attachment vertex $v$ coupled with
  its single-bump character implies that we can restrict our search to
  the solutions symmetric on the looping edge.  From the midpoint to
  the attachment vertex $v$, possible solutions are described by
  Lemma~\ref{lem:single_bump_DTN}; all of them are exponentially small
  in $L$ at $v$. Given a small boundary value $p$
  on the attachment vertex $v$, there is a unique solution $\Psi \in H^2(\Graph^c)$,
  which is described by Lemma \ref{lem:suppressed_solution} and
  Theorem~\ref{thm:nlin_DTN}.  Therefore, all solutions in the
  prescribed class of functions are in one-to-one correspondence with solutions of
  the following equation from the NK conditions:
  \begin{equation}
    \label{eq:matching_loop1}
    -2 q_L = N p_L + \mathcal{R}(p_L,\mu).
  \end{equation}
  This equation should be interpreted as an equation on $k$ (through
  $p_L$ and $q_L$ that depends on $k$).  This equation replaces equation
  \eqref{eq:matching_pendant1} in the proof of
  Theorem~\ref{thm:localized_pendant}. The rest of the proof is
  identical to the previous one with the change $N \mapsto N/2$ and
  the double factor in (\ref{eq:mass_answer_loop}) and \eqref{eq:energy_answer_loop}
  compared to (\ref{eq:mass_answer_pendant}) and \eqref{eq:energy_answer_pendant}
  thanks to the splitting
  \begin{equation}
    \label{estimate-charge-1-cn-loop}
    L^2(\Graph_{\mu})
    = L^2(-L,0) \oplus L^2(0,L) \oplus L^2(\Graph^c).
  \end{equation}
  Monotonicity of $\mathcal{Q}$ is established by a similar expansion of the
  derivative resulting in $\frac{d \mathcal{Q}}{d\mu} = 2 + \mathcal{O}(\mu^2 e^{-2\mu \ell})$.
\end{proof}

\subsection{Internal edge}

We finally arrive to the ``generic'' type of edge: an edge which
connects two distinct vertices of degree larger than two.  We call such
edges \emph{internal}.

\begin{figure}[t]
  \centering
  \includegraphics{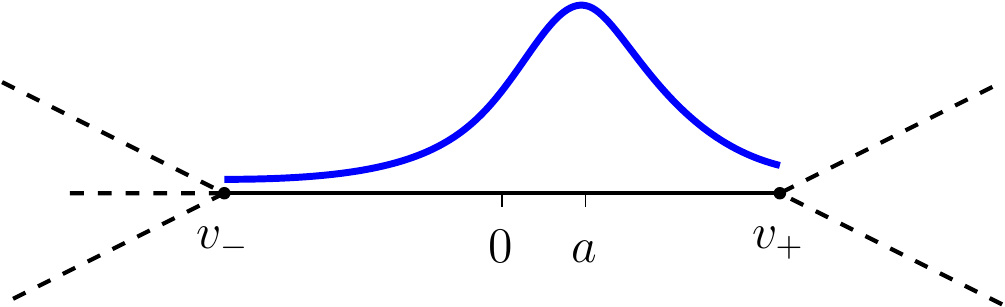}
  \caption{A localized solution localized on a internal edge for
    $N_-=3$ and $N_+=2$. The maximum of the solution is shifted
    to the right of the edge midpoint at the displacement $a$.}
  \label{fig:edge_generic}
\end{figure}

\begin{theorem}
  \label{thm:generic_edge}
  Let $\Graph_\mu$ be a graph with $\mu$-scaled edge lengths and with
  a internal edge $e$ connecting vertices $v_-$ and $v_+$ of degrees
  $N_-+1\geq3$ and $N_++1\geq3$ correspondingly, see
  Fig.~\ref{fig:edge_cases}(c).  Identify the edge $e$ with the
  interval $[-\mu \ell, \mu \ell]$, so that the length of $e$ is
  $2\mu \ell$.  Introduce the notation
  \begin{equation}
    \label{eq:internal_max_loc}
    a_* = \frac12 \tanh^{-1}\left(\frac{N_- - N_+}{N_+ N_- - 1}\right),
  \end{equation}
  and let $(-a_0,a_0)$ be an arbitrary interval containing $a_*$.
  
  Then, for large enough $\mu$,
  there is a unique solution $\Psi \in H^2_{\Graph_{\mu}}$ to the
  stationary NLS equation (\ref{statNLS-limit-graph}) with the following properties:
  \begin{itemize}
  \item on the internal edge $e$ the solution is strictly positive and
    achieves its maximum in the interval $[-a_0,a_0]$; it decreases
    monotonically from its maximum towards the attachment vertex $v$,
  \item on the remainder graph $\Graph_\mu^c$ the solution is positive
    and has no internal local maxima.
  \end{itemize}
  On the internal edge, the solution is described by
  \begin{equation}
    \label{eq:shifted_dnoidal}
    \Psi(z) = \frac{1}{\sqrt{2-k^2}}
    \dn\left(\frac{z-a}{\sqrt{2-k^2}};k\right),
    \quad z \in [-\mu \ell, \mu \ell],
  \end{equation}
  where
  \begin{equation}
    \label{eq:a_answer}
    a = a_* + \bigO{e^{-\mu \Lmin}},
  \end{equation}
  and
  \begin{equation}
    \label{eq:k_answer_generic}
    k = 1 + 8\sqrt{\frac{N_--1}{N_-+1}}
    \sqrt{\frac{N_+-1}{N_++1}} e^{-2 \mu \ell}
    + \bigO{e^{-2\mu \ell - \mu \Lmin}},
  \end{equation}
  where $\Lmin$ is the length of the shortest edge in $\Graph^c$.
  The corresponding solution $\Phi \in H^2_{\Graph}$
  to the stationary NLS equation (\ref{statNLS}) with $\Lambda = -\mu^2$
  on the original graph $\Graph$
  concentrates on the internal edge, so that
  \begin{equation}
    \label{eq:norm_remainder_generic}
    \|\Phi\|_{L^2(\Graph^c)}^2 \leq C \mu e^{-2 \mu \ell},
  \end{equation}
  whereas the mass and energy integrals $\mathcal{Q} := \mathcal{Q}(\Phi)$
and $\mathcal{E} := \mathcal{E}(\Phi)$ in (\ref{energy}) are expanded asymptotically by
\begin{equation}
    \label{eq:mass_answer_generic}
    \mathcal{Q} = 2 \mu  - 16\sqrt{\frac{N_--1}{N_-+1}}
    \sqrt{\frac{N_+-1}{N_++1}}  \mu^2 \ell e^{-2 \mu \ell}
+ \mathcal{O}\left(\mu e^{-2 \mu \ell}\right).
\end{equation}
and
  \begin{equation}
    \label{eq:energy_answer_generic}
    \mathcal{E} = -\frac{2}{3} \mu^3 + \bigO{\mu^4 e^{-2 \mu \ell}}.
  \end{equation}
  The mass integral $\mathcal{Q}$ is a $C^1$ increasing function of
  $\mu$ when $\mu$ is large.
\end{theorem}

\begin{remark}
  \label{rem:case_N1_generic}
  Similar to Remark~\ref{rem:case_N1_pendant},
  estimate~\eqref{eq:mass_answer_generic} should also remain valid in
  the case $N_- = 1$ (or $N_+ = 1$).  Of course, if $N_- = 1$, then
  $v_-$ is a ``spurious'' vertex which can be absorbed into the edge
  thus increasing the length $\ell$ and producing a better estimate.
\end{remark}

\begin{proof}
  Denote by $a$ the location of the maximum of the solution on the internal edge
  $[-\mu \ell, \mu \ell]$ as in Fig.~\ref{fig:edge_generic}.  The
  distance from the vertex $v_-$ to the maximum is $\mu \ell+a$ and
  from the maximum to vertex $v_+$ is $\mu \ell-a$. Assume that $a$
  is defined in an $\mu$-independent interval $[-a_0,a_0]$ for
  a large $a_0 > 0$. Assume the degree of $v_-$ is $N_-+1$ and degree of $v_+$ is $N_++1$, with
  $N_\pm \geq 2$ and, without loss of generality, $N_- \geq N_+$.

  We will now find $a$ and $k$ such that there is a
  solution of the form $\Psi_{\rm n}(z-a)$ on the marked edge, see
  equation~\eqref{eq:dnoidal_universal}.  The distance from the
  maximum at $a$ to either vertex is of order $\mu$. Therefore the
  solution on both sides of the maximum satisfies the setting of
  Lemma~\ref{lem:single_bump_DTN} with the shared value of $k$ in the
  smaller of the two allowed regions, see
  Remark~\ref{rem:monotonicity_k_interval}. Solution values
  $p_{\mu \ell\pm a}$ on the vertices are exponentially small,
  $0 \leq p_{\mu \ell\pm a} \leq C_{\pm} e^{-\mu\ell}$ for $a \in (-a_0,a_0)$
  with $\mu$-independent constants $C_{\pm} > 0$.  Therefore
  we can apply Theorem~\ref{thm:nlin_DTN} resulting in the matching
  conditions
  \begin{equation}
    \label{eq:matching_general1}
    \begin{cases}
      -q_{\mu \ell+a} = N_- p_{\mu \ell+a} + \mathcal{R}^-(p_{\mu \ell + a},\mu),\\
      -q_{\mu \ell-a} = N_+ p_{\mu \ell-a} + \mathcal{R}^+(p_{\mu \ell - a},\mu),
    \end{cases}
  \end{equation}
  where the remainder terms satisfy the bounds
    \begin{eqnarray}
  \label{remainder-estimate-1-internal}
    \left\{ \begin{array}{l}
    |\mathcal{R}^{\pm}(p_{\mu \ell\pm a},\mu)| \leq C (p_{\mu \ell\pm a} e^{-\mu \Lmin} + p_{\mu \ell\pm a}^3), \\
    |\partial_{p_{\mu \ell\pm a}} \mathcal{R}^{\pm}(p_{\mu \ell\pm a},\mu)| \leq C(e^{-\mu \Lmin} + p_{\mu \ell\pm a}^2), \\
    | \partial_{\mu} \mathcal{R}^{\pm}(p_{\mu \ell\pm a},\mu) | \leq C \mu^{-1} p_{\mu \ell\pm a},\end{array} \right.
  \end{eqnarray}
  for some $C > 0$ independently of large $\mu$ and small $p$, similarly
  to the estimates (\ref{remainder-estimate-1}).  As before, solutions
  to (\ref{eq:matching_general1}) are in one-to-one correspondence
  with solutions $\Psi$ of the NLS with the desired properties.

  For convenience we rescale parameter $k$ by substituting
  \begin{equation}
  \label{scaling-internal}
    k-1 = 8e^{-2 \mu \ell} x,  \qquad x\in(x_-, x_+)
    \qquad x_{\pm} = \pm e^{-2 a_0} + \bigO{\mu e^{-2 \mu \ell}}.
  \end{equation}
Thanks to equations~\eqref{eq:single-bump-DtN} with $L = \mu \ell \pm a$ and scaling (\ref{scaling-internal})
we can write
  \begin{equation}
  \label{p-q-representation}
    \begin{cases}
     p_{\mu \ell \pm a} = 2 e^{-(\mu \ell \pm a)} \left[ 1 - x e^{\pm 2a} + \mathcal{R}^{\pm}_p(x,a,\mu) \right], \\
     q_{\mu \ell \pm a} = -2 e^{-(\mu \ell \pm a)} \left[ 1 + x e^{\pm 2a} + \mathcal{R}^{\pm}_q(x,a,\mu) \right],
    \end{cases}
  \end{equation}
  where $\mathcal{R}^{\pm}_p(x,a,\mu)$ and $\mathcal{R}^{\pm}_q(x,a,\mu)$ satisfy the same estimates as in
  \eqref{remainder-estimate-2} for every $a \in (-a_0,a_0)$ and $x \in (x_-,x_+)$ and
  an additional estimate due to the additional parameter $a$:
  \begin{eqnarray}
  \label{remainder-estimate-2-internal}
    |\partial_{a} \mathcal{R}^{\pm}_p(x,\mu) | \leq C \mu e^{-2\mu \ell},
  \end{eqnarray}
  for some $C > 0$ independently of large $\mu$ with $a \in (-a_0,a_0)$ and $x \in (x_-,x_+)$.
  This estimate comes from
  $$
  \frac{\partial p_{\mu \ell \pm a}}{\partial a} = \pm q_{\mu \ell \pm a}
  $$
  and \eqref{p-q-representation}, which gives
  $$
  \frac{\partial \mathcal{R}^{\pm}_p}{\partial a} = \pm \left( \mathcal{R}_p^{\pm} - \mathcal{R}_q^{\pm} \right)
  $$
  and thus gives \eqref{remainder-estimate-2-internal}. Substituting \eqref{scaling-internal} and
  \eqref{p-q-representation} into \eqref{eq:matching_general1} yields
  \begin{equation}
    \label{eq:matching_general2}
    \begin{cases}
      (1-N_-)e^{-a} + x (1+N_-) e^{a} = \widetilde{\mathcal{R}}^-(x,a,\mu),\\
      (1-N_+)e^{a} + x (1+N_+) e^{-a} = \widetilde{\mathcal{R}}^+(x,a,\mu),
    \end{cases}
  \end{equation}
  which is similar to \eqref{eq:matching_pendant2} and where the remainder terms
  $\widetilde{\mathcal{R}}^{\pm}(x,a,\mu)$ and their derivatives in $x$, $a$, and $\mu$
  satisfy estimates similar to the bounds (\ref{remainder-estimate-last}) with
  $|\partial_a  \widetilde{\mathcal{R}}^{\pm}| \leq C |\widetilde{\mathcal{R}}^{\pm}|$.

  Next, we prove that there is a
  solution to the system \eqref{eq:matching_general2}
  in the neighborhood of the point
  \begin{equation}
    \label{eq:solutions_ax}
    a_* = \frac12 \tanh^{-1}\left(\frac{N_- - N_+}{N_+ N_- - 1}\right),
    \qquad
    x_* = \sqrt{\frac{N_--1}{N_-+1}}\sqrt{\frac{N_+-1}{N_++1}},
  \end{equation}
  which are the solutions of system~\eqref{eq:matching_general2} with
  $\widetilde{\mathcal{R}}^{\pm}(x,a,\mu) \equiv 0$.  It can be easily
  seen that both $\tanh^{-1}$ and square roots are well-defined for
  $N_\pm \geq 2$.  Also, $a_* \in (-a_0,a_0)$ for a sufficiently large fixed $a_0 > 0$,
  whereas a neighborhood of $x_*$ belongs to the allowed region $(x_-, x_+)$
  for large enough $\mu$ because
  \begin{equation*}
    e^{-2 |a_*|} = \sqrt{ \frac{(N_+-1)(N_-+1)}{(N_--1)(N_++1)} } > x_*.
  \end{equation*}
  Applying the inverse of the (nonlinear) left-hand side
  of~\eqref{eq:matching_general2} to the right-hand side turns the
  system into a fixed-point problem. The map of the
  fixed-point problem is contractive in $(x_-,x_+) \times (-a_0,a_0)$.
  By the Contraction Mapping Principle
  (see Theorem \ref{thm:contraction_mapping}), there exists a unique
  solution of the system of two nonlinear equations
  (\ref{eq:matching_general2}) satisfying the estimate
 \begin{equation}
    \label{eq:x_answer_general}
    a = a_* + \bigO{e^{-\mu \Lmin}},
    \qquad
    x = x_* + \bigO{e^{-\mu \Lmin}},
  \end{equation}
  thus obtaining \eqref{eq:a_answer} and \eqref{eq:k_answer_generic}.
  Since $0 \leq p_{\mu \ell\pm a} \leq C_{\pm} e^{-\mu\ell}$ for $a \in (-a_0,a_0)$ and $x \in (x_-,x_+)$,
  the estimate (\ref{eq:norm_remainder_generic}) follows from (\ref{eq:nlin_DTN_solution})
  and the scaling transformation (\ref{scaling}).
In order to prove the expansion (\ref{eq:mass_answer_generic}), we partition
  \begin{equation}
\label{estimate-charge-1-cn-generic}
\| \Psi \|_{L^2(\Graph_{\mu})}^2 = \| \Psi \|_{L^2(-\mu \ell,a)}^2 + \| \Psi \|_{L^2(a,\mu \ell)}^2 + \| \Psi \|_{L^2(\Graph^c)}^2.
\end{equation}
Each of the two leading-order integrals in (\ref{estimate-charge-1-cn-generic}) is expanded
similarly to (\ref{estimate-charge-2-cn}), after which the expansion (\ref{estimate-charge-3-cn})
with (\ref{eq:k_answer_generic}) yields (\ref{eq:mass_answer_generic}).
The expansion (\ref{eq:energy_answer_generic}) is derived from a decomposition similar to
(\ref{estimate-charge-1-cn-generic}). Furthermore, in a similar fashion, the derivative
of $\mathcal{Q}$ may be estimated as $\frac{d \mathcal{Q}}{d\mu} = 2 + \mathcal{O}(\mu^2 e^{-2\mu \ell})$ establishing
monotonicity of $\mathcal{Q}$.
\end{proof}

\section{Search for the edge-localized state with smallest energy}
\label{sec:comparing}

The ground state of the constrained minimization problem
(\ref{minimizer}) with fixed mass $q$, when it exists, has the
following properties

\begin{proposition}
  \label{prop-ground-state}
  Let $\Phi \in H^1_{\Graph}$ be the ground state of the constrained
  minimization problem (\ref{minimizer}). Then
  \begin{enumerate}
  \item $\Phi$ is real and positive (up to a non-zero factor),
  \item $\Phi \in H^2_{\Graph}$ and is a solution to the stationary
    NLS equation (\ref{statNLS}) with some $\Lambda\in\mathbb{R}$,
  \item $\Phi$ is ``non-oscillatory'' on every edge $e$ of the graph:
    the number of preimages of any value $\phi \in \mathbb{R}$ in the
    open edge $e$ does not exceed 2,
    \begin{equation}
      \label{eq:number_preimages}
      \#\{x\in e : \Phi(x) = \phi\} \leq 2.
    \end{equation}
  \end{enumerate}
\end{proposition}

\begin{proof}
  The first two properties are established in \cite[Proposition
  3.3]{AdaSerTil_cvpde15}. The last property is a consequence of the
  Polya--Szeg\"{o} inequality (see \cite[Proposition
  3.1]{AdaSerTil_cvpde15}).  Informally, if
  \eqref{eq:number_preimages} is violated, we can rearrange the
  function $\Phi \in H^1_{\Gamma}$ in a way that lowers energy while conserving the mass.

  To give full details, if \eqref{eq:number_preimages} is violated, one can
  find an interval $(\phi_-, \phi_+)$ of values with 3 or more
  preimages.  Starting with $\phi$, we can easily show this by
  considering three cases: all preimages of $\phi$ are local extrema,
  none of preimages of $\phi$ are local extrema, and at least one is a
  local extremum and at least one is not.

  Without loss of generality (and passing to a smaller interval if
  necessary) we can assume that $\Phi(u) < \phi_-$ and $\Phi(v)
  \not\in (\phi_-, \phi_+)$, where $e=(u,v)$.  There are now two cases
  to consider, $\Phi(v) > \phi_+$ and $\Phi(v) < \phi_-$, shown
  on the top left and bottom left panels of Fig. \ref{fig:rearrange} respectively.

  \begin{figure}
    \centering
    \includegraphics[scale=1]{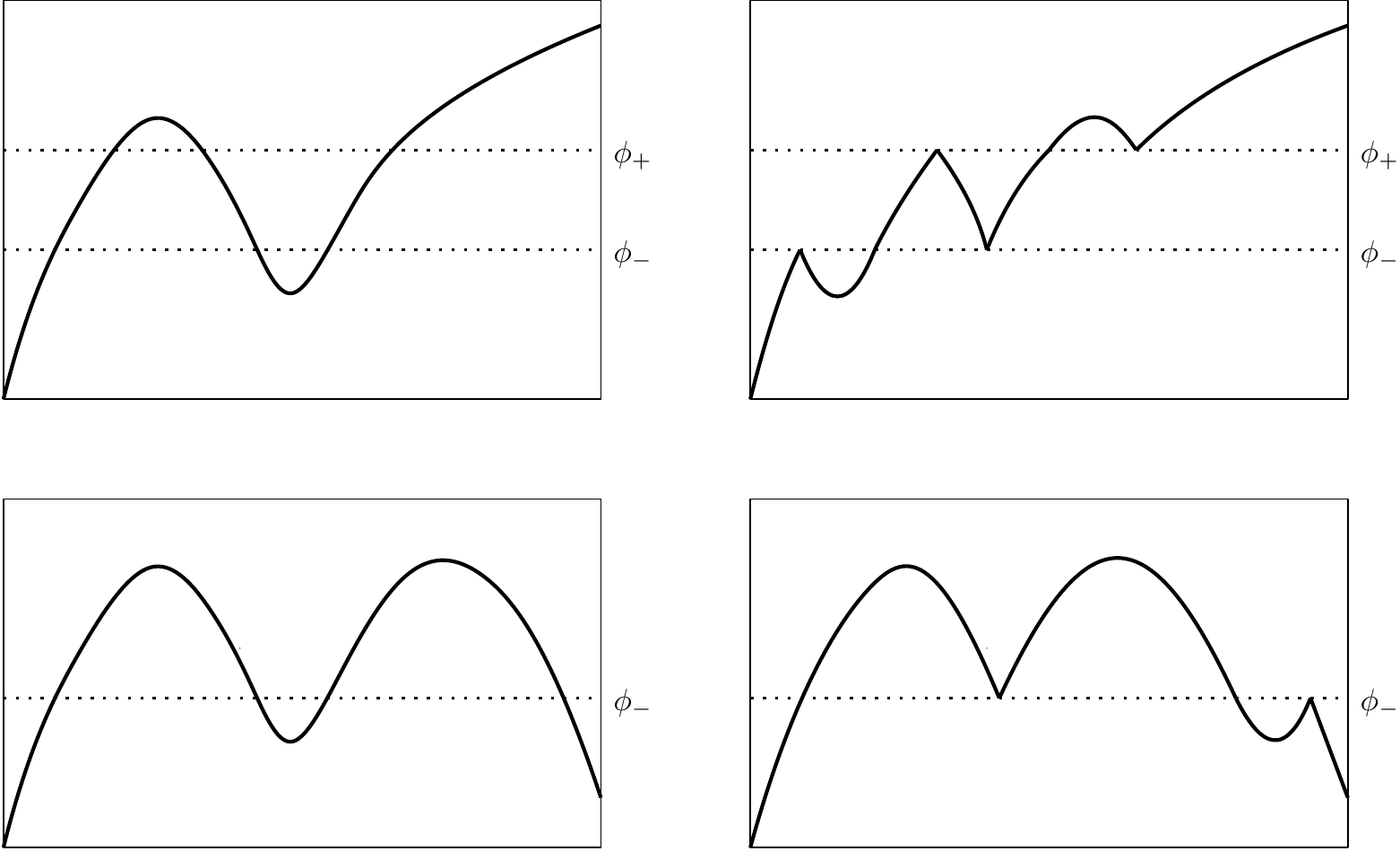}
    \caption{Two cases of the function $\Phi$ (left column) and their
      corresponding preliminary rearrangements (right column)
      described in the proof of Proposition~\ref{prop-ground-state}.}
    \label{fig:rearrange}
  \end{figure}

  In the former case, we rearrange the function $\Phi$ on the edge $e$
  by first collecting together all open intervals where $\Phi<\phi_-$
  (preserving their relative order in the edge $e$).  Then we collect
  all pre-images of the interval $(\phi_-, \phi_+)$ and, finally, all
  open intervals where $\Phi>\phi_+$, see Fig.~\ref{fig:rearrange} (top
  right). It is easy to see that the resulting function is continuous
  (in particular, its values on the vertices $u$ and $v$ are
  unchanged) and piecewise differentiable, thus an admissible
  $H^1_{\Graph}$ test function.  Applying monotone rearrangement to
  the middle part (pre-images of the interval $(\phi_-, \phi_+)$) we
  obtain an equimesurable with $\Phi$ function with a smaller
  derivative and therefore strictly smaller energy for the same mass.

  In the latter case (both $\Phi(u) < \phi_-$ and $\Phi(v) < \phi_-$),
  we retain the first interval of $\Phi^{-1}(-\infty, \phi_-)$, then
  place all preimages of $(\phi_-, \infty)$, then all remaining
  preimages of $(-\infty, \phi_-)$.  As before, we obtain an
  admissible $H^1(\Graph)$ function, see
  Fig.~\ref{fig:rearrange} (bottom right).  We apply \emph{symmetric
    rearrangement} to its middle part (preimages of
  $(\phi_-, \infty)$) which has at least 2 preimages for all values
  and at least 3 for values in $(\phi_-, \phi_+)$.  Therefore the
  energy becomes strictly smaller.
\end{proof}

The edge-localized states constructed in Theorems
\ref{thm:localized_pendant}, \ref{thm:localized_loop}, and
\ref{thm:generic_edge} satisfy the properties listed in Proposition~\ref{prop-ground-state} and are therefore good
candidates for the ground state. Moreover, by an application of Lemma~\ref{lem:suppressed_solution},
each edge-localized state has a unique maximum, which is also
suggestive of a ground state. (There is currently no proof
that a ground state must have unique maximum on the whole graph.)

It is therefore relevant to ask what characterizes an edge that
would support a localized solution $\Phi \in H^2_{\Graph}$ with the
smallest energy $E_q$ for a given mass $q$.  This is not
straightforward since we have expressions for both energy and mass
as functions of the Lagrange multiplier $\Lambda = -\mu^2$, which
now needs to be eliminated.  Additionally, in many cases the
asymptotic representations of $\mathcal{Q}(\mu)$ and
$\mathcal{E}(\mu)$ differ only in the exponentially-small correction
to the leadign term.

Section~\ref{sec:comparison_lemma} provides a tool that will enable
comparison of the energy $\mathcal{E} = \mathcal{E}(\Phi)$ for a given
mass based on comparison of mass and energy for given Lagrange
multiplier $\Lambda$.
Section~\ref{sec:comparison_results} applies
the tool to distinguish between the pendant, looping, and internal
edges of different lengths and to provide the proof of
Theorem~\ref{thm:main_modest2} in the case of bounded
graphs. Corollary~\ref{cor:main_modest3} is proven in
Section~\ref{sec:corollary} for the case of unbounded
graphs. Section~\ref{sec:ground_state_heuristics} discusses
the relevance of the edge-localized states in the search for the ground state.

\subsection{Comparison lemma}
\label{sec:comparison_lemma}

Let $\Phi \in H^2_{\Graph}$ be a solution to the stationary NLS equation (\ref{statNLS}) with the Lagrange multiplier
$\Lambda$ and define $\mathcal{Q} := \mathcal{Q}(\Phi)$ and
$\mathcal{E} := \mathcal{E}(\Phi)$ by (\ref{energy}).
Since $\Phi$ is a critical point of the augmented energy $S_{\Lambda}(U) := \mathcal{E}(U) - \Lambda \mathcal{Q}(U)$,
it follows that that $\mathcal{Q}$ and $\mathcal{E}$ satisfy the differential equation
\begin{equation}
\label{slope-E-versus-Lambda}
\frac{d \mathcal{E}}{d \Lambda} = \Lambda \frac{d \mathcal{Q}}{d \Lambda},
\end{equation}
provided they are $C^1$ in $\Lambda$. The following comparison lemma is deduced from analysis of the differential equation
(\ref{slope-E-versus-Lambda}).

\begin{lemma}
  \label{lem:comparison}
  Assume that there are two solution branches with the $C^1$ maps $\Lambda \mapsto \mathcal{Q}_{1,2}(\Lambda)$ and
  $\Lambda \mapsto \mathcal{E}_{1,2}(\Lambda)$, where $\Lambda \in (-\infty,\Lambda_0)$ for some $\Lambda_0 < 0$, satisfying
  \begin{align}
    \label{eq:Qcomparison}
    \lim_{\Lambda \to -\infty} |\Lambda| \left| \mathcal{Q}_2(\Lambda) - \mathcal{Q}_1(\Lambda) \right| = 0
  \end{align}
  and
    \begin{align}
    \label{eq:Ecomparison}
    \lim_{\Lambda \to -\infty} \left| \mathcal{E}_2(\Lambda) - \mathcal{E}_1(\Lambda) \right| = 0,
    \end{align}
  If $\mathcal{Q}_1(\Lambda) < \mathcal{Q}_2(\Lambda)$ for every $\Lambda \in (-\infty,\Lambda_0)$, then $\mathcal{E}_1(\Lambda) > \mathcal{E}_2(\Lambda)$.
  If, additionally, $\mathcal{Q}_{1,2}$ are decreasing on $(-\infty,\Lambda_0)$ and
  the values $\Lambda_1, \Lambda_2 \in (-\infty, \Lambda_0)$ are such that $\mathcal{Q}_2(\Lambda_2) = \mathcal{Q}_1(\Lambda_1) = q$, then
  $\mathcal{E}_1(\Lambda_1) > \mathcal{E}_2(\Lambda_2)$.
\end{lemma}

\begin{proof}
  Integrating equation (\ref{slope-E-versus-Lambda}) by parts we get
  \begin{align}
  \nonumber
    \mathcal{E}_1(\Lambda) - \mathcal{E}_2(\Lambda)
    & =  \int_{-\infty}^{\Lambda} \frac{d}{d s} \left[ \mathcal{E}_1(s) - \mathcal{E}_2(s) \right] d s \\
    \nonumber
    & =  \int_{-\infty}^{\Lambda} s \frac{d}{d s} \left[ \mathcal{Q}_1(s) - \mathcal{Q}_2(s) \right] d s \\
    & =  \Lambda \left[ \mathcal{Q}_1(\Lambda) - \mathcal{Q}_2(\Lambda) \right]
      - \int_{-\infty}^{\Lambda} \left[ \mathcal{Q}_1(s) - \mathcal{Q}_2(s) \right] ds,
    \label{eq:E12_diff}
  \end{align}
  where the boundary terms as $\Lambda \to -\infty$ vanish due to
  (\ref{eq:Qcomparison}) and (\ref{eq:Ecomparison}).  If
  $\mathcal{Q}_1(\Lambda) < \mathcal{Q}_2(\Lambda)$ for every
  $\Lambda \in (-\infty,\Lambda_0)$ with negative $\Lambda_0$, then
  the right-hand side of (\ref{eq:E12_diff}) is strictly positive and
  $\mathcal{E}_1(\Lambda) > \mathcal{E}_2(\Lambda)$ for every
  $\Lambda \in (-\infty,\Lambda_0)$.

 In order to prove the second assertion, we observe the following.
 It follows from $\mathcal{Q}_1(\Lambda) < \mathcal{Q}_2(\Lambda)$ for every
  $\Lambda \in (-\infty,\Lambda_0)$ that if $\mathcal{Q}_2(\Lambda_2) = \mathcal{Q}_1(\Lambda_1) = q$,
  then $\Lambda_1 < \Lambda_2$, see Fig.~\ref{figQ}.  We can now expand
  \begin{equation*}
    \mathcal{E}_1(\Lambda_1) - \mathcal{E}_2(\Lambda_2)
    = \mathcal{E}_1(\Lambda_1) - \mathcal{E}_2(\Lambda_1) + \mathcal{E}_2(\Lambda_1) - \mathcal{E}_2(\Lambda_2).
  \end{equation*}
  Using (\ref{eq:E12_diff}) we can estimate
  \begin{equation*}
    \mathcal{E}_1(\Lambda_1) - \mathcal{E}_2(\Lambda_1)
    > \Lambda_1 \left[ \mathcal{Q}_1(\Lambda_1) - \mathcal{Q}_2(\Lambda_1) \right].
  \end{equation*}
  We also have
  \begin{align*}
    \mathcal{E}_2(\Lambda_1) - \mathcal{E}_2(\Lambda_2)
    &= -\int_{\Lambda_1}^{\Lambda_2} \frac{d\mathcal{E}_2}{d s} ds \\
    &= -\int_{\Lambda_1}^{\Lambda_2} s \frac{d \mathcal{Q}_2}{d s} d s \\
    &= \Lambda_1 \mathcal{Q}_2(\Lambda_1) - \Lambda_2 \mathcal{Q}_2(\Lambda_2) +
      \int_{\Lambda_1}^{\Lambda_2} \mathcal{Q}_2(s) ds.
  \end{align*}
  Combining the two expressions and denoting $\mathcal{Q}_2(\Lambda_2) =
  \mathcal{Q}_1(\Lambda_1) = q$ we get
  \begin{align*}
    \mathcal{E}_1(\Lambda_1) - \mathcal{E}_2(\Lambda_2)
    &> \Lambda_1 \mathcal{Q}_1(\Lambda_1) - \Lambda_2 \mathcal{Q}_2(\Lambda_2)
      + \int_{\Lambda_1}^{\Lambda_2} \mathcal{Q}_2(s) d s \\
    &= -(\Lambda_2-\Lambda_1) q + \int_{\Lambda_1}^{\Lambda_2} \mathcal{Q}_2(s) d s \\
    &= \int_{\Lambda_1}^{\Lambda_2}  (\mathcal{Q}_2(s)-q) d s.
  \end{align*}
  Since $\mathcal{Q}_2$ is decreasing on $(-\infty,\Lambda_0)$, we have $\mathcal{Q}_2(\Lambda) \geq \mathcal{Q}_2(\Lambda_2) = q$
  for every $\Lambda \in (\Lambda_1,\Lambda_2)$, see Fig.~\ref{figQ}, so that the previous expression implies that
  $\mathcal{E}_1(\Lambda_1) - \mathcal{E}_2(\Lambda_2) > 0$.
\end{proof}

\begin{figure}
\begin{center}
  \includegraphics[width=12cm]{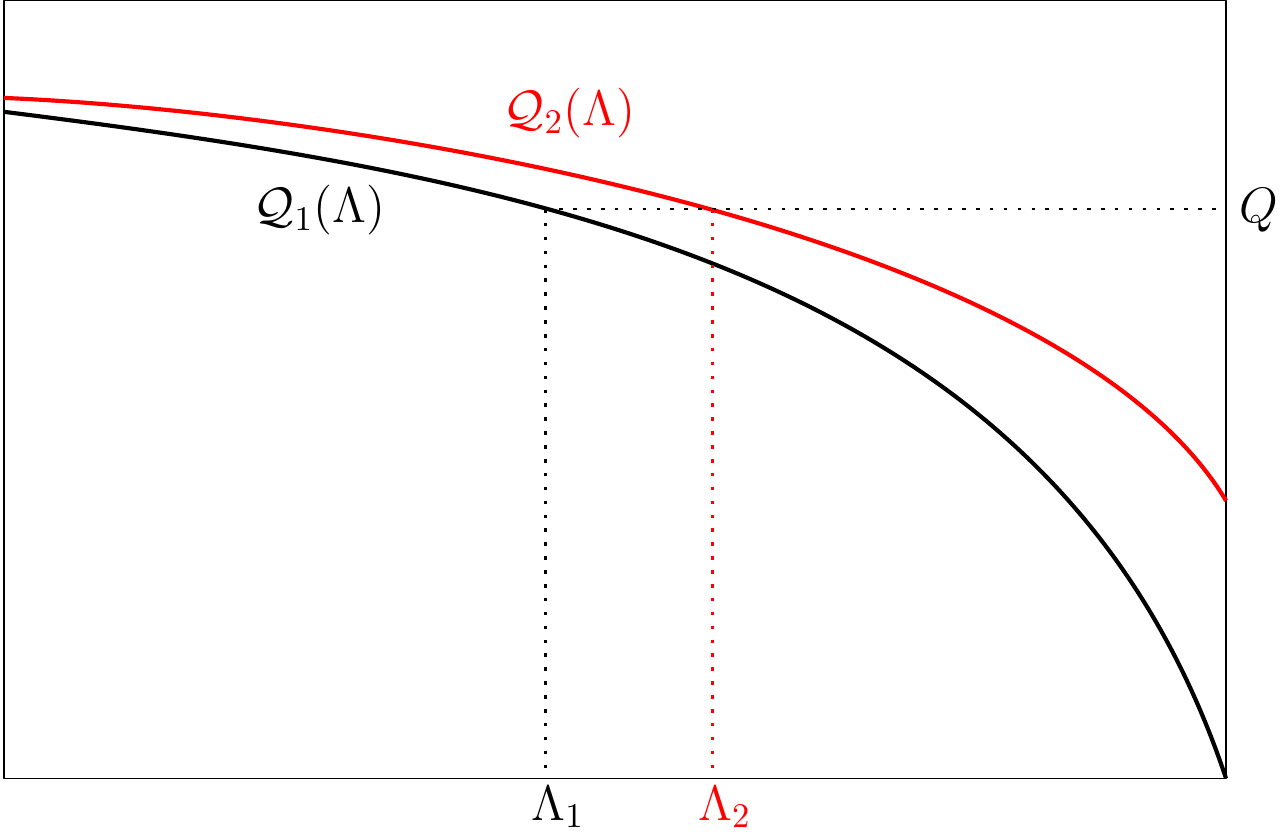}
\end{center}
\caption{Schematic diagram of the maps $\Lambda \mapsto \mathcal{Q}_{1,2}$ in Lemma \ref{lem:comparison}.}
\label{figQ}
\end{figure}

\begin{remark}
Lemma \ref{lem:comparison} presents a surprising fact that if two $C^1$ monotonic maps $\Lambda \mapsto \mathcal{Q}_{1,2}$
for the two branches of the edge-localized states converge to each other, then
the stationary state with the minimal $\mathcal{Q}$ for fixed (large negative) Lagrange multiplier $\Lambda$
corresponds to the maximal $\mathcal{E}$ for fixed (large positive) mass $q$.
Because of a trivial sign error, the swap between the two branches of stationary solutions on the $(\Lambda,\mathcal{Q})$
and $(\mathcal{Q},\mathcal{E})$ diagrams was overlooked in \cite{MarPel_amrx16} for the particular case of the dumbbell graph.
\end{remark}

\subsection{Proof of Theorem~\ref{thm:main_modest2}}
\label{sec:comparison_results}

Consider a compact graph $\Graph$, i.e.~a graph with finitely
many edges, each of finite length.  We will deduce which edge of the
graph $\Graph$ gives an edge-localized state of smallest energy
$\mathcal{E}$ for a given (large) mass $\mathcal{Q}$ providing the
proof of Theorem~\ref{thm:main_modest2}.

The first comparison is between a pendant and a non-pendant (looping or internal) edge.
For a pendant edge, Theorem \ref{thm:localized_pendant} gives to the leading term,
\begin{equation}
  \label{eq:QE_pendant}
  \mathcal{Q} \sim \mu, \qquad \mathcal{E} \sim -\frac13 \mu^3, \quad \Rightarrow \quad \mathcal{E}_p \sim -\frac{1}{3} \mathcal{Q}^3,
\end{equation}
whereas for a non-pendant edge, Theorems~\ref{thm:localized_loop} and
\ref{thm:generic_edge} give to the leading term
\begin{equation}
  \label{eq:QE_nonpendant}
  \mathcal{Q} \sim 2\mu, \qquad \mathcal{E} \sim -\frac23 \mu^3, \quad
  \Rightarrow \quad \mathcal{E}_{np} \sim -\frac{1}{12} \mathcal{Q}^3
  > \mathcal{E}_p.
\end{equation}
Therefore, \emph{for a given mass $\mathcal{Q}$, localizing on a pendant edge of
  any length is preferable to localizing on a non-pendant edge}.

For comparing similar edges, we can apply Lemma~\ref{lem:comparison}
where the assumptions (\ref{eq:Qcomparison}), (\ref{eq:Ecomparison}), and
the monotonicity of the map $\Lambda \mapsto \mathcal{Q}$ have been
verified in Theorems \ref{thm:localized_pendant}, \ref{thm:localized_loop}, and
\ref{thm:generic_edge}.

Comparing two pendant edges via
equation~\eqref{eq:mass_answer_pendant}, we see from the
exponentially small term that \emph{the state localized on a longer
  pendant edge} has larger mass $\mathcal{Q}$ for fixed $\mu$.  Hence,
by Lemma \ref{lem:comparison}, it has smaller $\mathcal{E}$ at fixed
mass $\mathcal{Q}$.  If two edges have the same length, \emph{the
  pendant edge incident to fewer edges is more energetically optimal}.

Comparing two non-pendant edges via
equations~\eqref{eq:mass_answer_loop} and
\eqref{eq:mass_answer_generic} we see that \emph{the looping edges
  incident to $N=1$ or $2$ edges} are energetically favorable since
$\mathcal{Q} > 2 \mu$ for $N = 1$ or $\mathcal{Q} \approx 2 \mu$ for
$N = 2$, whereas $\mathcal{Q} < 2\mu$ for a looping edge with
$N \geq 3$ or an internal edge. Moreover, \emph{the shorter looping
  edge with $N = 1$} has smaller energy $\mathcal{E}$ at fixed mass
$\mathcal{Q}$.  No conclusion on the lengths can be drawn for the
looping edge with $N = 2$ unless the higher-order exponentially small
correction is computed and analyzed.

For the looping edge incident to $N \geq 3$ edges and for the internal edges,
we can see from equations~\eqref{eq:mass_answer_loop} and
\eqref{eq:mass_answer_generic} that the length of the
edge is the primary factor (\emph{the longer the edge, the lower the
energy}).  To break a tie in the case of two edges of the same length
\emph{the energy is lowest on the edge with the smaller}
\begin{equation}
  \label{eq:N_factors}
  \frac{N-2}{N+2} \qquad \mbox{or} \qquad \sqrt{\frac{N_--1}{N_-+1}}
    \sqrt{\frac{N_+-1}{N_++1}}.
\end{equation}

Combining all comparisons together provides the proof of
Theorem~\ref{thm:main_modest2}.

\subsection{Proof of Corollary~\ref{cor:main_modest3}}
\label{sec:corollary}

Consider an unbounded graph $\Graph$ with finitely many edges and finitely many vertices
such that at least one edge as a half-line.  By \cite[Corollary 3.4]{AdaSerTil_jfa16},
if there exists a stationary state with energy $\mathcal{E}$ satisfying
\begin{equation}
\label{criterion-ground-state}
\mathcal{E} \leq -\frac{1}{12} \mathcal{Q}^3
\end{equation}
for a given mass, then there exists a ground state in the constrained minimization problem (\ref{minimizer}).
The energy level $\mathcal{E} = -\frac{1}{12} \mathcal{Q}^3$ is the energy of the NLS soliton (\ref{eq:sech_soliton})
after scaling (\ref{scaling}) on the infinite line.

In the case of a pendant, the criterion (\ref{criterion-ground-state}) is always satisfied thanks
to the estimate (\ref{eq:QE_pendant}), in agreement with \cite[Proposition 4.1]{AdaSerTil_jfa16}.
If no pendant edges are present in the graph $\Graph$,
the criterion (\ref{criterion-ground-state}) can be restated with the help of the comparison lemma
(Lemma \ref{lem:comparison}) as follows.
If there exists an edge-localized state with mass $\mathcal{Q}$ satisfying
\begin{equation}
\label{criterion-ground-state-Q}
\mathcal{Q} \geq 2\mu
\end{equation}
for a given Lagrange multiplier $\Lambda = -\mu^2$, then there exists
a ground state on the graph $\Graph$.

Thanks to the estimate~\eqref{eq:mass_answer_loop} and
\eqref{eq:mass_answer_generic}, the criterion
(\ref{criterion-ground-state-Q}) is satisfied for the edge-localized state on the looping edge
incident to $N=1$ edge since $\mathcal{Q} > 2 \mu$ and is definitely
not satisfied for the edge-localized states on the looping edge incident to $N \geq 3$ edges or on
the internal edge since $\mathcal{Q} < 2 \mu$.  The case of the
looping edge incident to $N = 2$ edges is borderline since
$\mathcal{Q} \approx 2 \mu$ and no conclusion can be drawn without
further estimates.

Comparison between masses of edge-localized states on the pendant edges
of different lengths or on the looping edges incident to $N = 1$ edge
of different lengths is the same as in the case of bounded graphs.

\subsection{Searching for the ground state}
\label{sec:ground_state_heuristics}

In this section we outline some heuristic arguments why the
edge-localized states should be the only candidates for the ground
state.  Making these arguments mathematically precise remains a
challenging open question of high priority.

The condition of being positive and non-oscillatory (see
Proposition~\ref{prop-ground-state}) on a long edge of the rescaled
graph $\Graph_\mu$ imposes restrictions on the solutions $\Psi$ to
equation~\eqref{eq:stationaryNLSstandard}.  On every edge $e$ they
must be either identically constant or to be close to a portion of the
shifted $\sech$-solution \eqref{eq:sech_soliton} on every interval
between a local maximum and a local minimum on the edge $e$.

Considering only the second possibility, and ignoring all parts of the
graph where the solution falls below the small value $p_0$ from
Theorem~\ref{thm:nlin_DTN}, we can break the solution into $M$
portions of half-$\sech$ solutions.  The mass and energy of these
portions are given at the leading-order by
\begin{equation}
  \label{eq:QE_vertex_general-0}
  \mathcal{Q} \sim M \mu, \qquad
  \mathcal{E} \sim -\frac{M}{3} \mu^3,
\end{equation}
so that
\begin{equation}
  \label{eq:QE_vertex_general}
  \mathcal{E}_g \sim -\frac{1}{3 M^2}  \mathcal{Q}^3.
\end{equation}
We conclude that any value of $M$ above $2$ results in a worse energy
than that achievable by any edge-localized state, see
equations~\eqref{eq:QE_pendant} and \eqref{eq:QE_nonpendant}.  But the
value $M=1$ is only possible when the maximum is achieved on a vertex
of degree 1 (i.e.\ a pendant edge) and the value $M=2$ corresponds to
a single point of maximum.  This means that the solution localizes on a single edge.
(A solution localizing on \emph{two} pendant edges  would also result in $M=2$, but
they have larger energy compared to the solution localizing on one of the pendants.)

\section{Numerical examples}
\label{sec:examples}

We now discuss in detail the graph $\Gamma$ of
Fig.~\ref{db_Kgeq2_sv_fig} and its edge-localized states. The graph
$\Gamma$ consists of two identical side loops and three identical
internal edges connected at a single vertex of each loop.  This
corresponds to the case $(iv)$ of Theorem~\ref{thm:main_modest2}.

Let us generalize the graph $\Gamma$ with two side loops and $K$ internal edges.
For any stationary state $\Phi \in H^2_{\Gamma}$ of the stationary NLS equation (\ref{statNLS})
centered on one of the $K$ internal edges, Theorem \ref{thm:generic_edge} with
$N_- = N_+ = K + 1$ implies that the mass integral ${\mathcal Q}_{\rm int} := Q(\Phi)$
in (\ref{energy}) is expanded asymptotically as $\mu \to \infty$ in the form:
\begin{equation}
\label{Q-cn-dumbbell-1}
\mathcal{Q}_{\rm int} = 2 \mu - \frac{16 K}{K+2} \mu^2 \ell_0 e^{-2 \mu \ell_0}
+ \bigO{\mu e^{-2 \mu \ell_0}},
\end{equation}
where $\mu := |\Lambda|^{1/2}$ and $\ell_0$ is the half-length of each internal edge.
On the other hand, for any stationary state $\Phi \in H^2_{\Gamma}$ centered at one of the two side loops,
Theorem \ref{thm:localized_loop} with $N = K$ implies that
the mass integral ${\mathcal Q}_{\rm loop} := Q(\Phi)$ in (\ref{energy}) is expanded asymptotically as $\mu \to \infty$
in the form:
\begin{equation}
\label{Q-cn-dumbbell-2}
\mathcal{Q}_{\rm loop} = 2 \mu - \frac{16 (K-2)}{K+2} \mu^2 \ell_* e^{-2 \mu \ell_*}
+ \bigO{\mu e^{-2 \mu \ell_*}},
\end{equation}
where $\ell_*$ is the half-length of the side loop.

\begin{figure}
\begin{tabular}{cc}
\includegraphics[width=6cm]{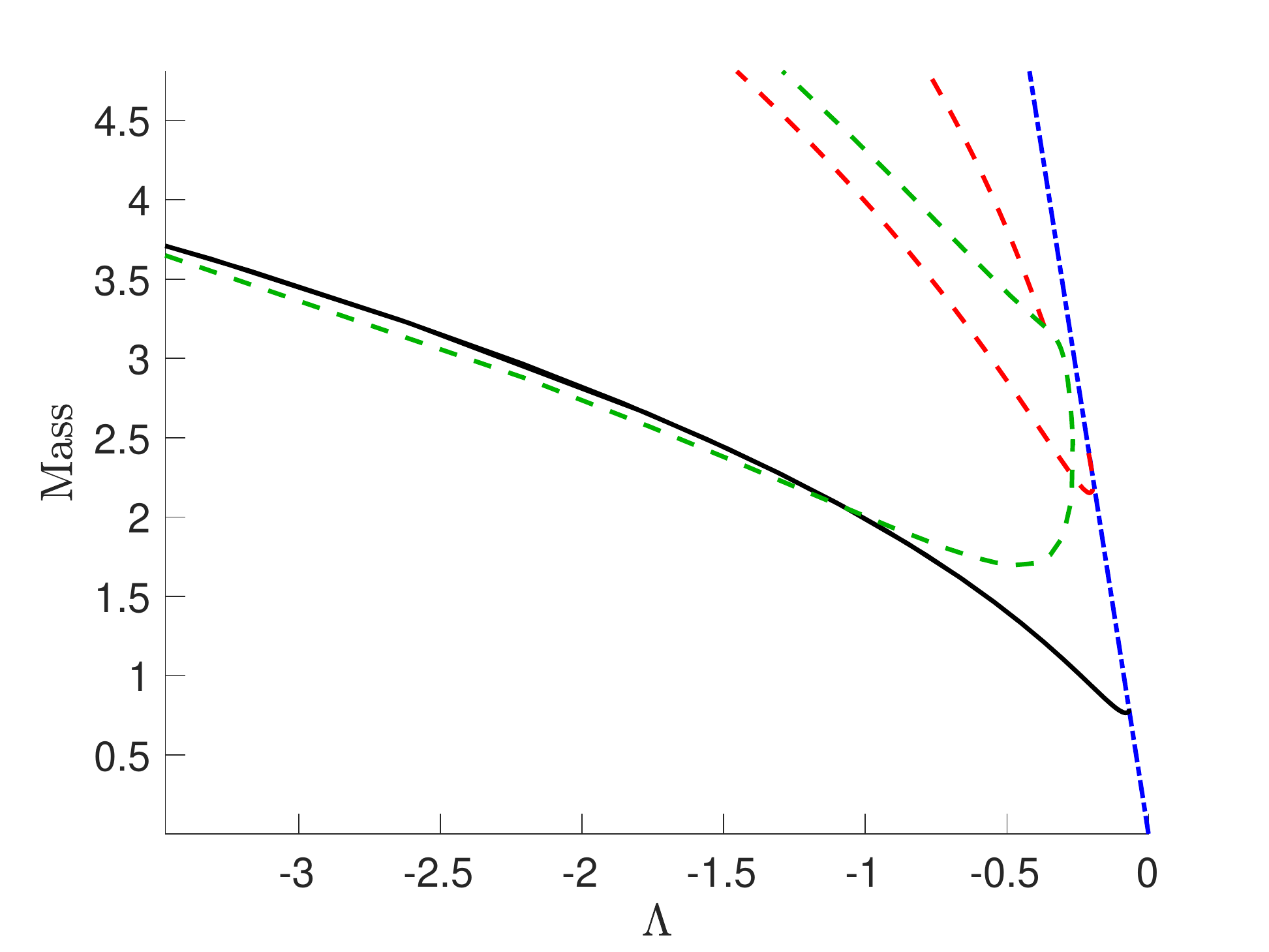} &
\includegraphics[width=6cm]{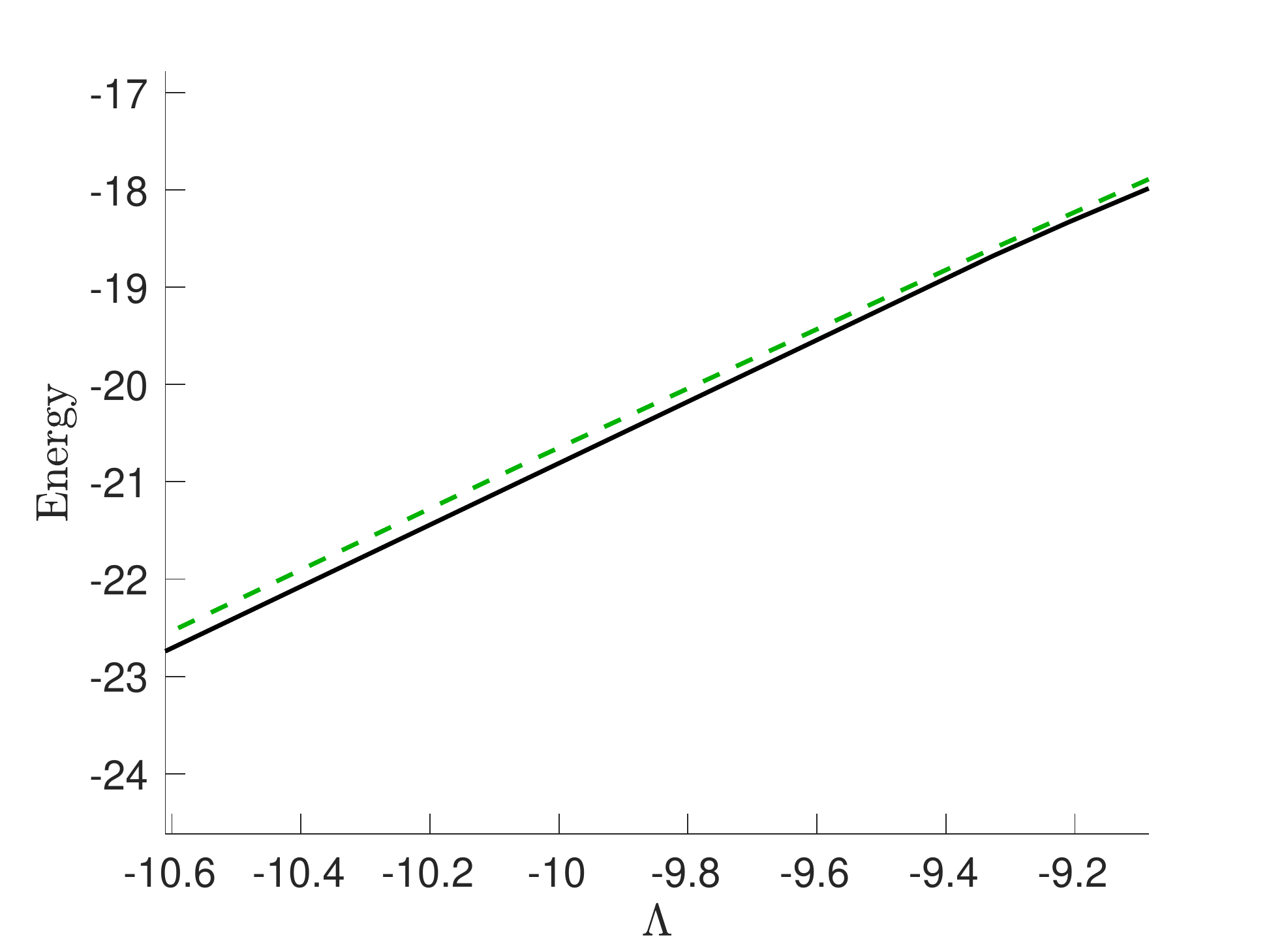} \\
\includegraphics[width=6cm]{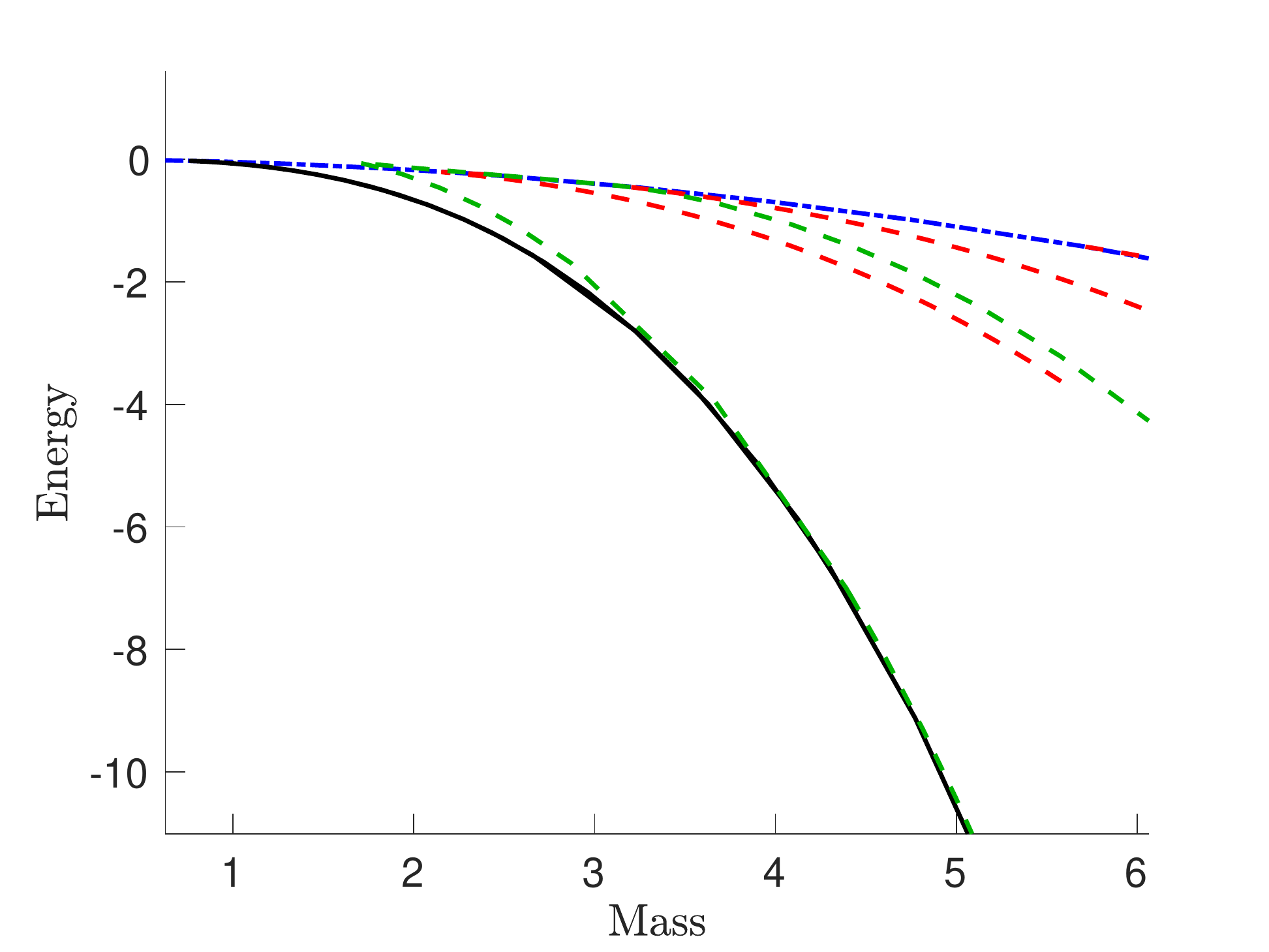} &
\includegraphics[width=6cm]{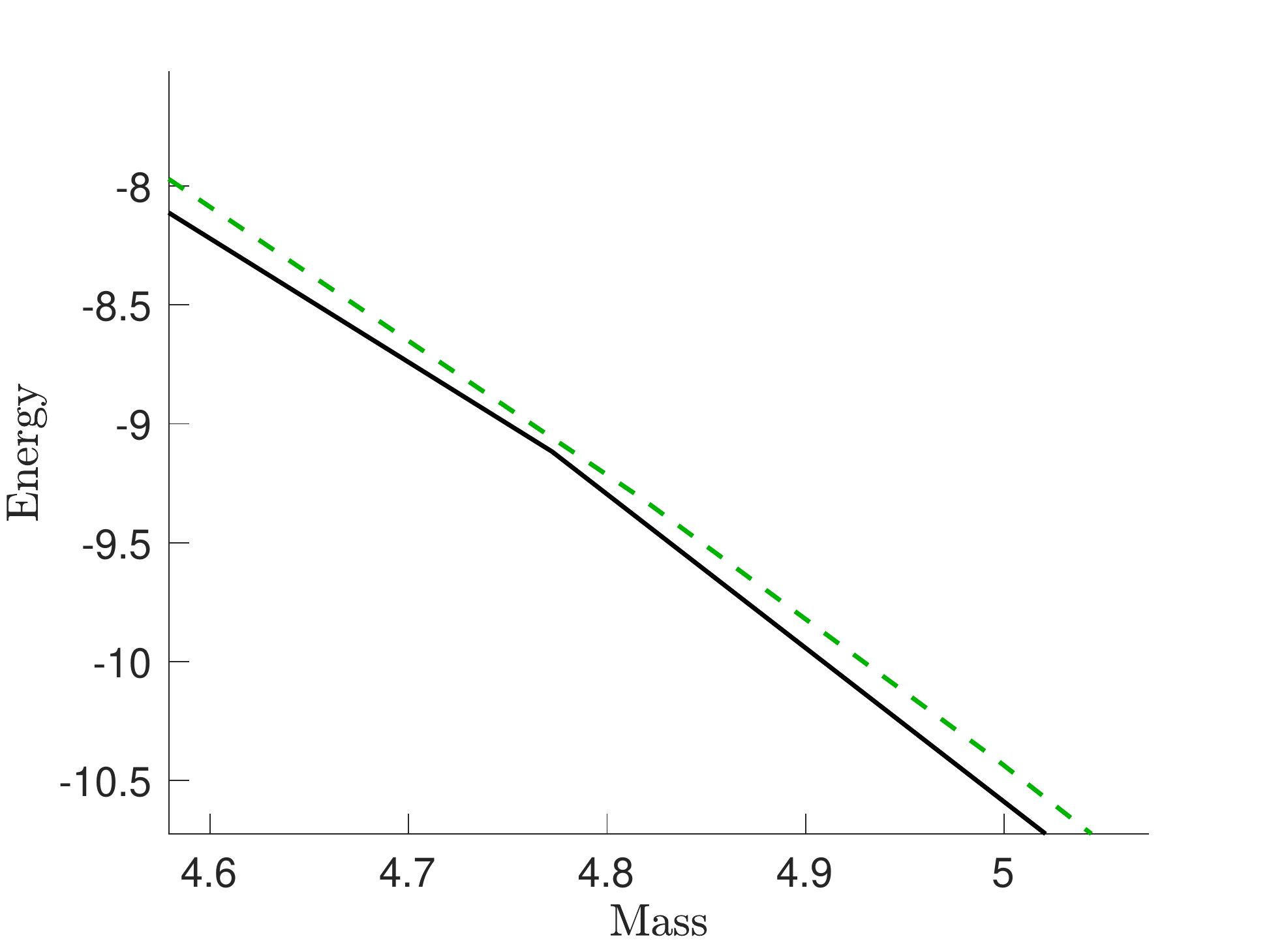} \\
\end{tabular}
\caption{Stationary states for the dumbbell graph with two loops and three internal edges
shown on Fig.~\ref{db_Kgeq2_sv_fig}, where the internal edges are shorter relative to the loops
(edge length is $\pi$ and loop lengths is $2\pi$).
We plot (top left) Mass $\mathcal{Q}$ vs $\Lambda$; (top right) Energy $\mathcal{E}$ vs $\Lambda$;
(bottom left) $\mathcal{E}$ versus $\mathcal{Q}$; and (bottom right)
the blow-up of the previous graph around mass $4.8$.
The black (---) (color online) line shows the loop-centered state,
the green dashed ({\color{green}$- -$}) (color online) line shows the edge-centered state,
and the dashed blue {\color{blue} --} solid line (color online) shows the constant
state from which the loop-centered state bifurcates.  The dashed red ({\color{red} $- -$})
lines show the state bifurcating off the constant state along the second eigenfunction
(concentrated on two edges) and undergoing a pitchfork bifurcation as in \cite{G19}.}
\label{db_Kgeq2_sv_fig3}
\end{figure}

\begin{figure}
\begin{tabular}{cc}
\includegraphics[width=6cm]{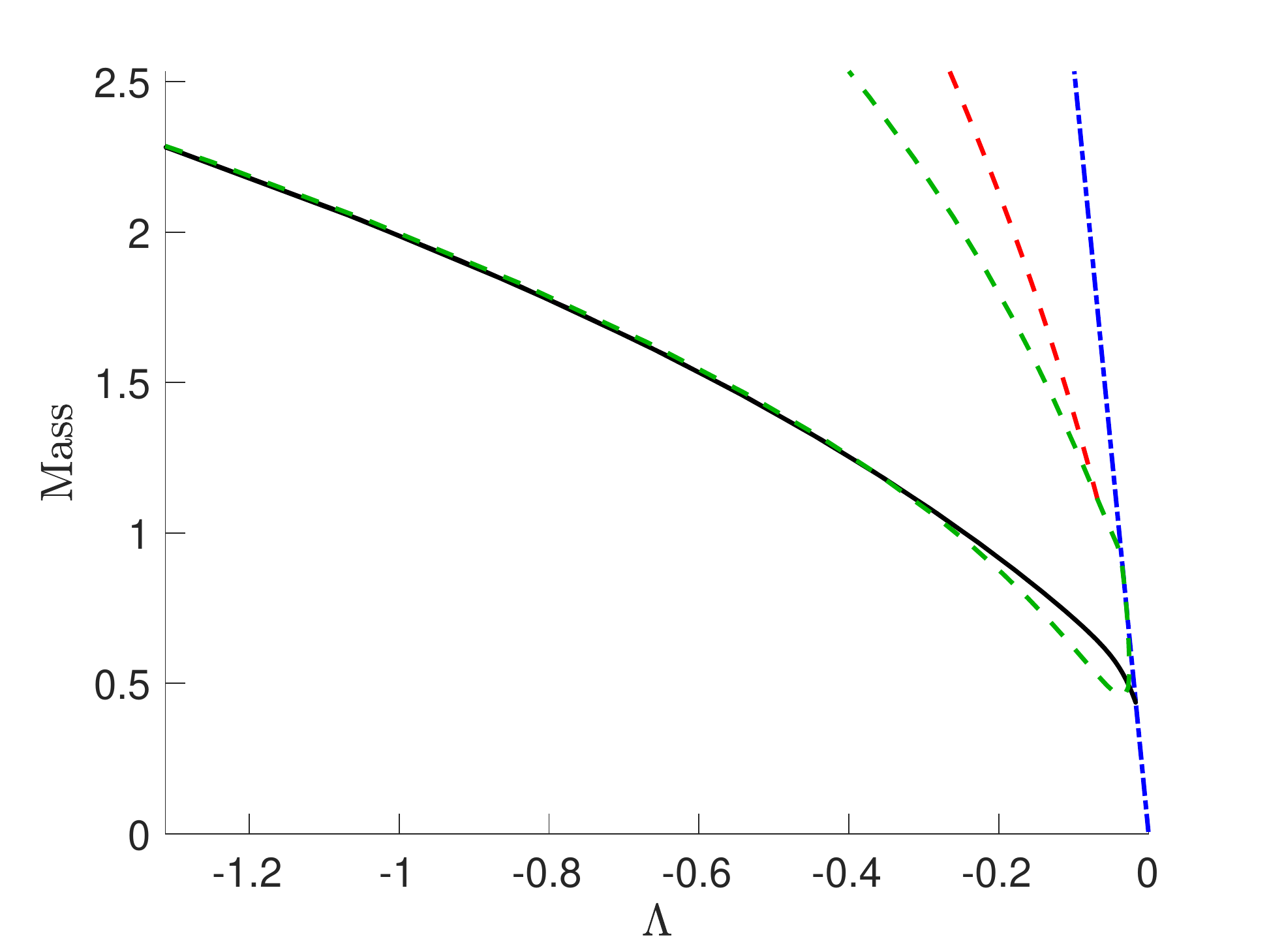} &
\includegraphics[width=6cm]{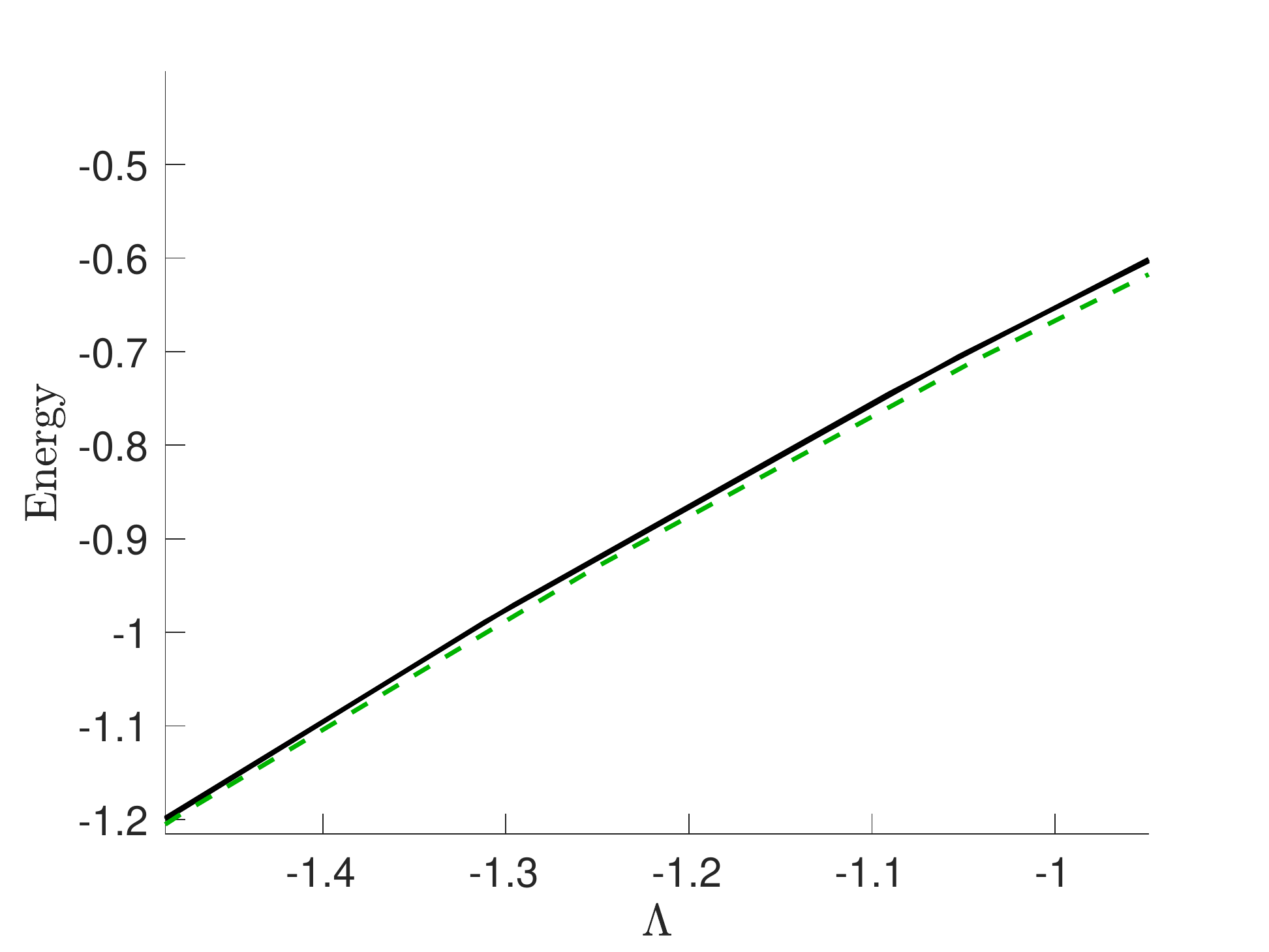} \\
\includegraphics[width=6cm]{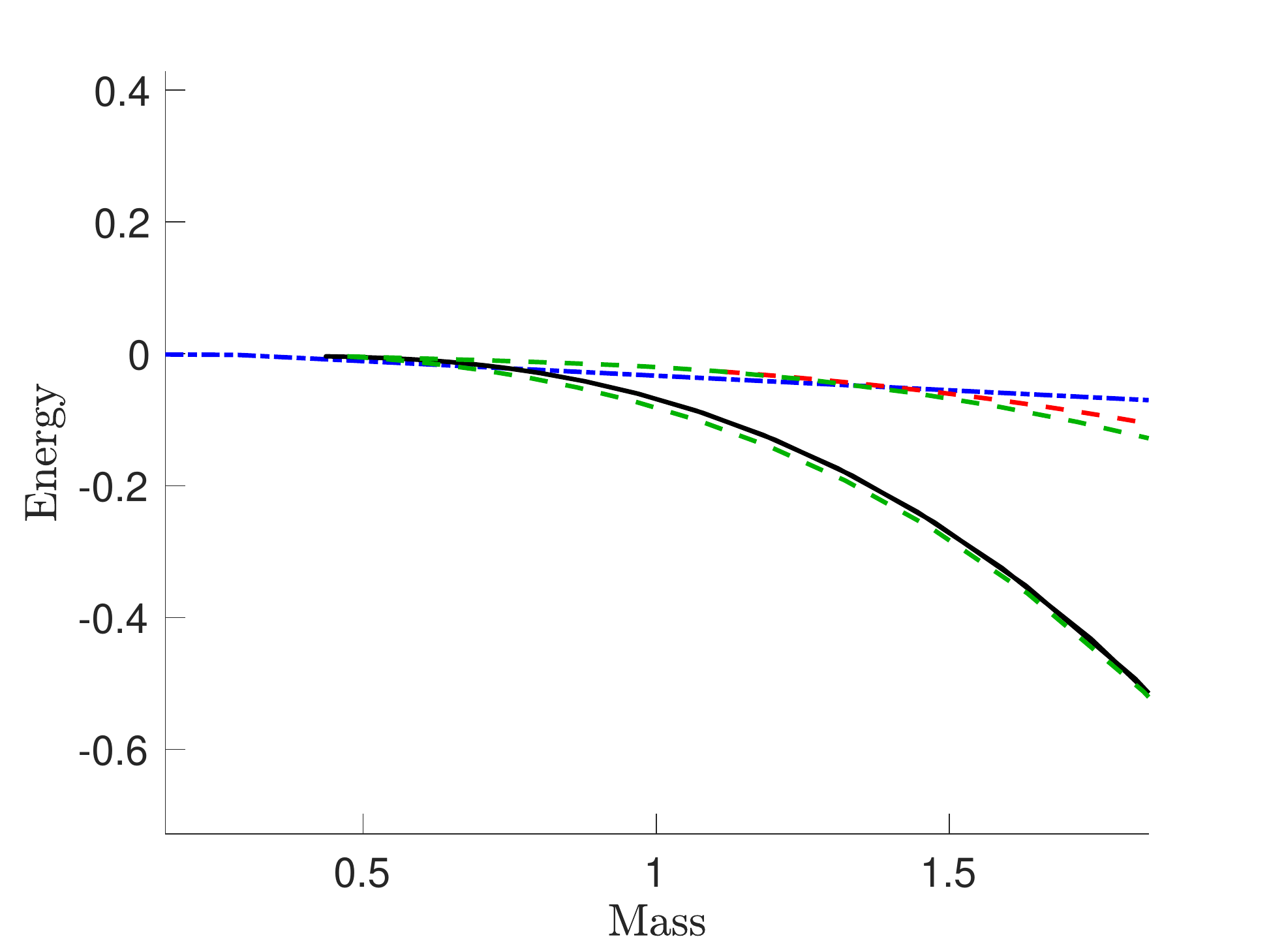} &
\includegraphics[width=6cm]{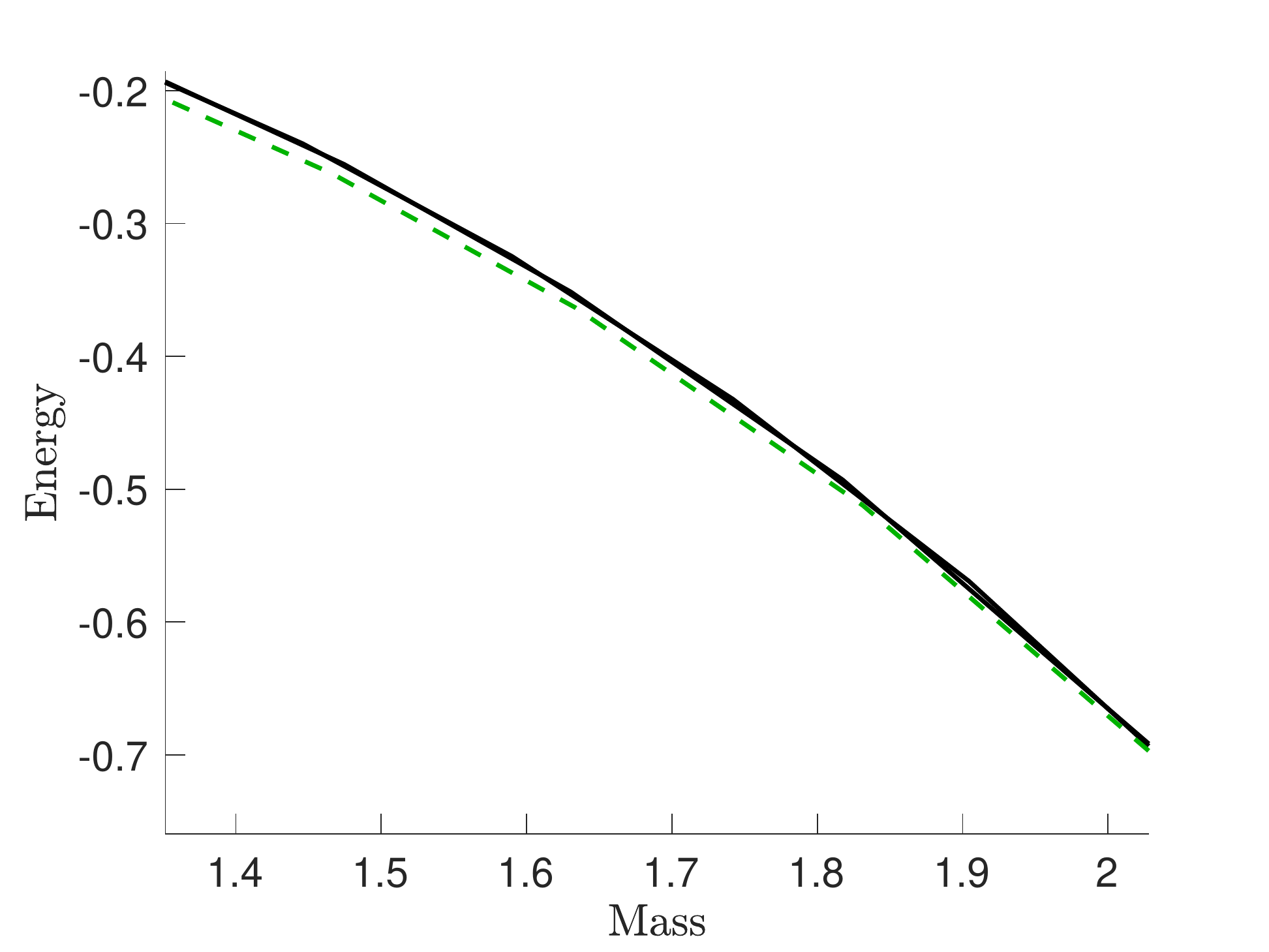} \\
\end{tabular}
\caption{The same as Fig. \ref{db_Kgeq2_sv_fig3} but such that the three internal edges are now longer relative to the loops
(edge length is $4\pi$ and loop lengths is $2\pi$).
A switch has occurred so that the edge-centered state has smaller energy
in the large mass limit.  The dashed red ({\color{red} $- -$}) line shows the state undertaking the pitchfork bifurcation
with the edge-centered state.  In this case, the edge-centered state bifurcates directly from the constant state.}
\label{db_Kgeq2_sv_fig4}
\end{figure}

It is clear from comparing (\ref{Q-cn-dumbbell-1}) and (\ref{Q-cn-dumbbell-2}) with $K = 3$ that the lengths $\ell_0$ and $\ell_*$ determine
the state with smaller energy. The loop-centered state has smaller energy for large mass when the the internal edges are short
relative to the loops and the internal edge-centered state has smaller energy for large mass
when the the internal edges are long relative to the loops.
The former case is illustrated in Fig.~\ref{db_Kgeq2_sv_fig3} (with $2
\ell_0 = \pi$ and $2 \ell_* = 2 \pi$) and the latter case is shown in 
Fig.~\ref{db_Kgeq2_sv_fig4} (with $2 \ell_0 = 4 \pi$ and $2 \ell_* = 2 \pi$).

In addition to plotting the branches of loop-localized and
edge-localized states we also show the branches of other states bifurcating off the constant solution.
The edge-localized states were found for large mass by using Petviashvili's method, see \cite{pelinovsky2004convergence}
and \cite{olson2016petviashvilli}, then continued to small mass.  The constant solution and its bifurcations
were constructed by using an arclength parametrization, see \cite{G19} based on \cite{nayfeh2008applied}.
In both cases, the constant state is the ground state for small mass \cite{CDS} which undertakes two
bifurcations considered in \cite{MarPel_amrx16} and \cite{G19}. After the first bifurcation,
the loop-centered state becomes the state with smaller energy and it remains such for every larger mass
if the loop is long relative to the internal edge (Fig. \ref{db_Kgeq2_sv_fig3}). On the other hand,
for long internal edges relative to the loops, the edge-centered state has the smaller energy
for very large mass (Fig. \ref{db_Kgeq2_sv_fig4}).

Figures~\ref{db_Kgeq2_sv_fig}, \ref{db_Kgeq2_sv_fig3} and \ref{db_Kgeq2_sv_fig4} all solve
the stationary NLS equation (\ref{statNLS}) approximated numerically using the quantum graphs
software package by R. Goodman \cite{G_QG}.

\subsection{Other examples of dumbbell graphs}

The case on one internal edge corresponds to the canonical dumbbell graph considered in \cite{MarPel_amrx16} and \cite{G19}.
It corresponds to the case $(ii)$ of Theorem~\ref{thm:main_modest2}.
It follows from (\ref{Q-cn-dumbbell-1}) and (\ref{Q-cn-dumbbell-2}) with $K = 1$ that $Q_{\rm int} < 2 |\Lambda|^{1/2} < Q_{\rm loop}$.
By the comparison lemma (Lemma \ref{lem:comparison}), the loop-centered state has a smaller energy $\mathcal{E}$
at a fixed large mass $\mathcal{Q}$ independently of lengths of the edges and loops. We reiterate here that the opposite incorrect conclusion
was reported in \cite{MarPel_amrx16} because of a trivial sign error, however, the fact that the edge-localized
state cannot be the ground state for the dumbbell graph can be shown with the technique of energy-decreasing symmetric
rearrangements from \cite{AdaSerTil_cvpde15}.  

The same conclusion holds for the dumbbell graph with two internal edges, since
expansions (\ref{Q-cn-dumbbell-1}) and (\ref{Q-cn-dumbbell-2}) with $K = 2$ imply
$Q_{\rm int} < 2 |\Lambda|^{1/2} \approx Q_{\rm loop}$. This example
corresponds to the case $(iii)$ of Theorem~\ref{thm:main_modest2}.
Hence the loop-centered state has a smaller energy
at a fixed large mass independently of lengths of the loops and the edges.

For the dumbbell graphs with more than three internal edges, $K > 3$,
the comparison is similar to Figs. \ref{db_Kgeq2_sv_fig3} and \ref{db_Kgeq2_sv_fig4} for $K = 3$.
The longest of the internal edges or loops is selected for the edge-localized state of smaller
energy at fixed large mass. The dumbbell graphs with $K \geq 3$ corresponds to the case $(iv)$ of Theorem~\ref{thm:main_modest2}.

\subsection{Example of the tadpole graphs}
\label{sec:tadpole}

\begin{figure}
  \centering
  \includegraphics{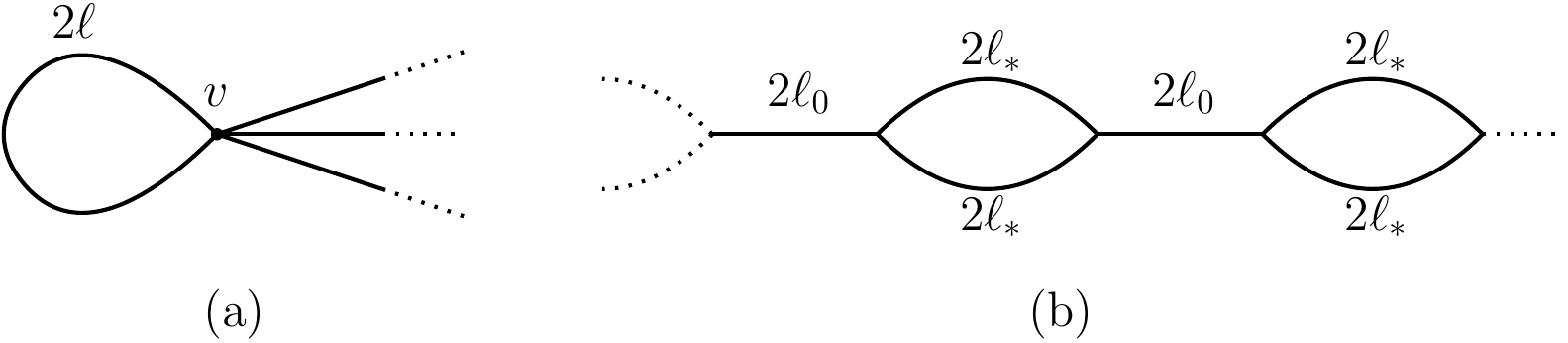}
  \caption{Example graphs considered in Sections~\ref{sec:tadpole} $(a)$ and
    \ref{sec:periodic} $(b)$.}
  \label{fig:example_graphs}
\end{figure}

As a particular unbounded graph, we consider a tadpole graph with a
single loop connected at one vertex point with $K \geq 1$ half-lines,
see Fig.~\ref{fig:example_graphs}(a) for an example. There is only one
edge of a finite length with reflection symmetry.

The case $K = 1$ corresponds to the canonical tadpole graph considered
in \cite{NPS} and in \cite{AdaSerTil_jfa16}.  Since
$\mathcal{Q}_{\rm loop} > 2 \mu$ in this case,
Corollary~\ref{cor:main_modest3} states that there exists a ground
state. The loop-centered state is a proper candidate for the ground state.
Indeed, it was proven in \cite{AdaSerTil_jfa16}, see Corollary 3.4 and Fig.~4 of \cite{AdaSerTil_jfa16}, by
using energy-decreasing symmetry rearrangements that the loop-centered states if the ground
state of the tadpole graph.

For the case $K = 2$, Corollary~\ref{cor:main_modest3} is
inconclusive because $\mathcal{Q}_{\rm loop} \approx 2 \mu$.  However,
this is an exceptional case, for which the tadpole graph with two
half-lines can be unfolded to an infinite line, for which the NLS
soliton is a valid stationary state with
$\mathcal{Q}_{\rm loop} = 2 \mu$. This loop-centered state is the
ground state for any value of $\mathcal{Q}$, see Example~2.4 and Fig.~3 of \cite{AdaSerTil_cvpde15} and Fig.~1 of
\cite{AdaSerTil_jfa16}.

For $K \geq 3$, we have $\mathcal{Q}_{\rm loop} < 2 \mu$ and the loop-centered state
is not a proper candidate for the ground state by Corollary~\ref{cor:main_modest3}.
Indeed, there is no ground state according to Theorem 2.5 of \cite{AdaSerTil_cvpde15}.

\subsection{Example of a periodic graph}
\label{sec:periodic}

Here we consider the periodic graphs \cite{D19,GilgPS,P18,PS17}, the
basic cell of which consists of one internal edge and one loop
repeated periodically, see Fig.~\ref{fig:example_graphs}(b).  We will use the convention that a connecting edge is of length $2\ell_0$ and the loop components of length $2 \ell_*$ (for a total loop length of $4 \ell_*$).   Existence
of stationary states pinned to the symmetry points of the internal
edge and the two halves of the loop was proven in the small-mass limit
in Theorem 1.1 in \cite{PS17}.  Characterization of stationary states
as critical points of a certain variational problem was developed in
Theorem 3.1 in \cite{P18}.

This example of the periodic graph is beyond validity of the
variational theory in \cite{AdaSerTil_jfa16,AST2019} or the comparison
theory in our Corollary~\ref{cor:main_modest3}. However, existence of
the ground state at every mass was proven for the periodic graph in
\cite{D19} without elaborating the symmetry of the ground state.  We
thus expect the estimates of Section~\ref{sec:constructing}, in
particular equation \eqref{eq:mass_answer_generic}, to hold without
any changes.

Under this assumption, we show that the symmetry of the edge-localized state of smallest energy
depends on the relative lengths between the internal edge and the
half-loop.  It follows from \eqref{eq:mass_answer_generic} that the
edge-localized state at the internal edge has the mass
$\mathcal{Q}_{\rm int}$ given by
\begin{equation}
\label{Q-periodic-1}
\mathcal{Q}_{\rm int} = 2 \mu - \frac{16}{3} \mu^2 \ell_0 e^{-2 \mu \ell_0}
+ \bigO{\mu e^{-2 \mu \ell_0}},
\end{equation}
where $\ell_0$ is the half-length of the internal edge, whereas the edge-localized state at the half-loop has
the mass $\mathcal{Q}_{\rm loop}$ given by
\begin{equation}
\label{Q-periodic-2}
\mathcal{Q}_{\rm loop} = 2 \mu - \frac{16}{3} \mu^2 \ell_* e^{-2 \mu \ell_*}
+ \bigO{\mu e^{-2 \mu \ell_*}},
\end{equation}
where $\ell_*$ is the quarter-length of the loop. Comparing (\ref{Q-periodic-1}) and (\ref{Q-periodic-2}) yields that
$\mathcal{Q}_{\rm int} < \mathcal{Q}_{\rm loop}$ if $\ell_0 < \ell_*$ and
$\mathcal{Q}_{\rm int} > \mathcal{Q}_{\rm loop}$ if $\ell_0 > \ell_*$.
By the Comparison Lemma (Lemma \ref{lem:comparison}), the loop-centered state
has smaller energy if $\ell_0 < \ell_*$ and the edge-centered state
has larger energy if $\ell_0 > \ell_*$, hence the state of smaller energy
localizes at the longer edge. The symmetric case $\ell_* = \ell_0$ is not conclusive
because $\mathcal{Q}_{\rm int} \approx \mathcal{Q}_{\rm loop}$ and computations of the
higher-order exponentially small terms are needed.

\begin{figure}[htp]
\begin{tabular}{cc}
\includegraphics[width=7cm]{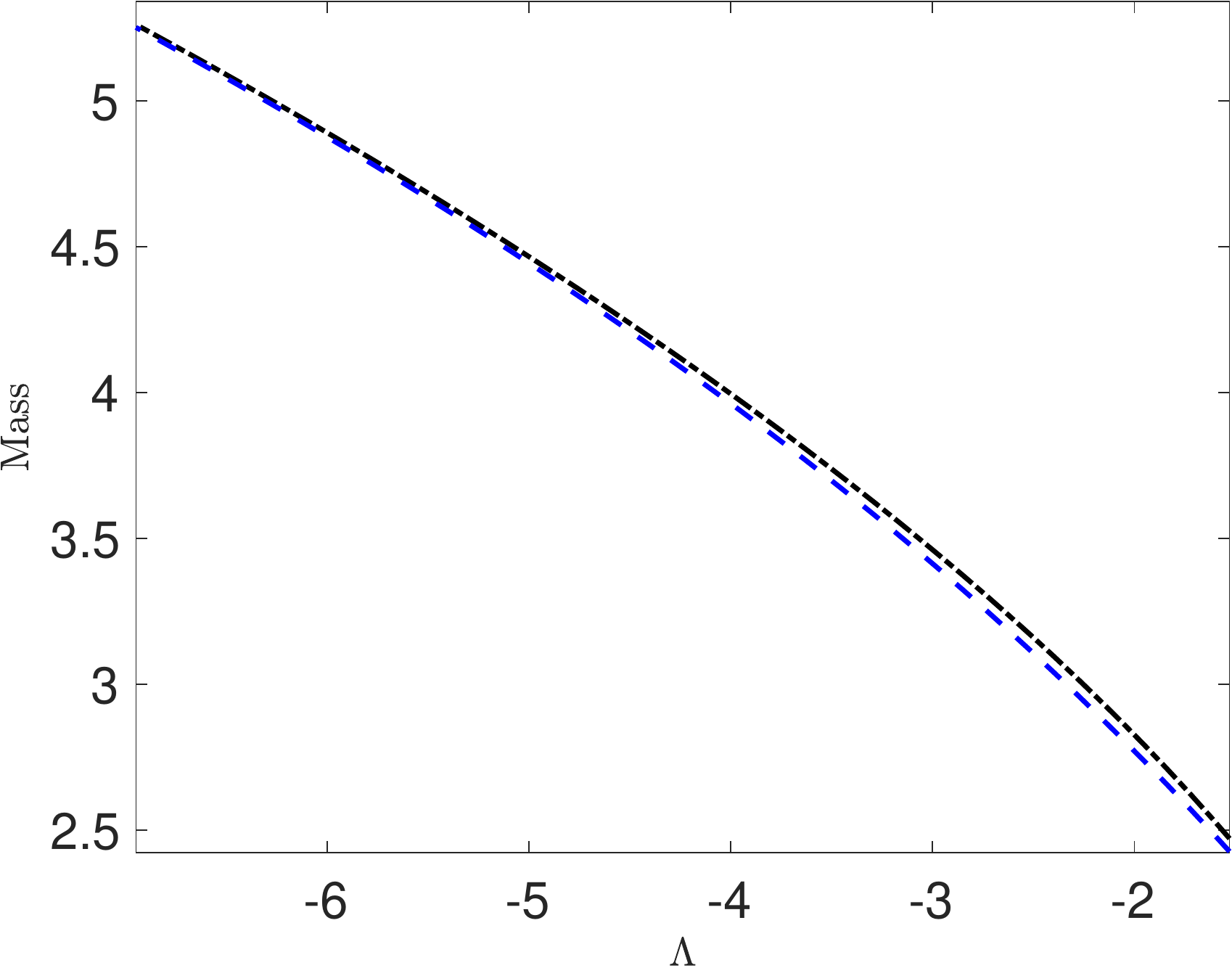} &
\includegraphics[width=7cm]{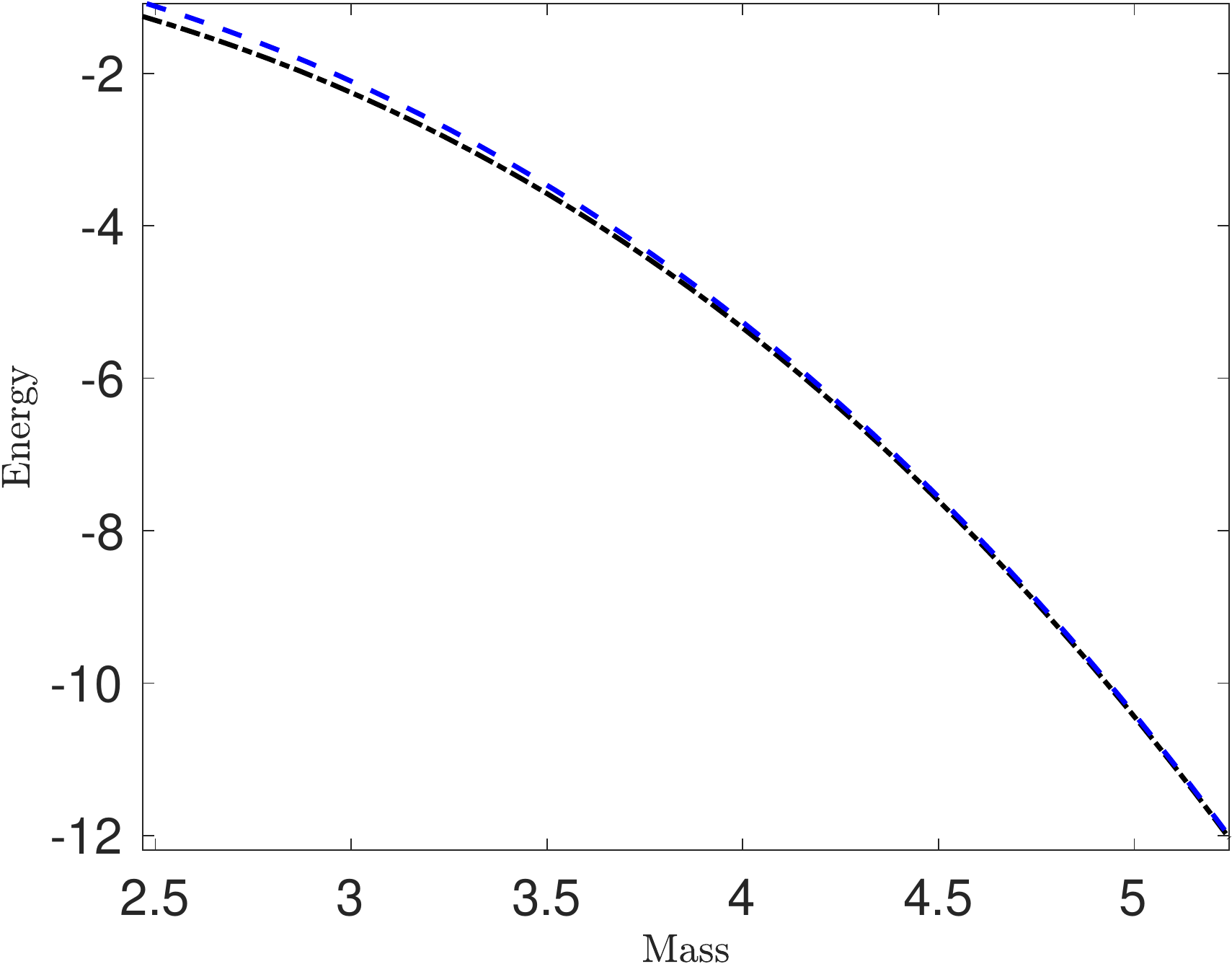} \\
\includegraphics[width=7cm]{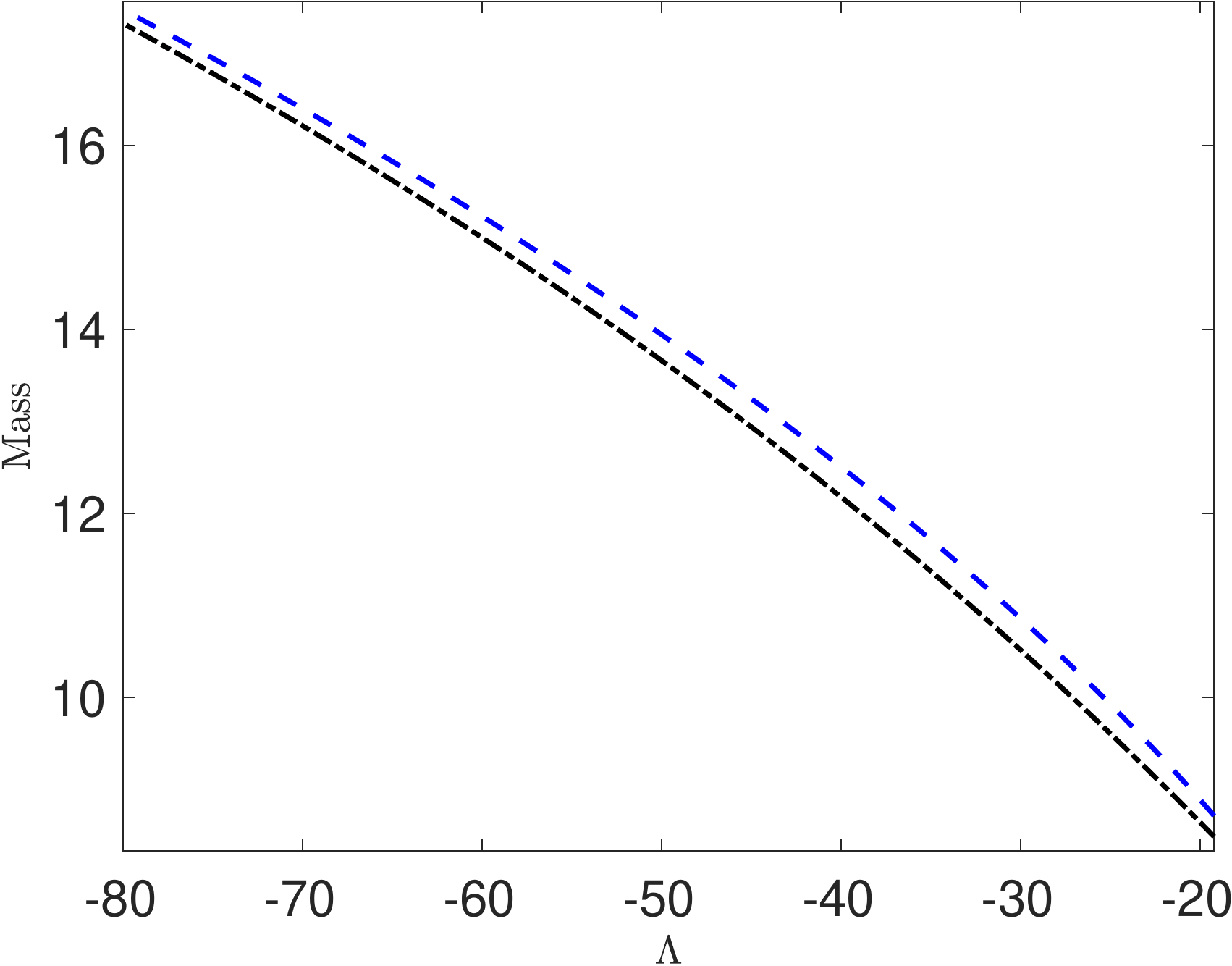} &
\includegraphics[width=7cm]{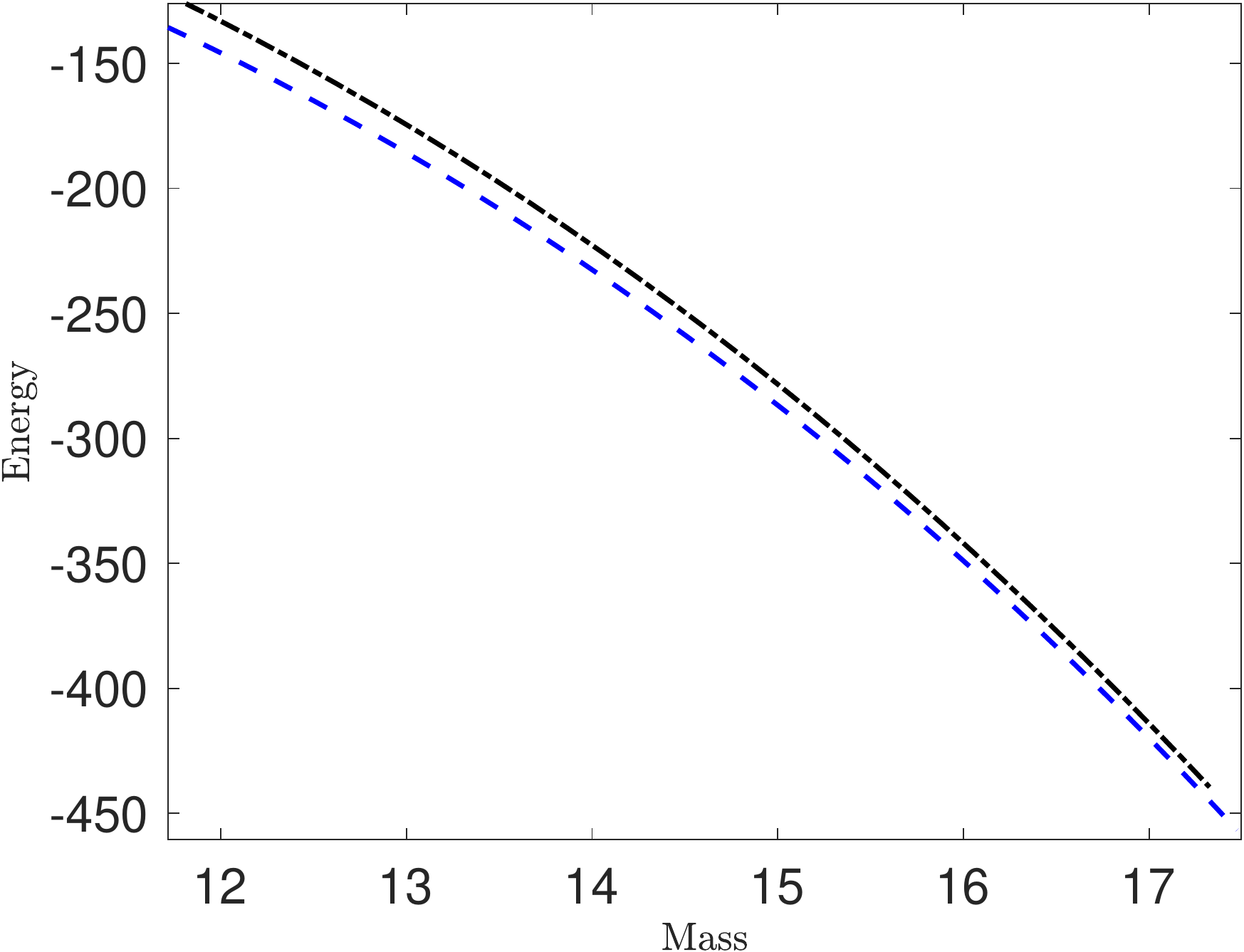} \\
\end{tabular}
\caption{Stationary states in the stationary NLS equation (\ref{statNLS}) for a periodic graph
showing the mass $\mathcal{Q}$ vs Lagrange multiplier $\Lambda$ (left panels) and the
energy $\mathcal{E}$ versus the mass $\mathcal{Q}$ (right panels).
Top panels show computations for $\ell_0 =4 \pi$ and $\ell_* = \pi/4$. Bottom panels show
computations for $\ell_0 =\pi/8$ and $\ell_* = \pi/4$.  The black dot-dash ($- \cdot$) (color online) line shows the edge-centered state,
the blue dashed ({\color{blue}$- -$}) (color online) line shows the loop-centered state.}
\label{periodic_fig}
\end{figure}

Figure \ref{periodic_fig} shows results of numerical computations of stationary states on
the periodic graph with a loop of length $4 \ell_* = \pi$ and a horizontal edge of length $2\ell_0$.
If $\ell_0 = 4 \pi > \ell_*$, the state of smaller energy is centered at the horizontal edge (top panels).
If $\ell_0 = \pi/8 < \ell_*$, the state of smaller energy is centered at the half-loop, as predicted above.

In order to compute the stationary solutions on a periodic graph, we returned
  to the finite difference scheme discussed in \cite{MarPel_amrx16}
  and implemented an approximation to the graph by truncating the
  periodic system after a small number of cells in the middle of the
  internal edges and connecting the two endpoints with periodic
  boundary conditions.  For large $\mu$ we observed that the
  predicted asymptotics are verified numerically even in the case of
  one cell.

\appendix
\section{Proof of Theorem~\ref{thm:asymptotic_dtn_lengths}}
\label{section:DtN_appendix}

Consider a graph $\Graph$ with a finite number of vertices and a
finite number of edges, which either connect a pair of vertices and
have finite length or have only one vertex and are identified with the
half-line.  We impose Neumann-Kirchhoff (NK) conditions at every
vertex.  Declare a subset $B$ of the graph's vertices to be the
\emph{boundary}.  We are interested in the asymptotics of the DtN map
on the boundary $B$ for the operator $-\Delta + \mu^2$ as
$\mu \to \infty$. The parameter $\mu$ is treated as the spectral
parameter $\lambda := -\mu^2$ for the spectrum of $-\Delta$.

Let the boundary vertices be denoted $b_1, \ldots, b_{|B|}$, and
let $\vecp = (p_1, \ldots, p_{|B|})^T \in \bbR^{|B|}$ be a vector of
``Dirichlet values'' on the vertices.  Assume $\mu>0$ and
consider a function $f \in H^2(\Graph)$ satisfying
\begin{equation}
  \begin{cases}
    \left(-\Delta + \mu^2 \right) f = 0, \qquad
    & \mbox{on every }e\in\Graph,\\
    f \mbox{ satisfies NK conditions}
    & \mbox{for every} \;\; v\in V\setminus B,\\
    f(v_j) = p_j,\qquad
    & \mbox{for every} \;\; v_j \in B.
  \end{cases}
    \label{eq:DtN_eig_equationA}
\end{equation}
Let $\Neu(f)_j = \sum_{e \sim v_j} \partial f(v_j)$ be the Neumann data of the function $f$ at the vertex $v_j \in B$,
where $\partial$ denotes the outward derivative from the vertex $v_j$.
Note that $f$ is not required to satisfy the current conservation conditions at $v_j \in B$.

The map $\DTN_\Graph(\mu) : \vecp \mapsto \Neu(f) \in \bbR^{|B|}$ is
called the DtN map at the spectral level $\lambda = -\mu^2$.  We will
derive its asymptotics as $\mu\to\infty$ by investigating the
scattering solutions in the same regime.  We refer to
\cite[Sec~3.5]{BerKuc_graphs} for more information on DtN map on a
compact quantum graph.  We remark that in the presence of infinite
edges all definitions work when $\lambda = -\mu^2 < 0$ is below the
absolutely continuous spectrum of $-\Delta$ but cease to work (in
general) when $\lambda \geq 0$.

Closely related to the DtN map is the scattering matrix $\Sigma(\mu)$,
see \cite[Sec~5.4]{BerKuc_graphs}, defined on a compact graph.
Attaching an infinite edge (a \emph{lead}) to each boundary vertex (a
single edge per vertex), we look for $\tilde{f}$ solving
$(-\Delta + \mu^2)\tilde{f}=0$ on the augmented graph and satisfying
NK vertex conditions at every vertex.
The space of such solutions is $b$-dimensional; writing the solution on the lead $e$ in the form
\begin{equation}
  \label{eq:scattering_sol_on_leads}
  \tilde{f}(x_e) = c_e^{in} e^{\mu x_e} + c_e^{out} e^{-\mu x_e},
\end{equation}
the space of solutions may be parametrized by the vectors
${\bf c}^{in} = (c_1^{in},\ldots,c_b^{in})^T$.  The scattering matrix
$\Sigma(\mu)$ describes the scattering on incoming waves into
the outgoing ones,
\begin{equation}
  \label{eq:scat_def}
  {\bf c}^{out} = \Sigma(\mu) {\bf c}^{in}.
\end{equation}
For the graph with NK vertex conditions, there is a fairly explicit
formula for the scattering matrix $\Sigma(\mu)$, derived in
\cite{KotSmi_prl00,BanBerSmi_ahp12,BerKuc_graphs},
\begin{equation}
  \label{eq:formula_scattering}
  \Sigma(\mu) = R +
  T_o e^{-\mu L} \left(I - \tilde{U} e^{-\mu L}\right)^{-1} T_i,
\end{equation}
where, informally speaking, $R$ governs reflection of waves from a
lead back into a lead, $T_i$ transmits incoming waves into the
interior of the graph, $T_o$ transmits interior waves into the
outgoing lead waves and $\tilde{U}$ describes scattering of waves in
the interior.  For a graph with scale-invariant vertex conditions
(such as NK), all these matrices have constant entries.  The
dependence on $\mu$ enters through the diagonal matrix
$e^{-\mu L}$ where $L$ is the diagonal matrix of internal edge
lengths.

We can also obtain a formula for the solution in the interior by
writing the solution on the edge $e$ as
\begin{equation}
  \label{eq:scattering_sol_inside}
  f_e(x) = a_e e^{-\mu x} + a_{\overline{e}} e^{-\mu(\ell_e-x)}.
\end{equation}
The vector ${\bf a}$ of the coefficients $a_e$ and $a_{\overline{e}}$
satisfies (see \cite[Thm.~2.1]{BanBerSmi_ahp12})
\begin{equation}
  \label{eq:formula_internal}
  {\bf a} = \left(I - \tilde{U} e^{-\mu L} \right)^{-1} T_i {\bf c}^{in}.
\end{equation}

\begin{theorem}
  \label{thm:asymptotics_scat}
  The scattering matrix at $\lambda = -\mu^2$ of a compact graph $\Gamma$ with the boundary
  set $B$ has the asymptotic
  expansion
  \begin{equation}
    \label{eq:scat_matr_expansion}
    \Sigma(\mu) = \diag\left(\frac{2}{d_b+1}-1\right)_{b\in B} +
    O\left(e^{-\mu \Lmin}\right),
    \ \mu \to \infty,
  \end{equation}
  where $d_b$ is the degree of the boundary vertex $b$ not counting
  the lead, $\Lmin$ is the length of the shortest edge and the
  remainder term is a matrix with the norm bounded by
  $Ce^{-\mu \Lmin}$.

  In the same asymptotic regime, the vector ${\bf a}$ of interior
  coefficients has the expansion
  \begin{equation}
    \label{eq:interior_asymptotics}
    {\bf a} = \left(T_i + O(e^{-\mu \Lmin})\right) {\bf c}^{in},
  \end{equation}
  where the correction is a matrix with the specified norm bound.
\end{theorem}

\begin{proof}
  In our setting --- all vertex conditions are Neumann-Kirchhoff,
  there are $|B|$ leads with at most one lead per vertex --- the matrix
  $R$ in equation~\eqref{eq:formula_scattering} is the $|B| \times |B|$
  diagonal matrix with entries $2/(d_b + 1) - 1$; this matrix provides
  the leading order term in \eqref{eq:scat_matr_expansion}.  To
  estimate the remainder, we note that
  \begin{equation}
    \label{eq:operator_norm_exp}
    \left\| e^{-\mu L} \right\| \leq e^{-\mu \Lmin},
  \end{equation}
  in the operator sense from $\mathbb{R}^{2E}$ to $\mathbb{R}^{2E}$,
  where $E$ is the number of edges of $\Gamma$.
  Since the matrix $\tilde{U}$ is sub-unitary
  (it is a submatrix of a unitary matrix), we have
  \begin{equation}
    \label{eq:resolvent_estimates}
    \left\|\tilde{U}\right\| \leq 1,
    \qquad
    \left\| \left(I - \tilde{U} e^{-\mu L}\right)^{-1} \right\|
    \leq \frac{1}{1 - e^{-\mu \Lmin}},
  \end{equation}
  and, overall,
  \begin{equation*}
    \left\| T_o e^{-\mu L}
    \left(I - \tilde{U} e^{-\mu L}\right)^{-1} T_i \right\|
    \leq \frac{\tilde{C} e^{-\mu \Lmin}}{1 - e^{-\mu \Lmin}}
    \leq C e^{-\mu \Lmin}.
  \end{equation*}
  The asymptotic expansion for ${\bf a}$ is obtained from
  equation~\eqref{eq:formula_internal}, expansion
  \begin{equation}
    \label{eq:resolvent_bootstrap}
    \left(I - \tilde{U} e^{-\mu L} \right)^{-1}
    = I + \tilde{U} e^{-\mu L}
    \left(I - \tilde{U} e^{-\mu L} \right)^{-1},
  \end{equation}
  and estimates \eqref{eq:operator_norm_exp} and \eqref{eq:resolvent_estimates}.
\end{proof}

In order to prove Theorem \ref{thm:asymptotic_dtn_lengths}, we establish asymptotics
of the DtN map defined by problem \eqref{eq:DtN_eig_equationA} and then use the scaling transformation.
The following theorem presents the asymptotic estimates for the boundary-value problem (\ref{eq:DtN_eig_equationA}).

\begin{theorem}
  \label{thm:asymptotic_dtn}
  There exists a unique solution $f \in H^2(\Graph)$ to the
  boundary-value problem~\eqref{eq:DtN_eig_equationA}
  which satisfies asymptotically, as $\mu \to \infty$,
  \begin{equation}
    \label{eq:asymptotics_L2_norm}
    \left\| f \right\|_{L^2(\Graph)}^2
    \leq C \left(\frac1{2\mu}
      + \mathcal{O}(\Lmin e^{-\mu \Lmin}) \right)\|\vecp\|^2
  \end{equation}
  and
  \begin{equation}
    \label{eq:asymptotics_dtn}
    \DTN_\Graph(\mu) = \mu \diag(d_b)_{b\in B}
    + \mathcal{O}\left(\mu e^{-\mu \Lmin}\right),
  \end{equation}
  where $d_j$ is the degree of the $j$-th boundary vertex, $\Lmin$
  is the length of the shortest edge in $\Graph$ and the remainder term is a
  matrix with the norm bounded by $C\mu e^{-\mu \Lmin}$.
\end{theorem}

\begin{proof}
  We intend to use the asymptotics we derived for the scattering
  matrix and a formula linking it to the DtN map $\DTN_\Graph(\mu)$
  (see, for example, \cite[Sec.~5.4]{BerKuc_graphs}).  However, we
  allow our graph to have infinite edges, a situation which is not
  covered in the results for the scattering matrix.  To overcome this
  limitation, we covert infinite edges into leads.  To avoid a
  situation when two leads join the same boundary vertex, we create,
  on each infinite edge, a dummy vertex $w$ of degree 2, see Fig. \ref{fig:dtn_map_inf}.
  The resulting graph we still denote by $\Graph$; by $W$ we denote
  the set of the newly created vertices and by $\Graph^c$ the
  compact graph containing all finite edges of the graph $\Graph$.  We also
  attach leads to the boundary vertices $b\in B$ and define the
  scattering matrix $\Sigma^c(\mu)$ of the compact graph $\Graph^c$
  with respect to \emph{all infinite edges}.

  \begin{figure}[t]
    \centering
    \includegraphics{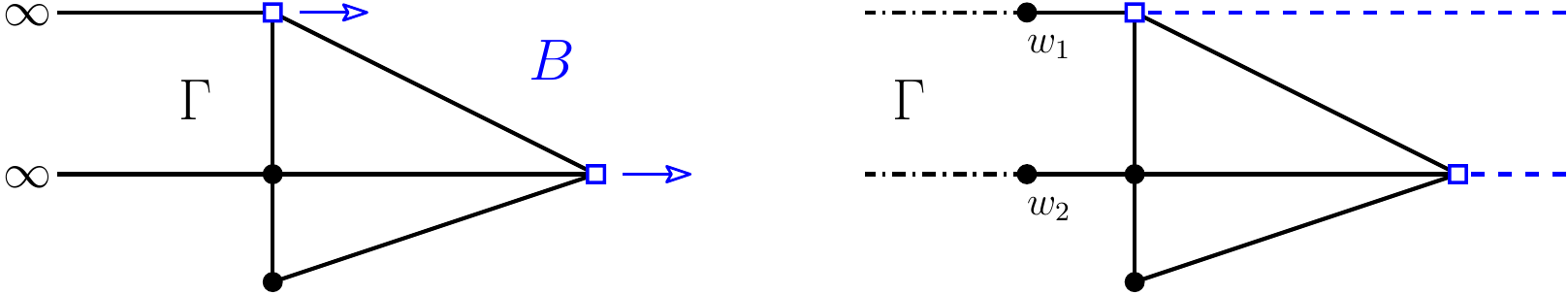}
    \caption{Left: a graph $\Graph$ with boundary verices $B$ marked as
      empty squares.  Right: after introducing dummy vertices $w_1$ and
      $w_2$ we obtain a compact graph $\Graph^c$ (solid edges only).
      The graph $\Graph$ combines solid and dash-dotted edges.  Dashed
      edges correspond to ``true'' leads corresponding to the boundary
      vertices $B$. }
    \label{fig:dtn_map_inf}
  \end{figure}

  The matrix $\Sigma^c(\mu)$ maps a vector ${\bf c}^{in}$ of incoming
  wave coefficients to the vector ${\bf c}^{out}$ of outgoing ones.
  The coefficients will be labelled by the attachment vertices of the
  corresponding lead, namely by $B \sqcup W$.  Denote by $P$ the
  operator from $\bbC^{|B|+|W|}$ to $\bbC^{|B|}$ acting as the orthogonal
  projection followed by restriction.

  Since we are looking for an $H^2(\Graph)$ solution of the boundary-value problem
  \eqref{eq:DtN_eig_equationA}, on the infinite edges of $\Graph$ the
  solution must have the \emph{purely radiating} form
  \begin{equation}
    \label{eq:purely_radiating}
    f(x_w) = c_w^{out} e^{-\mu x_w}.
  \end{equation}
  We now need to solve a ``mixed'' problem: on the vertices $w\in W$
  we are prescribing the purely radiating condition $c_w^{in}=0$ while
  on the vertices $b\in B$ of $\Graph^c$ we are prescribing the
  solution values $p_j$.  The latter condition may be expressed as
  \begin{equation}
    \label{eq:wave_to_value}
    c_b^{in} + c_b^{out} = \tilde{f}(v_b) = p_b,
  \end{equation}
  by substituting $x_e=0$ in \eqref{eq:scattering_sol_on_leads}.
  Splitting the vectors $\vecc$ into two parts corresponding to $B$
  and $W$, we get
  \begin{equation*}
    \begin{pmatrix}
      \vecc^{out}_B \\ \vecc^{out}_W
    \end{pmatrix}
    = \Sigma^c(\mu)
    \begin{pmatrix}
      \vecc^{in}_B \\ \mathbf{0}
    \end{pmatrix}.
  \end{equation*}
  In particular,
  \begin{equation}
    \label{eq:subunitary_scattering_def}
    \vecc^{out}_B = \Sigma(\mu) \vecc^{in}_B,
    \qquad\mbox{where}\quad
    \Sigma(\mu) := P \Sigma^c(\mu) P^*.
  \end{equation}
  We note that $\Sigma(\mu)$ is a $|B|\times|B|$ block of the matrix
  $\Sigma^c(\mu)$.  It is an analogue of the scattering matrix for
  the graph $\Graph$.

  Combining \eqref{eq:subunitary_scattering_def} and
  \eqref{eq:wave_to_value} we have
  \begin{equation}
    \label{eq:a_asympt}
    {\bf c}^{in} = \left(I + \Sigma(\mu)\right)^{-1} \vecp.
    \qquad\mbox{and}\qquad
    {\bf a} = \left(T_i + O(e^{-\mu \Lmin})\right)
    \left(I + \Sigma(\mu)\right)^{-1} \vecp,
  \end{equation}
  which, together with expansion \eqref{eq:scat_matr_expansion} for
  $\Sigma^c(\mu)$, implies
  \begin{equation}
    \label{eq:a_bound_p}
    \left\| {\bf a} \right\| \leq C \left\| \vecp \right\|.
  \end{equation}
  We note that coefficients ${\bf a}$ give the expansion of the
  solution on all edges of the graph $\Graph$, including the infinite
  edges (the same value of the coefficients applies on the finite and
  infinite portion, because the connecting vertex $w$ has degree 2).
  We can now estimate the norm of the solution $f$.  From expansion
  \eqref{eq:scattering_sol_inside} on the finite edges, we have
  \begin{equation*}
    \left\| f_e \right\|_{L^2(\Graph)}^2 = \frac{1-e^{-2\mu \ell_e}}{2\mu}
    \left( |a_e|^2 + |a_{\overline{e}}|^2 \right)
    + 2 \Real\left(a_e \overline{a_{\overline{e}}}\right) e^{-\mu \ell_e} \ell_e
    \leq \frac{1 + 2\ell_e\mu e^{-\mu \ell_e}}{2\mu}
    \left( |a_e|^2 + |a_{\overline{e}}|^2 \right).
  \end{equation*}
  On an infinite edge the solution has the form
  \eqref{eq:purely_radiating} with $c_w^{out}$ equal to $a_e$ on the
  finite edge ending in $w$.  Therefore, on the infinite edge together
  with the corresponding finite part,
  \begin{equation*}
    \left\| f_w \right\|_{L^2(\Graph)}^2 = \frac{1}{2\mu} |a_e|^2.
  \end{equation*}

  On the whole of $\Graph$, the norm of the function $f$
  satisfies the bound
  \begin{equation}
    \label{eq:L2_norm_bound}
    \left\| f \right\|_{L^2(\Graph)}^2 \leq \left(\frac1{2\mu} +
      O(\Lmin e^{-\mu \Lmin}) \right) \|{\bf a}\|^2.
  \end{equation}
  Here we used that
  \begin{equation*}
    \ell_e\mu e^{-\mu \ell_e} \leq \mu\Lmin e^{-\mu \Lmin}
    \qquad \mbox{if }\mu\Lmin > 1.
  \end{equation*}
  Combining \eqref{eq:L2_norm_bound} with \eqref{eq:a_bound_p} yields
  the desired estimate on $f$, equation~\eqref{eq:asymptotics_L2_norm}.

  We now express the DtN map from the matrix $\Sigma(\mu)$.  From the
  expansion~\eqref{eq:scattering_sol_on_leads} we get
  \begin{equation*}
    \Neu(f) = \mu \left(\vecc_B^{in} - \vecc_B^{out}\right)
     = \mu (I - \Sigma(\mu)) \vecc_B^{in}.
  \end{equation*}
  Combining this with \eqref{eq:a_asympt} we obtain
  \begin{equation}
    \label{eq:DtN_from_scat}
    \DTN_\Graph(\mu) = \mu \frac{I - \Sigma(\mu)}{I + \Sigma(\mu)},
  \end{equation}
  were the fraction notation can be used because two matrices
  commute.  This is the same expression as in
  \cite[Eq.~(5.4.8)]{BerKuc_graphs} only now $\Graph$ is allowed to
  have infinite edges and $\Sigma$ is defined via (\ref{eq:subunitary_scattering_def}).

  We recall that $\Sigma(\mu)$ is the $B$-block of the matrix
  $\Sigma^c(\mu)$ to which Theorem~\ref{thm:asymptotics_scat} applies.
  We denote $\Sigma(\mu) = R + S(\mu)$, where $R$ is the diagonal
  matrix and $S(\mu)$ is the remainder term in the asymptotic
  expansion of $\Sigma(\mu)$, equation~\eqref{eq:scat_matr_expansion}.
  Using the formula \eqref{eq:DtN_from_scat} we write
  \begin{equation}
    \label{eq:DtN_expand}
    \DTN_\Graph(\mu) = \mu \frac{I - R}{I + R} + \mu Q,
    \qquad Q :=
    (I-R-S)(I+R+S)^{-1} - (I+R)^{-1}(I-R).
  \end{equation}
  Factoring out the inverse matrices, we estimate the norm of $Q$ as
  \begin{align*}
    \|Q\|
    & \leq \left\|(I+R)^{-1}\right\|
    \|(I+R)(I-R-S) - (I-R)(I+R+S)\|
    \left\| (I+R+S)^{-1} \right\| \\
    & \leq 2 \|S\| \left\|(I+R)^{-1}\right\|
    \left\| (I+R+S)^{-1} \right\|.
  \end{align*}
  We have
  \begin{equation*}
    \left\| (I+R+S)^{-1} \right\| \to \left\|(I+R)^{-1}\right\| \leq
    \frac{\max d_b+1}{2},
  \end{equation*}
  and, combining with the estimate on $\|S\|$ from
  \eqref{eq:scat_matr_expansion}, we get $\|Q\| \leq C\mu
  e^{-\mu \Lmin}$. The first term in the expansion \eqref{eq:DtN_expand} can be
  evaluated explicitly,
  \begin{equation*}
    \frac{I - R}{I + R} = \frac{\diag\left(2 - \frac{2}{d_b+1}\right)}
    {\diag\left(\frac2{d_b+1}\right)}
    = \diag\left(d_b\right),
  \end{equation*}
  yielding \eqref{eq:asymptotics_dtn}.
\end{proof}

We will now rescale the problem by $\mu$ to obtain the results of
Theorem~\ref{thm:asymptotic_dtn_lengths}.

\begin{proof}[Proof of Theorem~\ref{thm:asymptotic_dtn_lengths}]
  The solution to problem~\eqref{eq:DtN_linear_rescaled} is obtained
  from the solution $f$ of the boundary-value problem \eqref{eq:DtN_eig_equationA} by the
  rescaling
  \begin{equation}
    \label{eq:rescale_u_f}
    u(\mu x) = f(x).
  \end{equation}
  This rescaling has the following effect on the Neumann data, the DtN map
  and the $L^2$ norm:
  \begin{equation}
    \label{eq:after_rescale_u_f}
    \Neu(u) = \frac1\mu \Neu(f),\qquad
    \DTN = \frac1\mu\DTN_\Graph,\qquad
    \|u\|_{L^2(\Graph_\mu)}^2 = \mu \|f\|_{L^2(\Graph)}^2,
  \end{equation}
  where $\DTN$ denotes the DtN map of $\Graph_\mu$.
  Asymptotics in Theorem~\ref{thm:asymptotic_dtn_lengths} now
  immediately follow from the corresponding asymptotics in
  Theorem~\ref{thm:asymptotic_dtn}.  Finally, we observe that $u$
  satisfies the differential equation $\Delta u = u$ and we can use
  the estimate (see \cite[Ch.~4, Eq.~(4.40)]{Burenkov_Sobolev_spaces})
  \begin{equation}
    \label{eq:Sobolev_norm_equiv}
    \|u\|_{L^2(\Graph_\mu)}^2 \leq
    \|u\|_{H^2(\Graph_\mu)}^2 \leq C
    \left( \|u\|_{L^2(\Graph_\mu)}^2 +
      \|u''\|_{L^2(\Graph_\mu)}^2\right)
    = 2C \|u\|_{L^2(\Graph_\mu)}^2,
  \end{equation}
  where $C$ is uniform in edge lengths as long as they are bounded
  away from 0 (which is clearly the case as $\mu\to\infty$).
\end{proof}

\section{Maximum principle for quantum graphs}
\label{sec:maximum_principle}

There are several results in the quantum graphs literature
establishing different versions of ``maximum principle'', see
\cite[2.4.3, Cor 2]{PokPry_umn04} and \cite[Thm 2]{BakFab_incol06}.
For our purposes, the most convenient form is that given in \cite[Lem
2.1]{HarMal_prep18} which we cite verbatim below.

\begin{lemma}[Lemma 2.1 in \cite{HarMal_prep18}]
  \label{lem:max_principle}
  Let $V(x)\geq0$ on an open subset $S$ of $\Graph$. Suppose that
  $w\in C^1$, and let $Hw := -w'' + V(x)w$ (in the weak
  sense) on edges, with ``super--Kirchhoff'' conditions at the
  vertices, namely,
  \begin{equation*}
    \sum_{e\sim v} w_e'(v) \geq 0, \qquad v\in S,
  \end{equation*}
  i.e., the sum of the outgoing derivatives of $w$ at every vertex is
  nonnegative. If $Hw \leq 0$ on the edges contained in S, then $\max(w,
  0)$ does not have a strict local maximum on $S$.
\end{lemma}

\begin{remark}
  In our setting, we have Kirchhoff conditions $\sum_{e\sim v} w_e'(v) = 0$ and
  the homogeneous equation $Hw = 0$.  Thus, Lemma \ref{lem:max_principle} is directly
  applicable except we would like to exclude non-strict maxima as
  well.  This is almost automatic if we impose the strict positivity
  $V > 0$ on $S$.  Indeed, let a maximum be acheived at
  point $x$ and $w(x)>0$.  If the maximum is non-strict, there is a
  sequence of points converging to $x$ where $w$ takes the same value
  as $w(x)$, therefore (possibly one-sided) derivative of $w$ at $x$
  is zero and $Hw > 0$ (in the weak sense) close to $x$.
\end{remark}

\section{Contraction mapping principle}
\label{sec:contraction_mapping}

In this section we collect classical results of nonlinear functional
analysis (see, for example \cite{Zeidler_vol1of5}) in the setting most
immediately applicable to our problem.

\begin{theorem}[Contraction Mapping Principle]
  \label{thm:contraction_mapping}
  Let $T$ be a map on a Banach space $Y$ with the norm $\| \cdot \|$
  mapping a ball $B_R = \{y\in Y: \|y\| < R\}$ to itself.  If $T$ is a contraction
  with a parameter $\lambda < 1$, i.e.
  \begin{equation}
    \label{eq:contraction_def}
    \|T(y_1) - T(y_2)\| \leq \lambda \|y_1-y_2\|,
    \qquad \forall y_1, y_2 \in B_R,
  \end{equation}
  then there exists a fixed point $y^* = T(y^*)$, which is unique in
  $B_R$.  The fixed point satisfies the estimate
  \begin{equation}
    \label{eq:fp_estimate}
    \|y^*\| \leq \frac{1}{1-\lambda} \| T(0) \|.
  \end{equation}
\end{theorem}

In the case when the contraction mapping $T$ smoothly depends on a
parameter $x$, the fixed point will also depend on the parameter
smoothly.  We remind some standard facts and definitions leading to
this result.

\begin{definition}
  The map $f : U \subseteq X \to Z$, with $X$ and $Z$ Banach spaces is
  \emph{Fr\'echet-differentiable} at $x\in U$ if there exists a bounded
  linear operator which we denote $D_x f: X\to Z$ such that
  \begin{equation}
    \label{eq:frechet_def}
    f(x+h) - f(x) = f'(x) h + o(\|h\|), \quad h\to 0,
  \end{equation}
  for all $h$ in some neighborhood of 0.

  If $U$ is open and the derivative $D_x f(x)$ exists for all $x\in U$
  and depends continuously (in the operator norm) on $x$, the map $f$
  is called $C^1$.
\end{definition}

The partial Fr\'echet derivatives for a mapping
$F: X \times Y \to Z$ are defined analogously.  A map $F$ is $C^1$ in
an open $U \subseteq X\times Y$ if and only if the partial derivatives
$D_x F$ and $D_y F$ are continuous in $U$.

\begin{theorem}[Smooth Implicit Function Theorem]
  \label{thm:implicit_function}
  Suppose that the mapping $F: U \subseteq X \times Y \to Z$, where
  $U$ is open and $X$, $Y$ and $Z$ are Banach spaces over $\bbR$ or
  $\bbC$, is such that
  \begin{enumerate}
  \item there is a point $(x_0,y_0) \in U$ satisfying $F(x_0,y_0)=0$,
  \item $F$ is $C^1$ in $U$,
  \item the partial derivative $D_y F(x_0,y_0) : Y \to Z$ is bijective.
  \end{enumerate}
  Then there is a positive number $r_0$ and a $C^1$ map
  $y(\cdot): B_{r_0}(x_0) \subset X \to Y$ such that $F(x,y(x)) = 0$ and
  $y(x_0)=y_0$.  Furthermore, there is a number $r>0$ such that for
  any $x \in B_{r_0}(x_0)$, $y(x)$ is the only solution of $F(x,y)=0$
  satisfying $\|y-y_0\| < r$.
\end{theorem}

Combining the above two theorems gives a smooth dependence of the
fixed point on a parameter.

\begin{corollary}[Contraction Mapping with a Parameter]
  \label{cor:contraction_smooth}
  Let $T: X\times Y \to Y$ be a $C^1$ mapping on an open set
  $U \subseteq X\times Y$.  Suppose for some $x_0\in X$ and
  $V \subseteq Y$ such that $\{x_0\}\times V \subset U$, the mapping
  $T(x_0,y) : Y \to Y$ is a contraction which maps $V$ into itself.

  Then there is a positive number $r_0$ and a $C^1$ map
  $y(\cdot): B_{r_0}(x_0) \subset X \to Y$ such that
  \begin{equation*}
    T\Big(x,y(x)\Big) = y(x).
  \end{equation*}
\end{corollary}

\begin{proof}
  We first apply Theorem~\ref{thm:contraction_mapping} to $T(x_0,y)$
  obtaining a fixed point $y_0 \in V$.  Then we apply
  Theorem~\ref{thm:implicit_function} with $F(x,y) = y - T(x,y)$.  The
  partial derivative $F_y(x_0,y_0)$ is bijective because it is
  identity minus an operator which strictly smaller than 1:
  $\|T_y(x_0,y_0)\|\leq \lambda < 1$ since $T(x_0,y)$ is a
  contraction.
\end{proof}

\section{Useful estimates on elliptic functions}
\label{sec:elliptic}

We introduce the elliptic integrals of the first and second kind, respectively:
\begin{equation}
\label{elliptic-integrals}
F(\varphi;k) := \int_0^{\varphi} \frac{d \theta}{\sqrt{1-k^2 \sin^2\theta}}, \quad
E(\varphi;k) := \int_0^{\varphi} \sqrt{1-k^2 \sin^2\theta} d \theta.
\end{equation}
From this definition, the complete elliptic integrals of the first and second kind
are given by
\begin{equation}
K(k) := F\left(\frac{\pi}{2};k\right) \quad \mbox{\rm and} \quad
E(k) := E\left(\frac{\pi}{2};k\right)
\end{equation}
respectively. In addition, Jacobi's elliptic functions are given by
\begin{equation}
{\rm sn}(u;k) = \sin \varphi, \quad
{\rm cn}(u;k) = \cos \varphi, \quad
{\rm dn}(u;k) = \sqrt{1-k^2 \sin^2 \varphi},
\end{equation}
where $u$ is related to the elliptic integrals by
\begin{equation}
\label{u:ellintdef}
F(\varphi;k) = u, \quad
E(\varphi;k)= \int_0^u {\rm dn}^2(s;k) ds.
\end{equation}
Many properties of elliptic integrals and Jacobi's elliptic functions are
collected together in \cite{GraRyzh_table}.

The following technical result was proven in Appendix of \cite{MarPel_amrx16}.
\begin{proposition}
\label{prop-elliptic}
For every $\xi \in \mathbb{R}$, it is true that
\begin{align}
\label{sn-der}
& {\rm sn}(\xi;1) =  \tanh(\xi), \quad \partial_k {\rm sn}(\xi;1) = -\frac{1}{2} \left[ \sinh(\xi) \cosh(\xi) - \xi \right] {\rm sech}^2(\xi), \\
\label{cn-der}
& {\rm cn}(\xi;1)  =  {\rm sech}(\xi), \quad \partial_k {\rm cn}(\xi;1) = \frac{1}{2} \left[ \sinh(\xi) \cosh(\xi) - \xi \right] \tanh(\xi) {\rm sech}(\xi), \\
\label{dn-der}
& {\rm dn}(\xi;1)  =  {\rm sech}(\xi), \quad \partial_k {\rm dn}(\xi;1) = -\frac{1}{2} \left[ \sinh(\xi) \cosh(\xi) + \xi \right] \tanh(\xi) {\rm sech}(\xi).
\end{align}
Moveover, for sufficiently large $\xi_0$, there is a positive constant $C$ such that
\begin{equation}
\label{bound-second-derivative}
|\partial_k {\rm sn}(\xi;k) - \partial_k {\rm sn}(\xi;1)| +
| \partial_k {\rm cn}(\xi;k) - \partial_k {\rm cn}(\xi;1)| +
|\partial_k {\rm dn}(\xi;k) - \partial_k {\rm dn}(\xi;1)| \leq C \xi_0 e^{-\xi_0},
\end{equation}
holds for every $\xi \in (\xi_0,K(k))$ and every $k \in (k_*,1)$ with $k_* = 1 - \mathcal{O}(e^{-2\xi_0})$.
\end{proposition}

We can now address the ``reverse Sobolev estimate'' on the real line
(see also Lemma~\ref{lem:reverse_Sobolev_graph}).

\begin{proposition}
  \label{prop:H2_from_Linf}
  There exist positive $L_0$, $c_0$, and $C$ such that every
  real positive solution $\Psi
  \in H^2(0,L)$ of the stationary
  NLS equation $-\Psi'' + \Psi = 2 |\Psi|^2 \Psi$ satisfying
  \begin{equation}
  \label{bound-on-Psi}
  |\Psi(z)| < c, \quad z \in [0,L],
  \end{equation}
  for every $c \in (0,c_0)$ and every $L \in (L_0,\infty)$, also satisfies the bound
  \begin{equation}
    \label{eq:H2_from_Linf_edge}
    \| \Psi \|_{H^2(0,L)} \leq C \|\Psi\|_{L^\infty(0,L)}.
  \end{equation}
\end{proposition}

\begin{proof}
It is sufficient to obtain the estimates on $\| \Psi \|^2_{L^2(0,L)}$
since the stationary NLS equation implies that
$$
\| \Psi'' \|_{L^2(0,L)} \leq (1 + 2 \| \Psi \|_{L^{\infty}(0,L)}^2) \| \Psi \|_{L^2(0,L)}
\leq (1 + 2 c^2) \| \Psi \|_{L^2(0,L)}.
$$

It follows from the phase portrait for $-\Psi'' + \Psi - 2 \Psi^3 = 0$,
see Fig.~\ref{fig:phase_port_nls}, that the solutions $\Psi \in H^2(0,L)$
satisfying \eqref{bound-on-Psi} for small $c > 0$ and large $L > 0$
have at most one local minimum and no internal maxima on $[0,L]$.
Moreover, either $\Psi$ is sign-definite on $[0,L]$ or $\Psi'$ is sign-definite on $[0,L]$.
Without loss of generality, we give the proof for sign-definite (positive) solutions
expressed by the ${\rm dn}$-elliptic functions (\ref{eq:dnoidal}). The proof for sign-indefinite 
solutions expressed by the ${\rm cn}$-elliptic functions (\ref{eq:cnoidal}) is similar. 

We partition $[0,L]$ into $[0,L_1]$ and $[L_1,L]$, where $L_1$ is the point
of minimum of $\Psi$, such that $\Psi'(z) < 0$ for $z \in [0,L_1)$
and $\Psi'(z) > 0$ for $z \in (L_1,L]$. Without loss of generality, assume
$L_1 \in \left[\frac{L}{2},L\right]$ so that
$$
\| \Psi \|^2_{L^2(0,L)} \leq 2 \| \Psi \|^2_{L^2(0,L_1)}.
$$
We use the exact solution \eqref{eq:dnoidal} and write for some $L_2 > 0$:
$$
\Psi(z) = \frac{1}{\sqrt{2-k^2}} {\rm dn}\left( \frac{z + L_2}{\sqrt{2-k^2}}; k \right), \quad z \in [0,L_1],
$$
where $-L_2 < 0$ is the location of the maximum of $\Psi(z)$ to the left of the interval $[0,L_1]$. 
Since $\Psi'(z) < 0$ for $z \in [0,L_1)$, we have 
$$
\| \Psi \|_{L^{\infty}(0,L_1)} = \frac{1}{\sqrt{2-k^2}} {\rm dn}\left( \frac{L_2}{\sqrt{2-k^2}}; k \right)
$$
and since $\Psi'(L_1) = 0$, we have 
$$
L_1 + L_2 = \sqrt{2 - k^2} K(k).
$$
This implies that if $c \in (0,c_0)$ for some small $c_0 > 0$, then $L_2 > 0$ is sufficiently large
and if $L \in (L_0,\infty)$ for some large $L_0 > 0$, then $L_1 > 0$ is also sufficiently large, whereas
$k \in (k_-,1)$ with $k_-$ satisfying $|k_- - 1| \leq A_1 e^{-2(L_1+L_2)}$, where 
a positive constant $A_1$ is independent on $L_1$ and $L_2$.
We obtain by direct substitution for $k \in (k_-,1)$ that
\begin{eqnarray*}
\| \Psi \|_{L^2(0,L_1)}^2 & = & \frac{1}{2-k^2} \int_0^{L_1} {\rm dn}^2\left(\frac{z+L_2}{\sqrt{2-k^2}};k\right) dz \\
& = & \frac{1}{\sqrt{2-k^2}} \left[ E\left( \phi \left[  \frac{L_1+L_2}{\sqrt{2-k^2}} \right]    ;k \right) - E\left( \phi \left[  \frac{L_2}{\sqrt{2-k^2}} \right]   ;k \right) \right],
\end{eqnarray*}
where $\varphi[u]$ is the inverse of the map $F(\varphi;k) = u$ in \eqref{u:ellintdef}.  Hence, we have
$$
\| \Psi \|_{L^2(0,L_1)}^2 \leq C(k;L_1,L_2) \| \Psi \|_{L^{\infty}(0,L_1)}^2,
$$
where
$$
C(k;L_1,L_2) := \sqrt{2-k^2} \frac{E\left( \phi \left[  \frac{L_1+L_2}{\sqrt{2-k^2}} \right]  ;k \right) - E\left( \phi \left[  \frac{L_2}{\sqrt{2-k^2}} \right]  ;k \right)}{
{\rm dn}^2 \left( \frac{L_2}{\sqrt{2-k^2}};k \right)}.
$$
We need to show that $C(k;L_1,L_2)$ is bounded uniformly for large $L_1$ and $L_2$ and for $k \in (k_-,1)$ 
with $|k_- - 1| \leq A_1 e^{-2(L_1+L_2)}$.
It follows from (\ref{dn-der}) and (\ref{bound-second-derivative}) of Proposition \ref{prop-elliptic} that
$$
E(  \phi [ \xi ];k) = \tanh(\xi) + (1-k) \xi + R_E(\xi;k),
$$
where $R_E$ is the remainder term satisfying $\| R_E(\cdot;k) \|_{L^{\infty}(0,K(k))} = \mathcal{O}(k-1)$
as $k \to 1$. In addition, we use the lower bounds:
$$
{\rm dn}^2(\xi;k) \geq 1 - k^2
$$
to estimate the remainder term $R_E$ and
$$
{\rm dn}^2(L_2;k) \geq {\rm sech}^2(L_2)
$$
to estimate the other two terms,
where the latter bound holds for sufficiently large $L_2$. Thus, it follows for every $k \in (k_-,1)$ with
$|k_- - 1| \leq A_1 e^{-2(L_1+L_2)}$ that
\begin{eqnarray*}
C(k;L_1,L_2) & \leq & \frac{\tanh(L_1+L_2) - \tanh(L_2)}{{\rm sech}^2(L_2)} + \frac{L_1 (1-k)}{{\rm sech}^2(L_2)} + A_2 \\
& = & \frac{\sinh(L_1) \cosh(L_2)}{\cosh(L_1 + L_2)} + L_1 (1-k) \cosh^2(L_2) + A_2 \\
& \leq & \frac{(1-e^{-2L_1}) (1 + e^{-2L_2})}{2(1 + e^{-2(L_1+L_2)})} + L_1 (1-k) e^{2L_2} + A_2 \\
& \leq & 1 + A_1 L_1 e^{-2L_1} + A_2,
\end{eqnarray*}
where a positive constant $A_2$ is independent on $L_1$ and $L_2$.
Hence, $C(k;L_1,L_2)$ is bounded
uniformly for sufficiently large $L_1$ and $L_2$, which proves the estimate for $\| \Psi \|_{L^2(0,L_1)}$ 
and hence the estimate \eqref{eq:H2_from_Linf_edge}.
\end{proof}


\bibliographystyle{hamsplain}
\bibliography{edge_loc}

\end{document}